\documentclass[11pt]{amsart}

\usepackage{amsmath,amssymb,amsfonts,amsthm}
\usepackage[margin=2cm]{geometry}
\usepackage{caption,subcaption}
\usepackage[colorlinks]{hyperref}
\usepackage{tikz}

\definecolor{coGB}{HTML}{1F77B4}
\definecolor{coWG}{HTML}{2CA02C}
\definecolor{coAR}{HTML}{D62728}

\hypersetup{
    colorlinks = True,
    linkcolor  = coAR,
    citecolor  = coWG,
    urlcolor   = coGB,
}

\newcommand{\xpos}{\boldsymbol{x}}
\newcommand{\ypos}{\boldsymbol{y}}

\newtheorem{theorem}{Theorem}[section]
\newtheorem{corollary}{Corollary}[theorem]

\title[Signal subspace imaging for SAR]{High-resolution, quantitative
  signal subspace imaging for synthetic aperture radar}

\author{Arnold D. Kim \and Chrysoula Tsogka}

\address{Department of Applied Mathematics, University of California,
  Merced, 5200 North Lake Road, Merced, CA 95343, USA}

\email{adkim@ucmerced.edu, ctsogka@ucmerced.edu}

\begin{document}

\maketitle

% abstract

\begin{abstract}
  We consider synthetic aperture radar imaging of a region containing
  point-like targets. Measurements are the set of frequency responses
  to scattering by the targets taken over a collection of individual
  spatial locations along the flight path making up the synthetic
  aperture. Because signal subspace imaging methods do not work on
  these measurements directly, we rearrange the frequency response at
  each spatial location using the Prony method and obtain a matrix
  that is suitable for these methods. We arrange the set of these
  Prony matrices as one block-diagonal matrix and introduce a signal
  subspace imaging method for it. We show that this signal subspace
  method yields high-resolution and quantitative images provided that
  the signal-to-noise ratio is sufficiently high. We give a resolution
  analysis for this imaging method and validate this theory using
  numerical simulations. Additionally, we show that this imaging
  method is stable to random perturbations to the travel times and
  validate this theory with numerical simulations using the random
  travel time model for random media.
\end{abstract}

% keywords

\section{Introduction}

Synthetic aperture imaging is used in many applications such as
ultrasonic non-destructive testing, mine detection, surveillance, and
radar imaging. The main idea behind synthetic aperture imaging is that
a single transmitter/receiver is used to probe an unknown region by
emitting known pulses into the medium and recording the time-dependent
responses as it moves along a given path. Fourier transforming these
time-dependent measurements yields their corresponding frequency
responses.  In this work we focus our attention on the synthetic
aperture radar imaging problem. However, the methodology used here can
be directly applied to other related problems.

Several imaging methods have been proposed in the literature for
imaging with SAR data.  The traditional SAR image is formed by
evaluating the data at each measurement location at the travel time
that it takes for the waves to propagate from the platform location to
a point in the imaging region on the ground and back. The resolution
of this image increases with the synthetic aperture and the system
bandwidth \cite{cheney2009fundamentals}.  When the phases of the waves
are recorded with high accuracy, SAR imaging produces high-resolution
images of the reflectivity on the ground.  It is well known however
that SAR imaging is quite sensitive to noise in the phase.  Such noise
may result from uncertainty in the platform motion and/or scattering
by randomly inhomogeneous media.  For SAR imaging with noise in the
phase, we refer to \cite{GS-08} for the application of coherent
interferometry (CINT) to SAR imaging and to the more recent work in
\cite{BG-20} on a high-resolution interferometric method for imaging
through scattering media.

Several approaches have been proposed in the literature to further
improve the resolution of SAR images. We refer to
\cite{baraniuk2007compressive,potter2010sparsity} for
sparsity-constrained $\ell_1$-minimization methods and to
\cite{borcea2016synthetic} for imaging effectively direction and
frequency dependent reflectivities using the multiple measurement
vector (MMV) framework \cite{malioutov05}. As in other applications,
using sparsity-constrained optimization methods significantly
increases the resolution of the SAR image. However, the computational
cost of optimization is significantly higher than that for sampling
methods such as SAR or CINT, which simply consist of evaluating an
imaging functional at each grid point on a mesh of the imaging region.

MUSIC (multiple signal classication) is another sampling method that
has been widely used in several imaging applications
~\cite{devaney2005time, fannjiang2011music,
  griesmaier2017multifrequency, moscoso2019robust}. To explain the
main idea of MUSIC, let us consider the single-frequency array imaging
problem.  For this problem the data is a matrix, called the array
response matrix. The $(i,j)$-th element of the array response matrix
corresponds to the data received at the $i$-th array element when the
$j$-th element is an emitter. The singular value decomposition (SVD)
of the array response matrix is used to determine the signal and noise
subspaces of the data.  Next, a model for the illumination vector
$\mathbf{a}(\ypos)$ is introduced, with $\ypos$ denoting a point in
the imaging region. The illumination vector is the vector of
measurements received along the receiving array due to a source at
$\ypos$.  If a target is located at $\ypos$, then $\mathbf{a}(\ypos)$
is in the signal subspace of the array response matrix. Thus, the
projection of $\mathbf{a}(\ypos)$ onto the noise subspace is zero or
very small. In MUSIC, one forms an image by evaluating one over this
norm of the projection of $\mathbf{a}(\ypos)$ onto the noise
subspace. The peaks appearing in the MUSIC image give the locations of
the targets with high resolution. Although MUSIC effectively and
efficiently produces high-resolution images, it does not apply
directly to SAR imaging data.

In this paper we introduce a modification and generalization of MUSIC
for SAR imaging. This imaging method modifies SAR data by using the
Prony method~\cite{prony1795essai} to rearrange frequency-dependent
data at one measurement location as a matrix.  Then we form a
block-diagonal matrix with the set of Prony matrices from all spatial
locations on the flight path. An image of the reflectivity on the
ground is then formed using a signal subspace method applied to this
block-diagonal matrix. This signal subspace method is a generalization
of MUSIC that projects the illuminating vector for each point $\ypos$
in the imaging region on both the noise and signal subspaces
\cite{gonzalez2021quantitative}. The noise subspace provides high
spatial resolution and the signal subspace provides quantitative
information about the targets. The result of combining these two
subspaces is a high-resolution quantitative imaging method.  The
relative balance between the noise and signal subspaces depends on the
noise level in the data which is controlled through a user-defined
regularization parameter, $\epsilon$.

There are two main results in this paper.  The first main result is
the resolution analysis for this modified and generalized MUSIC method
that shows an enhancement in resolution compared to classical SAR
imaging by a factor $\sqrt{\epsilon}$. Namely we obtain a cross-range
resolution of $O(\sqrt{\epsilon} (c/B) (L/a) )$ and a range resolution
of $O(\sqrt{\epsilon} (c/B) (L/R) )$. Here $c$ denotes the speed of
the waves, $B$ denotes the bandwidth, $a$ denotes the synthetic array
aperture, $L$ denotes the distance from the center of the flight path
to the center of the imaging region, and $R$ denotes the range offset
(see Fig. \ref{fig:schematic}).  The second main result is the
stability analysis of the method to random perturbations of the travel
times. This analysis shows that the method provides stable
reconstructions when $\epsilon$ is chosen to satisfy
$\sigma^{2} \ll \epsilon < 1 $ with $\sigma^{2}$ denoting the maximum
variance of the random perturbations of the travel times. Our
numerical simulations are in agreement with these theoretical
findings.  Moreover, they show that the proposed method provides
statistically stable results with signal-to-noise ratios comparable to
CINT, but with much better resolution.
  
The remainder of this paper is as follows. In Section \ref{sec:sar} we
give a brief description of synthetic aperture radar imaging and
define the measurements. In Section \ref{sec:prony} we describe the
Prony method that we use to rearrange the frequency data and show why
it is appropropriate for signal subspace imaging. We define the two
imaging functionals that we use for quantitative signal subspace
imaging in Section \ref{sec:method}. In Section \ref{sec:resolution}
we give a resolution analysis for the imaging method. We consider this
imaging method when the travel times have random perturbations in
Section \ref{sec:random} and give results for the expected value and
statistical stability of the image formed using this method. We show
numerical results that support our theory in Section
\ref{sec:results}. Section \ref{sec:conclusion} contains our
conclusions.

\section{Synthetic aperture radar imaging}
\label{sec:sar}

In synthetic aperture radar (SAR) imaging, a single
transmitter/receiver is used to collect the scattered electromagnetic
field over a synthetic aperture that is created by a moving platform
\cite{cheney2001mathematical, cheney2009fundamentals,
  moreira2013tutorial}. The moving platform is used to create a suite
of experiments in which pulses are emitted and resulting echoes are
recorded by the transmitter/receiver at several locations along the
flight path. Let $f(t)$ denote the broadband pulse emitted and let
$d(s,t)$ denote the data recorded. Here, the measurements depend on
the slow time $s$ that parameterizes the flight path of the platform,
$\boldsymbol{r}(s)$, and the fast time $t$ in which the round-trip
travel time between the platform and the imaging scene on the ground
is measured. In SAR imaging, one seeks to recover the reflectivity of
an imaging scene from these measurements.

Although SAR uses a single transmit/receive element, high-resolution
images of the probed scene can be obtained because the data are
coherently processed over a large synthetic aperture created by the
moving platform.  As illustrated in Fig.~\ref{fig:schematic}, the
platform is moving along a trajectory probing the imaging scene by
sending a pulse $f(t)$ and collecting the corresponding echoes. We
call range the direction that is obtained by projecting on the imaging
plane the vector that connects the center of the imaging region to the
central platform location. Cross-range is the direction that is
orthogonal to the range. Denoting the size of the synthetic aperture
by $a$ and the available bandwidth by $B$, the typical resolution of
the imaging system is $O((c/B) (L/R))$ in range and $O(\lambda L/a)$
in cross-range. Here $c$ is the speed of light and $\lambda$ the
wavelength corresponding to the central frequency while $L$ denotes
the distance between the platform and the imaging region and $R$ the
offset in range.
 
\begin{figure}[t]
  \centering
  \includegraphics[width=0.48\textwidth]{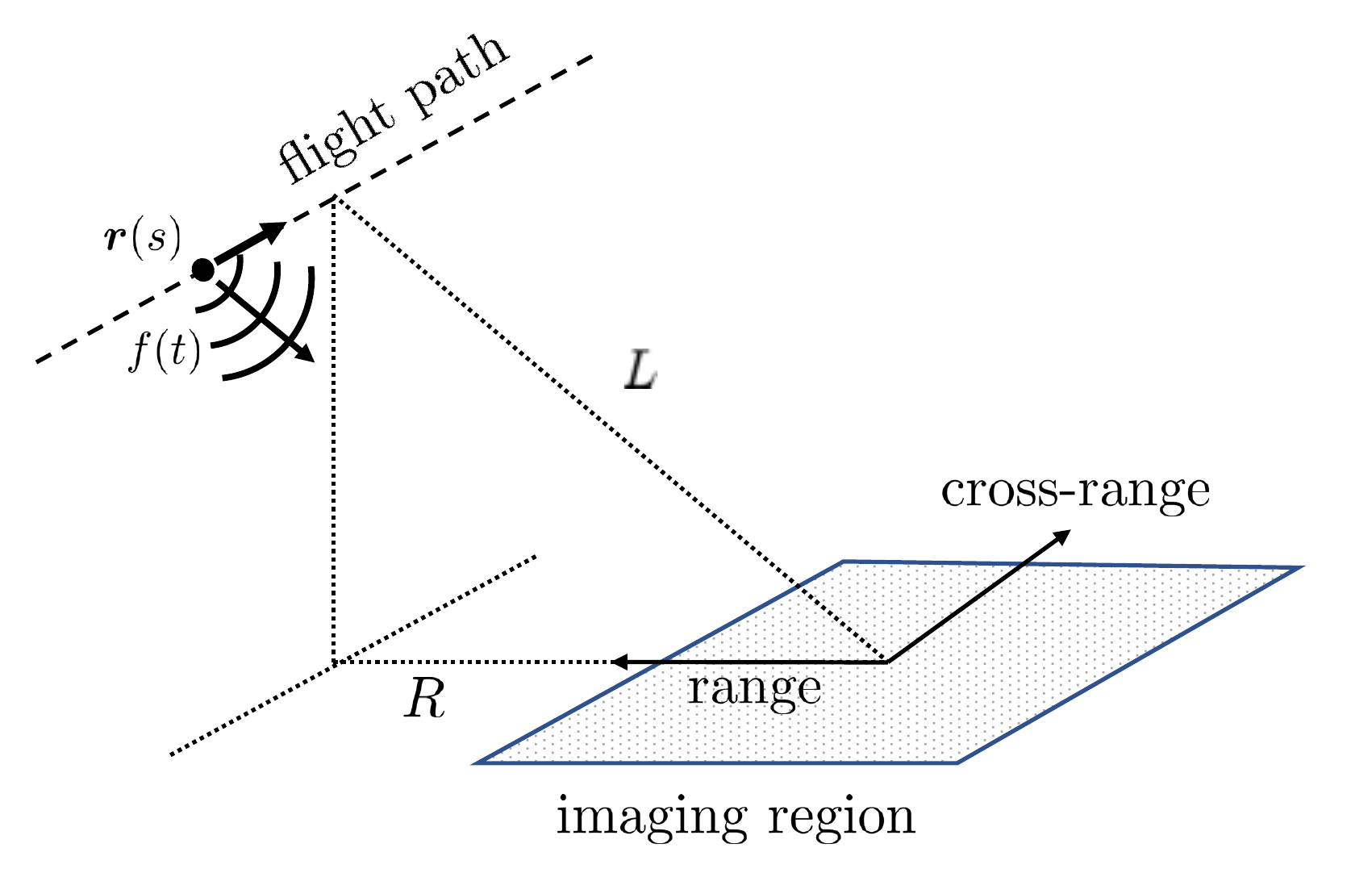}
  \caption{Setup for synthetic aperture radar imaging.}
  \label{fig:schematic}
\end{figure}

We use the start-stop approximation, which is typically done in SAR
imaging.  This approximation assumes that the change in displacement
between the targets and the platform is negligibly small compared to
the travel time it takes for the pulse emitted to propagate to the
imaging scene and return as echoes.  This approximation is valid in
radar since the speed of light is orders of magnitude larger than the
speed of the targets and the platform.

Using this start-stop approximation, we the consider the measurements
only at $N$ discrete values of $s$, corresponding to $d(s_{n},t)$ for
$n = 1, \cdots, N$.  Next, we suppose that $d(s_{n},t)$ is digitally
sampled at $2M-1$ values of $t$. Consequently, these data have a
discrete Fourier transform denoted by $d_{n}(\omega)$ evaluated at
$2M-1$ frequencies denoted by $\omega_{m}$ for $m = 1, \cdots,
2M-1$. This choice of $2M-1$ samples is to make the notation in
Section \ref{sec:prony} simpler.  With these assumptions, we find that
our measurement data is given by the $2M-1 \times N$ matrix $D$ whose
columns are
\begin{equation}
  \mathbf{d}_{n} = 
  \begin{bmatrix}
    d_{n}(\omega_{1}) \\ d_{n}(\omega_{2}) \\ \vdots \\
    d_{n}(\omega_{2M-1})
  \end{bmatrix}, \quad n = 1, \cdots, N.
  \label{eq:Dmatrix}
\end{equation}

\section{Rearranging frequency data}
\label{sec:prony}

The data matrix $D$ is not suitable for direct application of signal
subspace methods. Therefore, we introduce a rearrangement of the data
based on the Prony method~\cite{prony1795essai} which, for the $n$-th
column of $D$, yields the following $M \times M$ matrix,
\begin{equation}
  D_{n} = \begin{bmatrix} 
    d_{n}(\omega_{1}) & d_{n}(\omega_{2}) & \cdots &
    d_{n}(\omega_{M}) \\
    d_{n}(\omega_{2}) & d_{n}(\omega_{3}) & \cdots &
    d_{n}(\omega_{M+1}) \\
    \vdots & \vdots & \ddots & \vdots \\
    d_{n}(\omega_{M}) & d_{n}(\omega_{M+1}) & \cdots &
    d_{n}(\omega_{2M-1})
  \end{bmatrix}.
  \label{eq:Pronyfied}
\end{equation}
In this rearrangement, the first column is the truncation of
$\mathbf{d}_{n}$ to its first $M$ entries. Subsequent columns are
sequential upward shifts of $\mathbf{d}_{n}$ truncated to its first
$M$ entries.

To see why this rearrangement is suitable for signal subspace imaging,
consider the Born approximation for a single point target. Let
$\rho_{0}$ denote the reflectivity of the point target and
$\ypos_{0}$ denote its position. According to the Born
approximation, the scattered field at frequency $\omega_{m}$ is the
spherical wave,
\begin{equation}
  \psi^{s}(\xpos,\omega_{m}) = \rho_{0} \frac{e^{\mathrm{i} \omega_{m}
      | \xpos - \ypos_{0} |/c}}{4 \pi | \xpos - \ypos_{0} |}
  \psi^{inc}(\ypos_{0},\omega_{m}),
\end{equation}
with $c$ denoting the wave speed and
$\psi^{inc}(\ypos_{0},\omega_{m})$ denoting the field incident on the
point target.  Suppose that the signal emitted at position $\xpos_{n}$
on the flight path at frequency $\omega_{m}$ is a spherical wave with
unit amplitude. For that case, the measurement $d_{n}(\omega_{m})$
corresponding to the scattered field evaluated at $\xpos_{n}$ is given
by
\begin{equation}
  d_{n}(\omega_{m}) = \rho_{0} \frac{e^{\mathrm{i} 2 \omega_{m} |
    \xpos_{n} - \ypos_{0} |/c}}{( 4 \pi | \xpos_{n} -
    \ypos_{0} | )^{2} }.
\end{equation}
It follows that
\begin{equation}
  \mathbf{d}_{n} = \frac{\rho_{0}}{( 4 \pi | \xpos_{n} -
    \ypos_{0} |)^{2}}
  \begin{bmatrix}
    e^{\mathrm{i} 2 \omega_{1} | \xpos_{n} - \ypos_{0} |/c} \\
    e^{\mathrm{i} 2 \omega_{2} | \xpos_{n} - \ypos_{0} |/c} \\
    \vdots\\
    e^{\mathrm{i} 2 \omega_{2M-1} | \xpos_{n} - \ypos_{0} |/c}
  \end{bmatrix},
\end{equation}
from which we find that
\begin{equation}
  D_{n} = \frac{\rho_{0}}{( 4 \pi | \xpos_{n} - \ypos_{0}
    |)^{2}}
  \begin{bmatrix}
    e^{\mathrm{i} 2 \omega_{1} | \xpos_{n} - \ypos_{0} |/c} &
    e^{\mathrm{i} 2 \omega_{2} | \xpos_{n} - \ypos_{0} |/c} &
    \cdots &
    e^{\mathrm{i} 2 \omega_{M} | \xpos_{n} - \ypos_{0} |/c} \\
    e^{\mathrm{i} 2 \omega_{2} | \xpos_{n} - \ypos_{0} |/c} &
    e^{\mathrm{i} 2 \omega_{3} | \xpos_{n} - \ypos_{0} |/c} &
    \cdots &
    e^{\mathrm{i} 2 \omega_{M+1} | \xpos_{n} - \ypos_{0} |/c} \\
    \vdots & \vdots & \ddots & \vdots \\
    e^{\mathrm{i} 2 \omega_{M} | \xpos_{n} - \ypos_{0} |/c} &
    e^{\mathrm{i} 2 \omega_{M+1} | \xpos_{n} - \ypos_{0} |/c} &
    \cdots &
    e^{\mathrm{i} 2 \omega_{2M-1} | \xpos_{n} - \ypos_{0} |/c}
  \end{bmatrix}.
  \label{eq:Bn}
\end{equation}
Next, suppose that the frequencies are sampled according to
$\omega_{m} = \omega_{1} + (m-1) \Delta \omega$ for
$m = 1, \cdots, 2M-1$ with $\Delta \omega$ a fixed constant. For that
case, we can rewrite \eqref{eq:Bn} as
$D_{n} = \sigma_{0}^{(n)} \mathbf{u}_{0}^{(n)}
(\mathbf{v}_{0}^{(n)})^{\ast}$ with
$\sigma_{0}^{(n)} = M |\rho_{0}|/(4 \pi |\xpos_{n} - \ypos_{0}|)^{2}$,
\begin{equation}
  \mathbf{u}_{0}^{(n)} = \frac{e^{\mathrm{i}\theta_{0}/2}}{\sqrt{M}}
  \begin{bmatrix}
    e^{\mathrm{i} 2 \omega_{1} | \xpos_{n} - \ypos_{0} |/c} \\
    e^{\mathrm{i} 2 \omega_{2} | \xpos_{n} - \ypos_{0} |/c} \\
    \vdots \\
    e^{\mathrm{i} 2 \omega_{M} | \xpos_{n} - \ypos_{0} |/c}
  \end{bmatrix},
  \quad
  \text{and}
  \quad
  \mathbf{v}_{0}^{(n)} =
    \frac{e^{-\mathrm{i} \theta_{0}/2}}{\sqrt{M}}
    \begin{bmatrix}
      1 \\
      e^{-\mathrm{i} 2 \Delta \omega | \xpos_{n} - \ypos_{0} |/c} \\
      \vdots \\
      e^{-\mathrm{i} 2 (M-1) \Delta \omega | \xpos_{n} -
        \ypos_{0} |/c}
    \end{bmatrix}.
    \label{eq:uv}
\end{equation}
Here, we have written the reflectivity as
$\rho_{0} = |\rho_{0}| e^{\mathrm{i} \theta_{0}}$ and included
$e^{\mathrm{i} \theta_{0}}$ in $\mathbf{u}_{0}^{(n)}$ and
$\mathbf{v}_{0}^{(n)}$ arbitrarily in \eqref{eq:uv}.

Suppose there are $P$ non-interacting point targets in the region with
reflectivities $\rho_{p}$ at positions $\ypos_{p}$ for
$p = 1, \cdots, P$. It follows that
\begin{equation}
  D_{n} = \sum_{p = 1}^{P} \sigma_{p}^{(n)} \mathbf{u}_{p}^{(n)}
  ( \mathbf{v}_{p}^{(n)} )^{\ast}.
  \label{eq:outer-product}
\end{equation}
Here,
$\sigma_{p}^{(n)} = M |\rho_{p}|/(4 \pi |\xpos_{n} - \ypos_{p}|)^{2}$,
and $\mathbf{u}_{p}^{(n)}$ and $\mathbf{v}_{p}^{(n)}$ are defined just
like $\mathbf{u}_{0}^{(n)}$ and $\mathbf{v}_{0}^{(n)}$ in
\eqref{eq:uv}, but evaluated on $|\xpos_{n} - \ypos_{p}|$ instead.
This expression for $D_{n}$ is a sum of $P$ outer products, each of
which corresponds to an individual point target. This outer product
representation for $D_{n}$ indicates that signal subspace methods may
be effectively used on these matrices for imaging.

\section{Quantitative signal subspace method}
\label{sec:method}

To combine the matrices formed using the Prony method, we consider the
$MN \times MN$ block diagonal matrix,
\begin{equation}
  D_{\text{Prony}} = \begin{bmatrix}
    D_{1} & & & \\
    & D_{2} & & \\
    & & \ddots & \\
    & & & D_{N}
  \end{bmatrix}.
  \label{eq:D-Prony}
\end{equation}
Using the outer-product structure identified in
\eqref{eq:outer-product}, we can extend a recently developed
quantitative signal subspace imaging
method~\cite{gonzalez2021quantitative} to this block-diagonal Prony
matrix as follows.

Suppose we compute the singular value decomposition for each block:
$D_{n} = U_{n} \Sigma_{n} V_{n}^{\ast}$ for $n = 1, \cdots,
N$. According to \eqref{eq:outer-product}, for $P$ point targets, the
rank of each block will be $P$. Assuming that the signal-to-noise
ratio (SNR) is sufficiently high, the first $P$ singular values
residing in the diagonal entries of $\Sigma_{n}$ will be significantly
larger than the others. Those first $P$ singular values correspond to
the signal subspace. The remaining singular values correspond to the
noise subspace.  Since we can separate the first $P$ singular values,
we are able to compute the pseudo-inverse,
\begin{equation}
  \Sigma_{n}^{+} = \frac{1}{\sigma_{1}} \text{diag}\left( 1,
    \frac{\sigma_{2}}{\sigma_{1}}, \cdots,
    \frac{\sigma_{P}}{\sigma_{1}}, \frac{1}{\epsilon}, \cdots,
    \frac{1}{\epsilon} \right),
  \label{eq:sigma-plus}
\end{equation}
with $\epsilon > 0$ denoting a user-defined parameter.

For search point $\ypos$ in the imaging region, we introduce the
illumination block-vector
\begin{equation}
  \mathbf{a}(\ypos) = \begin{bmatrix}
    \mathbf{a}_{1}(\ypos) \\ \vdots \\ \mathbf{a}_{N}(\ypos)
  \end{bmatrix},
  \qquad
  \mathbf{a}_{n}(\ypos) = \frac{1}{4 \pi |\xpos_{n} -
    \ypos|}
  \begin{bmatrix}
    e^{\mathrm{i} 2 \omega_{1} | \xpos_{n} - \ypos |/c} \\
    \vdots \\
    e^{\mathrm{i} 2 \omega_{M} | \xpos_{n} - \ypos |/c}
  \end{bmatrix}.
  \label{eq:a-vector}
\end{equation}
Using \eqref{eq:sigma-plus} and \eqref{eq:a-vector}, we compute
the following imaging functional
\begin{equation}
  F_{\epsilon}(\ypos)
  = \frac{1}{N} \mathbf{a}^{\ast}(\ypos) 
  \begin{bmatrix}
    U_{1} \Sigma_{1}^{+} U_{1}^{\ast} & & \\
    & \ddots & \\
    & & U_{N} \Sigma_{N}^{+} U_{N}^{\ast}
  \end{bmatrix}
  \mathbf{a}(\ypos)
  = \frac{1}{N} \sum_{n = 1}^{N} \mathbf{a}_{n}^{\ast}(\ypos) U_{n}
  \Sigma_{n}^{+} U_{n}^{\ast} \mathbf{a}_{n}(\ypos).
  \label{eq:F-function}
\end{equation}
We also consider another imaging functional. For that imaging
functional, we introduce the complimentary illumination block-vector,
\begin{equation}
  \mathbf{b}(\ypos) = \begin{bmatrix}
    \mathbf{b}_{1}(\ypos) \\ \vdots \\ \mathbf{b}_{N}(\ypos)
  \end{bmatrix},
  \qquad
  \mathbf{b}_{n}(\ypos) = \frac{1}{4 \pi |\xpos_{n} - \ypos|}
  \begin{bmatrix}
    1 \\
    e^{-\mathrm{i} 2 \Delta \omega | \xpos_{n} - \ypos |/c} \\
    \vdots \\
    e^{-\mathrm{i} 2 (M-1) \Delta \omega | \xpos_{n} - \ypos |/c}
  \end{bmatrix},
  \label{eq:b-vector}
\end{equation}
and compute
\begin{equation}
  R_{\epsilon}(\ypos)
  = \frac{1}{N} \mathbf{b}^{\ast}(\ypos) 
  \begin{bmatrix}
    V_{1} \Sigma_{1}^{+} U_{1}^{\ast} & & \\
    & \ddots & \\
    & & V_{N} \Sigma_{N}^{+} U_{N}^{\ast}
  \end{bmatrix}
  \mathbf{a}(\ypos)
  = \frac{1}{N} \sum_{n = 1}^{N} \mathbf{b}_{n}^{\ast}(\ypos) V_{n}
  \Sigma_{n}^{+} U_{n}^{\ast} \mathbf{a}_{n}(\ypos).
  \label{eq:R-function}
\end{equation}
We form images through evaluation of $1/F_{\epsilon}(\ypos)$ and
$1/R_{\epsilon}(\ypos)$ over an imaging region. We show below that the
image formed using $F_{\epsilon}$ is useful for determining the
location and magnitude of reflectivities for point targets and the
image formed using $R_{\epsilon}$ is useful for determining the
complex reflectivities for point targets.

\begin{figure}[t]
  \centering
  \includegraphics[width=0.48\linewidth]{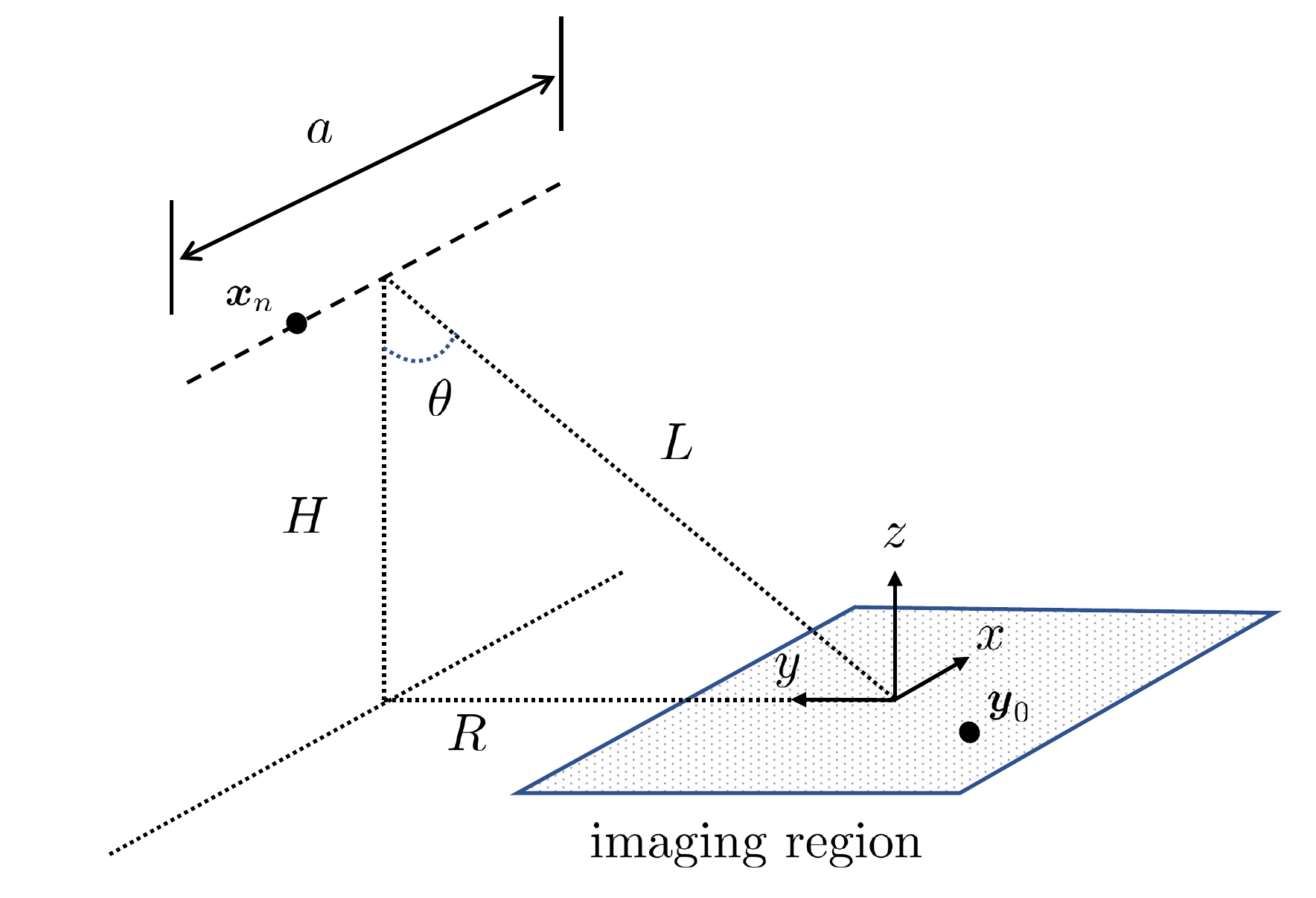}
  \caption{Sketch of the linear flight path over the planar imaging
    region used to study the resolution of the imaging method.}
  \label{fig:linear-flight-path}
\end{figure}

\section{Resolution analysis}
\label{sec:resolution}

To study the performance of imaging using \eqref{eq:F-function} and
\eqref{eq:R-function}, we consider one point target located in a
planar imaging region at position $\ypos_{0}$ with complex
reflectivity $\rho_{0}$. We use a coordinate system in which the
origin lies at the center of the planar imaging region.  The flight
path of the platform is linear and parallel to the $x$-axis. It is
offset from the origin along the $y$-axis by range $R$ and along the
$z$-axis by height $H$. Thus, spatial positions of the measurements
are $\xpos_{n} = ( x_{n}, R, H )$ for $n = 1, \cdots, N$ with
$x_{n} = -a/2 + a (n-1)/(N-1)$ and $a$ denoting the aperture. Let
$L = \sqrt{R^{2} + H^{2}}$ denote the distance from the center of the
flight path to the origin, and let $\theta$ denote the so-called look
angle with $\sin\theta = R/L$ and $\cos\theta = H/L$. We assume that
$L$ is the largest length scale in this problem. A sketch of this
linear flight path over a planar imaging region is shown in
Fig.~\ref{fig:linear-flight-path}.

To establish estimates for the resolution of image of one point target
produced through evaluation of $1/F_{\epsilon}(\ypos)$ over an imaging
region, let $\ypos_{0} = ( x_{0}, y_{0}, 0 )$ denote the target
location, $\ypos = ( x, y, 0)$ denote the search location, and $B$
denote the system bandwidth centered at frequency $f_{0}$ with
corresponding wavenumber $k_{0} = 2 \pi f_{0}/c$. With these
quantities defined, we prove the following theorem.

\begin{theorem}[Resolution estimates for a linear flight
  path]\label{thm:resolution}
  Assuming that the SNR is sufficiently high that we can distinguish
  the singular values corresponding to the signal subspace from those
  corresponding to the noise subspace, $1/F_{\epsilon}(\ypos)$ with
  $F_{\epsilon}(\ypos)$ given in \eqref{eq:F-function} attains a
  maximum of $|\rho_{0}|$ on $\ypos = \ypos_{0}$, and in the
  asymptotic limit, $L \gg 1$, $x_{0}/a \ll 1$, $y_{0}/R \ll 1$, and
  $\epsilon \ll 1$, this image has a cross-range resolution of
  $\Delta x^{\ast} = O(\sqrt{\epsilon} (c/B) (L/a) )$ and a range
  resolution of $\Delta y^{\ast} = O(\sqrt{\epsilon} (c/B) (L/R) )$.
\end{theorem}

\begin{proof}
  For a single point target, we have
  $D_{n} = \sigma_{0}^{(n)} \mathbf{u}_{0}^{(n)} \mathbf{v}_{0}^{(n)
    \ast}$ for $n = 1, \cdots, N$ with
  $\sigma_{0}^{(n)} = M |\rho_{0}|/(4 \pi |\xpos_{n} -
  \ypos_{0}|)^{2}$ and $\mathbf{u}_{0}^{(n)}$ and
  $\mathbf{v}_{0}^{(n)}$ given in \eqref{eq:uv}. Consequently, the
  column space of $D_{n}$ is
  $\mathcal{C}(D_{n}) = \text{span}\{ \mathbf{u}_{0}^{(n)} \}$ and
  $P^{(n)} = I - \mathbf{u}_{0}^{(n)} \mathbf{u}_{0}^{(n)\ast}$ is the
  projection onto subspace orthogonal to $\mathcal{C}(D_{n})$. Using
  \begin{equation*}
    \Sigma_{n}^{+} = \frac{1}{\sigma_{0}^{(n)}} \text{diag}\left( 1,
      \frac{1}{\epsilon}, \cdots, \frac{1}{\epsilon} \right),
  \end{equation*}
  we find that
  \begin{align}
    F_{\epsilon}(\ypos)
    &= \frac{1}{N} \sum_{n = 1}^{N} \left[ \frac{1}{\epsilon
      \sigma_{0}^{(n)}} \mathbf{a}_{n}^{\ast}(\ypos) (I -
      \mathbf{u}_{0}^{(n)} \mathbf{u}_{0}^{(n) \ast} )
      \mathbf{a}_{n}(\ypos) + \frac{1}{\sigma_{0}^{(n)}}
      \mathbf{a}_{n}^{\ast}(\ypos) \mathbf{u}_{0}^{(n)}
      \mathbf{u}_{0}^{(n) \ast} \mathbf{a}_{n}(\ypos) \right] \nonumber\\
    &= \frac{1}{|\rho_{0}|} \frac{1}{N} \sum_{n =
      1}^{N} \left[ \left( \frac{1}{\epsilon} - \left(
      \frac{1}{\epsilon} - 1 \right) | \Phi_{n} |^{2} \right)
      \frac{| \xpos_{n} - \ypos_{0} |^{2}}{|\xpos_{n} - \ypos |^{2}}
      \right],
      \label{eq:F-epsilon}
  \end{align}
  where we have introduced the quantity,
  \begin{equation}
    \Phi_{n} = \frac{1}{M} \sum_{m = 1}^{M} e^{\mathrm{i} 
      \omega_{m} \varDelta \tau_{n} },
    \label{eq:Phi_n}
  \end{equation}
  with
  $\varDelta \tau_{n} = 2 ( |\xpos_{n} - \ypos | - | \xpos_{n} -
  \ypos_{0} | )/c$ denoting the difference in travel times for the
  search and target locations.

  Evaluating \eqref{eq:F-epsilon} on $\ypos_{0}$, we find that
  $F_{\epsilon}(\ypos_{0}) = 1/|\rho_{0}|$, so
  $1/F_{\epsilon}(\ypos_{0}) = |\rho_{0}|$.  Because
  $|\Phi_{n}|^{2} \le 1$ and
  $|\xpos_{n} - \ypos_{0}|^{2}/|\xpos_{n} - \ypos|^{2} \le 1$ with
  both functions evaluating to $1$ only at $\ypos = \ypos_{0}$, this
  result corresponds the maximum value that $1/F_{\epsilon}(\ypos)$
  attains.

  Let
  \[
    \omega_{m} = \omega_{0} \left[ 1 + \beta \left( \frac{m-1}{M-1} -
        \frac{1}{2} \right) \right], \quad m = 1, \cdots, M,
  \]
  with $\beta = 2 \pi B/\omega_{0}$ denoting the fraction of the
  bandwith about the central frequency. Substituting these frequencies
  into \eqref{eq:Phi_n} and computing the sum, we find
  \[
    \Phi_{n} = \frac{e^{\mathrm{i} \omega_{0} (1 - \beta/2)
      \varDelta\tau_{n}}}{M} \frac{1 - e^{\mathrm{i} \omega_{0} \beta M
        \varDelta \tau_{n}/(M-1)}}{1 - e^{\mathrm{i} \omega_{0} \beta
        \varDelta \tau_{n}/(M-1)}},
  \]
  from which it follows that
  \[
    |\Phi_{n}|^{2} = \frac{1}{M^{2}} \frac{\sin^{2}\left( \frac{\pi M
          B \varDelta\tau_{n}}{M-1} \right)}{\sin^{2}\left( \frac{\pi
          B \varDelta\tau_{n}}{M-1} \right)}.
  \]
  In the expression above, we have resubstituted
  $\omega_{0} \beta = 2 \pi B$. Assuming we are in a small
  neighborhood about the target location, we expand the expression
  above about $\varDelta \tau_{n} = 0$ and obtain
  \[
    |\Phi_{n}|^{2} = 1 - \frac{\pi^{2} B^{2}}{3} \frac{M+1}{M-1}
    \varDelta \tau_{n}^{2} + O(\varDelta \tau_{n}^{4}).
  \]
  Additionally, we find that
  \[
    \frac{|\xpos_{n} - \ypos_{0}|}{|\xpos_{n} - \ypos|} = \frac{(x_{n}
      - x_{0})^{2} + (L \sin\theta - y_{0})^{2} + (L
      \cos\theta)^{2}}{(x_{n} - x)^{2} + (L \sin\theta
      - y)^{2} + (L \cos\theta)^{2}} = 1 + O(L^{-1}),
  \]
  where $\sin\theta = R/L$ and $\cos\theta = H/L$. Using these
  approximations, we find that
  \begin{equation}
    F_{\epsilon}(\ypos) = \frac{1}{|\rho_{0}|} \left[
      1 + \left( \frac{1}{\epsilon} - 1 \right)
      \frac{\pi^{2} B^{2}}{3} \frac{M+1}{M-1} \left( \frac{1}{N} \sum_{n =
        1}^{N} \varDelta \tau_{n}^{2} \right) \right] + O(\varDelta
    \tau_{n}^{4}, L^{-1}).
    \label{eq:F-local}
  \end{equation}
  Next, we use
  \begin{multline}
    \varDelta \tau_{n} = (|\xpos_{n} - \ypos| - |\xpos_{n} -
    \ypos_{0}|)/c
    =  \cos\theta \frac{(y - y_{0})}{c}\\
    + \frac{(x - x_{0})^{2} - 2 (x - x_{0} ) ( \xi_{n} - x_{0})}{2 c L} +
    \sin^{2}\theta \frac{( y - y_{0} )^{2} + 2 ( y - y_{0} ) y_{0}}{2 c L}
    + O(L^{-2}).
    \label{eq:Fresnel}
  \end{multline}
  For the cross-range resolution, we evaluate \eqref{eq:Fresnel} on
  $y = y_{0}$ and find that
  \[
    \varDelta \tau_{n}^{2} \big|_{y = y_{0}} \sim \frac{(x -
      x_{0})^{2}}{c^{2} L^{2}} ( x_{n} - x_{0} )^{2}.
  \]
  Using $x_{n} = -a/2 + a (n-1)/(N+1)$, we find that
  \[
    \frac{1}{N} \sum_{n = 1}^{N} ( x_{n} - x_{0} )^{2} =
    \frac{a^{2}}{12} \frac{N + 1}{N - 1} + x_{0}^{2},
  \]
  and so
  \[
    F_{\epsilon}(x,y_{0}) \sim \frac{1}{|\rho_{0}|} \left[
      1 + \left( \frac{1}{\epsilon} - 1 \right)
      \frac{\pi^{2} B^{2} a^{2}}{3 c^{2} L^{2}}  \frac{M+1}{M-1}
      \left( \frac{1}{12} \frac{N + 1}{N - 1} +
        \frac{x_{0}^{2}}{a^{2}} \right) ( x - x_{0} )^{2} \right].
  \]
  The full-width/half-maximum (FWHM) in cross-range $\Delta x^{\ast}$
  satisfies
  $1/F_{\epsilon}(x_{0} + \Delta x^{\ast},y_{0}) = 1/(2
  |\rho_{0}|)$. Substituting the approximation above into this
  definition, solving for $\Delta x^{\ast}$, and expanding that
  result about $\epsilon = 0$, we find
  \begin{align*}
    \Delta x^{\ast} &= \pm \sqrt{\epsilon} \frac{c}{B} \frac{L}{a}
                      \frac{6}{\pi} \sqrt{\frac{M-1}{M+1}}
                      \sqrt{\frac{N-1}{(N+1) + 12 (N-1)
                      (x_{0}^{2}/a^{2})}} + O(\epsilon^{3/2})\\
                    &= \pm \sqrt{\epsilon} \frac{c}{B} \frac{L}{a}
                      \frac{6}{\pi} \sqrt{\frac{M-1}{M+1}}
                      \sqrt{\frac{N-1}{N+1}} +
                      O\left(\epsilon^{3/2},\frac{x_{0}^{2}}{a^{2}} \right)\\
                    &= O\left( \sqrt{\epsilon} \frac{c}{B}
                      \frac{L}{a} \right).
  \end{align*}
  For the range resolution, we evaluate \eqref{eq:Fresnel} on
  $x = x_{0}$ and find that
  \[
    \varDelta \tau_{n}^{2} \big|_{x = x_{0}} = \frac{( y - y_{0}
      )^{2}}{c^{2}} \left( \cos\theta - \sin^{2}\theta
      \frac{y_{0}}{L} \right)^{2}.
  \]
  It follows that
  \[
    F_{\epsilon}(x_{0},y) \sim \frac{1}{|\rho_{0}|} \left[
      1 + \left( \frac{1}{\epsilon} - 1 \right) \frac{\pi^{2} B^{2}}{3
        c^{2}} \frac{M+1}{M-1} \left( \cos\theta - \sin^{2}\theta
        \frac{y_{0}}{L} \right)^{2} ( y - y_{0} )^{2} \right].
  \]
  The full-width/half-maximum (FWHM) in range $\Delta y^{\ast}$
  satisfies
  $1/F_{\epsilon}(x_{0},y_{0} + \Delta y^{\ast}) = 1/(2
  |\rho_{0}|)$. Substituting the approximation above into this
  definition, resubstituting $\cos\theta = R/L$, solving for
  $\Delta y^{\ast}$, and expanding that result about $\epsilon = 0$,
  we find
  \begin{align*}
      \Delta y^{\ast} &= \pm \sqrt{\epsilon} \frac{c}{B} \frac{L}{R}
                        \frac{1}{1 - \sin^{2}\theta (y_{0}/R)}
                        \frac{\sqrt{3}}{\pi} \sqrt{\frac{M-1}{M+1}} +
                        O(\epsilon^{3/2}) \\
                      &= \pm \frac{\sqrt{3}}{\pi} \sqrt{\epsilon}
                        \frac{c}{B} \frac{L}{R} +
                        O\left( \epsilon^{3/2},\frac{y_{0}}{R}
                        \right)\\
                      &= O\left( \sqrt{\epsilon} \frac{c}{B}
                        \frac{L}{R} \right).
  \end{align*}
  This completes the proof.
\end{proof}

Theorem \ref{thm:resolution} states that images of a point target
formed through evaluation of $1/F_{\epsilon}(\ypos)$ with
$F_{\epsilon}(\ypos)$ given in \eqref{eq:F-function} will form an
image that is peaked at the location of the target with magnitude
equal to $|\rho_{0}|$. Because the user-defined parameter $\epsilon$
can be made arbitrarily small, this imaging method will yield
high-resolution images provided that there is sufficient signal that
the non-trivial singular values provide accurate quantitative data.

In general, the reflectivity of a point target is complex. To recover
the complex reflectivity, we make use of the following theorem.

\begin{theorem}[Recovery of the complex
  reflectivity]\label{thm:reflectivity}
  For a point target located at $\ypos_{0}$ with complex reflectivity
  $\rho_{0}$, when the SNR is sufficient high that we can distinguish
  the signal subspace from the noise subspace,
  $1/R_{\epsilon}(\ypos_{0}) = \rho_{0}$ with $R_{\epsilon}(\ypos)$
  given in \eqref{eq:R-function}.
\end{theorem}

\begin{proof}
  Through direct evaluation of $R_{\epsilon}$ given in
  \eqref{eq:R-function} on $\ypos = \ypos_{0}$, we find
  $R_{\epsilon}(\ypos_{0}) = e^{-\mathrm{i} \theta_{0}}/|\rho_{0}|$.
  It follows that
  $1/R_{\epsilon}(\ypos_{0}) = |\rho_{0}| e^{\mathrm{i} \theta_{0}} =
  \rho_{0}$.
\end{proof}

Although Theorem \ref{thm:reflectivity} states that evaluating
$1/R_{\epsilon}(\ypos)$ yields the complex reflectivity, it is not
generally useful for determining the location of the target because
this function does not exhibit localized behavior that indicates the 
region about the target location. For this reason, we propose the
following two-stage imaging method.
\begin{enumerate}

\item[(i)] Evaluate $1/F_{\epsilon}(\ypos)$ with $F_{\epsilon}(\ypos)$
  given in \eqref{eq:F-function} to determine the location of
  targets. The value of $\epsilon$ may be varied to adjust the
  resolution of this image.

\item[(ii)] Evaluate $1/R_{\epsilon}(\ypos)$ with
  $R_{\epsilon}(\ypos)$ given in \eqref{eq:R-function} using the
  locations determined in (i) to determine the complex reflectivities
  of the targets.

\end{enumerate}

\section{Travel time uncertainty}
\label{sec:random}

We now consider the effect of uncertainty in the travel times on
images formed through evaluation of $1/F_{\epsilon}(\ypos)$ with
$F_{\epsilon}(\ypos)$ given in \eqref{eq:F-function}. Uncertainty in
travel times can arise from sampling clock jitter, deviations from the
assumed flight path, and random fluctuations in the propagating medium
among other practical issues. It is therefore important to understand
to what extent images formed using the method described above are
useful under uncertain conditions.

To model travel time uncertainty, we use
\begin{equation}
  \varDelta \tau_{n} = \varDelta \tau_{n}^{0} + \nu_{n}, \quad n = 1,
  \cdots, N.
  \label{eq:travel-time-model}
\end{equation}
with $\varDelta \tau_{n}^{0}$ denoting the difference in travel times
for a homogeneous medium and the vector,
$\boldsymbol{\nu} = (\nu_{1}, \cdots, \nu_{N})$ denoting a
multivariate distribution with $\mathbb{E}[\nu_{n}] = 0$ and
$\mathbb{E}[\nu_{n}^{2}] = \sigma_{n}^{2}$ for $n = 1, \cdots, N$. Let
$\sigma^{2} = \max \{ \sigma_{1}^{2}, \cdots \sigma_{N}^{2} \}$. Using
this model for travel time uncertainty, we prove the following
theorem.

\begin{theorem}[Travel time uncertainty]\label{thm:random}
  Assuming that the SNR is sufficiently high that we can distinguish
  the singular values corresponding to the signal subspace from those
  corresponding to the noise subspace, the image formed through
  evaluation of $1/F_{\epsilon}(\ypos)$ in a neighborhood about
  $\ypos = \ypos_{0}$ with $F_{\epsilon}(\ypos)$ given in
  \eqref{eq:F-function} and using \eqref{eq:travel-time-model} with
  $\sigma^{2}/\epsilon \ll 1$ has an expected value whose leading
  behavior is the result for the homogeneous medium plus a term that
  is $O(\sigma^{2}/\epsilon)$, and has a variance that is
  $O(\sigma^{2}/\epsilon)$.
\end{theorem}

\begin{proof}
  Since we consider a neighborhood about $\ypos = \ypos_{0}$, we start
  with \eqref{eq:F-local} and write
  \[
    F_{\epsilon} \sim
    \frac{1}{|\rho_{0}|} \left[ 1 + \alpha \left( \frac{1}{\epsilon} - 1
      \right) \left( \frac{1}{N} \sum_{n = 1}^{N} \varDelta
        \tau_{n}^{2} \right) \right],
  \]
  with $\alpha = \pi^{2} B^{2} (M+1)/(3 (M-1))$. Substituting
  \eqref{eq:travel-time-model} yields
  \[
    F_{\epsilon} \sim \frac{1}{|\rho_{0}|} \left[ 1 + \alpha \left(
        \frac{1}{\epsilon} - 1 \right) \left( \frac{1}{N} \sum_{n =
          1}^{N} \left( \varDelta \tau_{n}^{0} + \nu_{n} \right)^{2}
      \right) \right].
  \]
  Based on our resolution estimates, we introduce the stretched
  variables $\varDelta \tau_{n} = \sqrt{\epsilon} \varDelta T_{n}$ for
  $n = 1, \cdots, N$, and obtain
  \[
    F_{\epsilon} \sim \frac{1}{|\rho_{0}|} \left[ 1 + \alpha \left(
        1 - \epsilon \right) \left( \frac{1}{N} \sum_{n =
          1}^{N} \left( \varDelta T_{n} + \nu_{n}/\sqrt{\epsilon}
        \right)^{2} \right) \right].
  \]
  Let $f_{Y}(y_{1}, \cdots, y_{N})$ denote the probability density
  function for $(\nu_{1}, \cdots, \nu_{N})$. The expected value of the
  image is then
  \[
    \mathbb{E}\left[ \frac{1}{F_{\epsilon}} \right] \sim |\rho_{0}|
    \int \cdots \int \frac{f_{Y}(y_{1}, \cdots, y_{N})}{1 + \alpha
      \left( 1 - \epsilon \right) \left( \displaystyle \frac{1}{N}
        \sum_{n = 1}^{N} \left( \varDelta T_{n} +
          y_{n}/\sqrt{\epsilon} \right)^{2} \right)} \mathrm{d}y_{1}
    \cdots \mathrm{d}y_{N}.
  \]
  Substituting $y_{n} = \sigma \eta_{n}$ yields
  \[
    \mathbb{E}\left[ \frac{1}{F_{\epsilon}} \right] \sim |\rho_{0}|
    \int \cdots \int \frac{\sigma^{N} f_{Y}(\sigma \eta_{1}, \cdots,
      \sigma \eta_{N})}{1 + \alpha \left( 1 - \epsilon \right) \left(
        \displaystyle \frac{1}{N} \sum_{n = 1}^{N} \left( \varDelta
          T_{n} + \sigma \eta_{n}/\sqrt{\epsilon} \right)^{2}
      \right)} \mathrm{d}\eta_{1} \cdots \mathrm{d}\eta_{N}.
  \]
  Assuming that $\delta = \sigma/\sqrt{\epsilon} \ll 1$, we expand
  about $\delta = 0$ and find
  \[
    \left[ 1 + \alpha \left( 1 - \epsilon \right) \left( \displaystyle
        \frac{1}{N} \sum_{n = 1}^{N} \left( \varDelta T_{n} + \delta
          \eta_{n} \right)^{2} \right) \right]^{-1} = I_{0}
    - 2 \delta \alpha (1 - \epsilon) I_{0}^{2} \left( \frac{1}{N}
      \sum_{n = 1}^{N} \varDelta T_{n} \eta_{n} \right) + O(\delta^{2}),
  \]
  with
  \[
    I_{0} = \left[ 1 +
      \alpha (1 - \epsilon ) \frac{1}{N} \sum_{n = 1}^{N} (\varDelta
      T_{n})^{2} \right]^{-1},
  \]
  denoting the normalized image formed in the homogeneous
  medium. Substituting this expansion into the integral above for the
  expected value of the image and using $\mathbb{E}[\nu_{n}] = 0$ for
  $n = 1, \cdots, N$, we find that
  \[
    \mathbb{E}\left[ \frac{1}{F_{\epsilon}} \right] = |\rho_{0}| I_{0}
    + O(\delta^{2}).
  \]
  Next, by using the expansion
  \[
    \left[ 1 + \alpha \left( 1 - \epsilon \right) \left( \displaystyle
        \frac{1}{N} \sum_{n = 1}^{N} \left( \varDelta T_{n} + \delta
          \eta_{n} \right)^{2} \right) \right]^{-2} = I_{0}^{2}
    - \delta \alpha (1 - \epsilon) I_{0}^{3} \left( 
      \sum_{n = 1}^{N} \varDelta T_{n} \eta_{n} \right)
    + O(\delta^{2}),
  \]
  we determine that
  \[
    \mathbb{E}\left[ \left( \frac{1}{F_{\epsilon}} \right)^{2} \right]
    = |\rho_{0}|^{2} I_{0}^{2} + O(\delta^{2}).
  \]
  Therefore,
  \[
    \text{Var}\left[ \frac{1}{F_{\epsilon}} \right] = \mathbb{E}\left[
      \left( \frac{1}{F_{\epsilon}} \right)^{2} \right] - \left(
      \mathbb{E}\left[  \frac{1}{F_{\epsilon}} \right] \right)^{2} =
    O(\delta^{2}).
  \]
\end{proof}
Theorem \ref{thm:random} states that when
$\sigma/\sqrt{\epsilon} \ll 1$, the leading behavior of the
expectation of the image with random perturbations to the travel time
is exactly the same as the image in the homogeneous medium. The
recovery of the magnitude of the reflectivity $|\rho_{0}|$, and the
resolution estimates of Theorem \ref{thm:resolution} are different by
a term that is $O(\sigma^{2}/\epsilon)$. Because the variance of the
image is $O(\sigma^{2}/\epsilon)$, we determine that this image formed
is statistically stable.

An immediate consequence of Theorem \ref{thm:random} is given
in the following corollary.
\begin{corollary}[Resolution with travel time
  uncertainty]\label{thm:sigma-epsilon}
  When $\sigma^{2}$ is known or can be reliably estimated, one can set
  the value of $\epsilon$ so that $\sigma^{2} \ll \epsilon$ and
  Theorem \ref{thm:random} will hold. 
\end{corollary}
Setting $\epsilon$ in this way connects the resolution of the image
with the variance of the random perturbations to the travel time.

\section{Numerical results}
\label{sec:results}

To validate the theoretical results from above, we use numerical
simulations to generate data for various scattering scenes. The
following values for the parameters are based on the GOTCHA data
set~\cite{casteel2007challenge}. In particular, we have set
$R = 3.55\, \text{km}$ and $H = 7.30\, \text{km}$, so that
$L = \sqrt{H^{2} + R^{2}} = 8.12\, \text{km}$.  The synthetic aperture
created by the linear flight path is $a = 0.13\, \text{km}$. The
central frequency is $f_{0} = 9.6\, \text{GHz}$ and the bandwidth is
$B = 622\, \text{MHz}$. Using $c = 3 \times 10^{8}\, \text{m/s}$, we
find that the central wavelength is $\lambda_{0} = 3.12\,
\text{cm}$. The imaging region is at the ground level $z = 0$. We use
$2M-1 = 39$ frequencies so that $M = 20$, and $N = 32$ spatial
measurements.

\begin{figure}[t]
  \centering
  \includegraphics[width=0.32\linewidth]{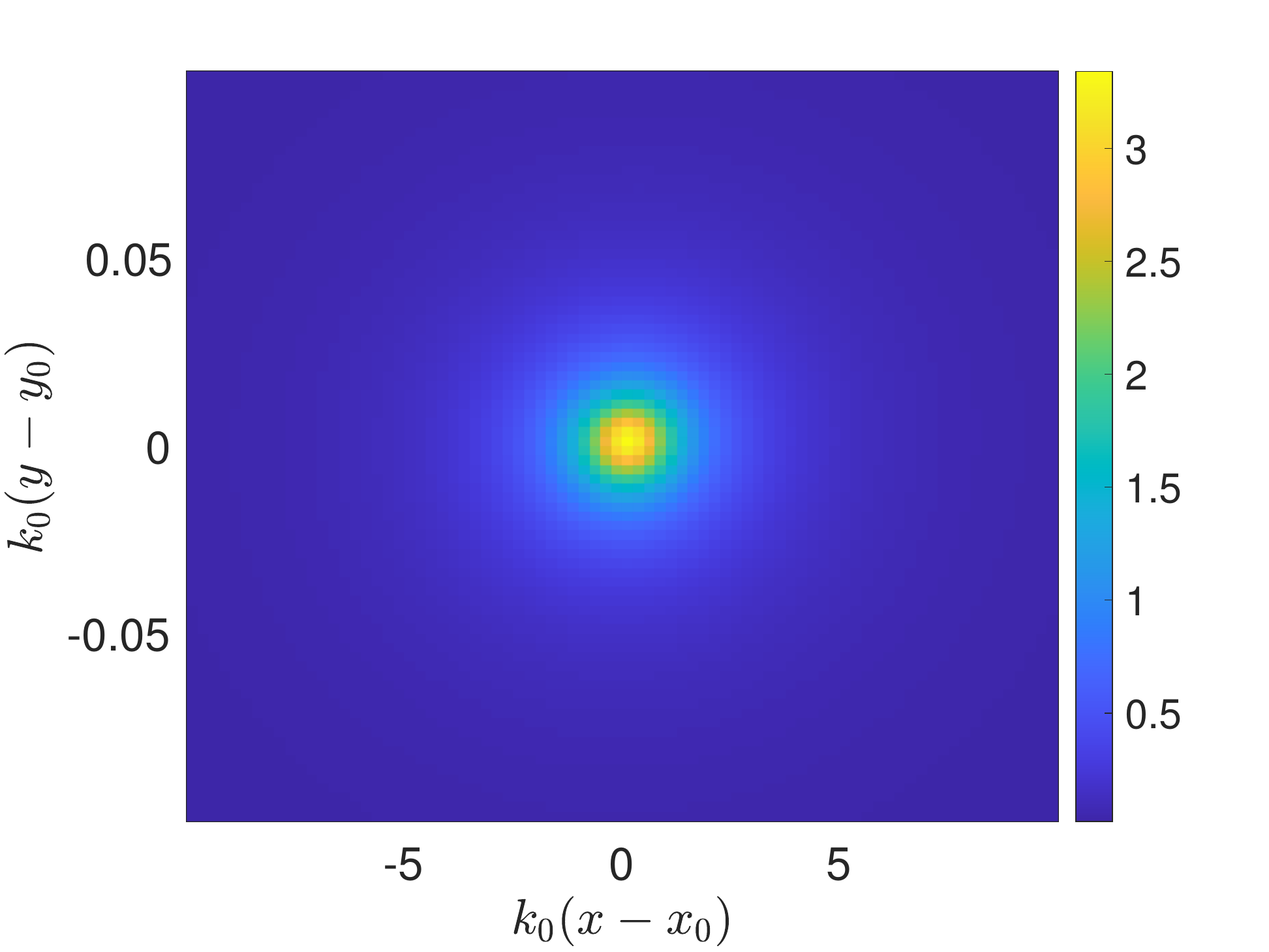}
  \includegraphics[width=0.32\linewidth]{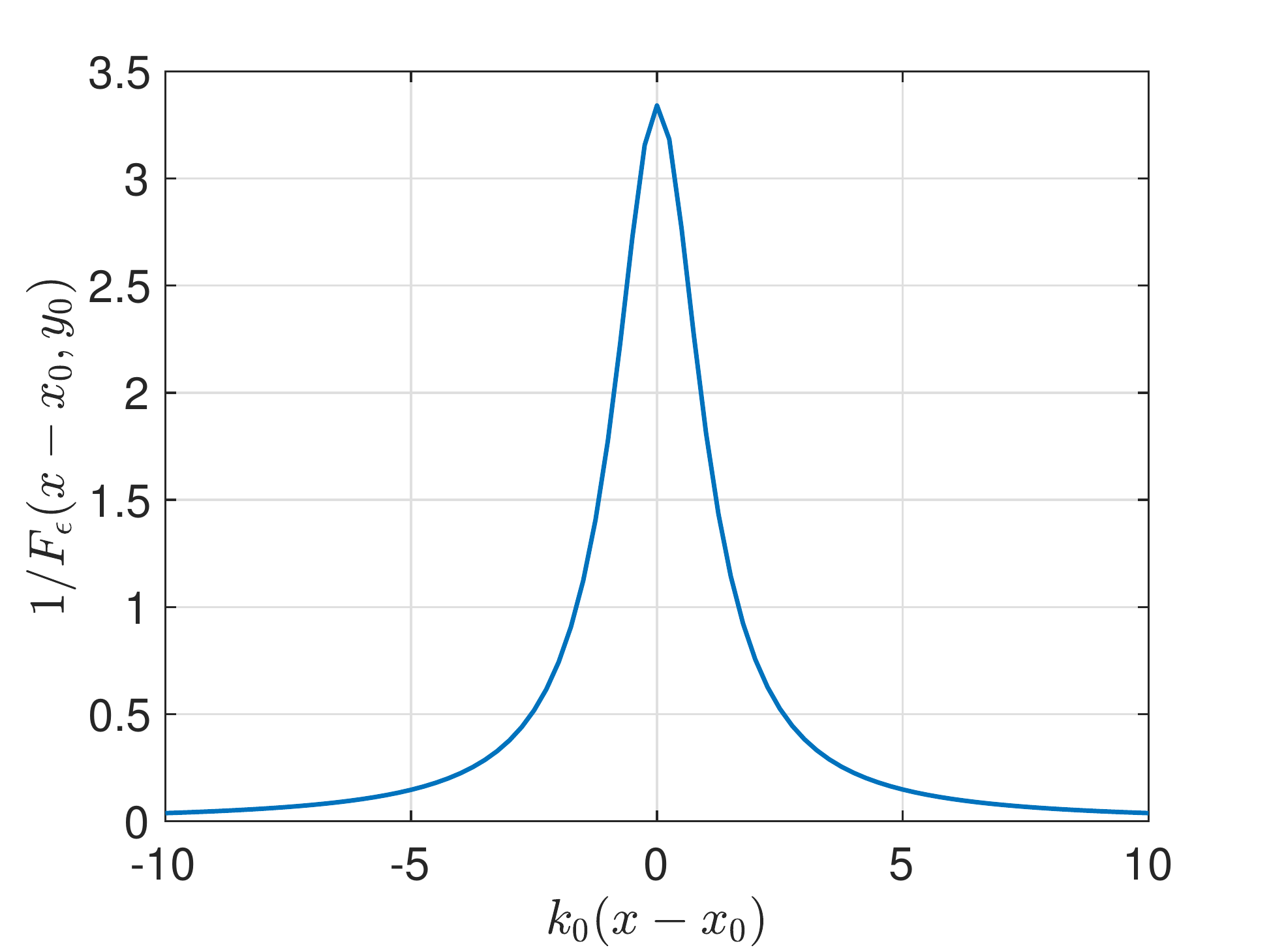}
  \includegraphics[width=0.32\linewidth]{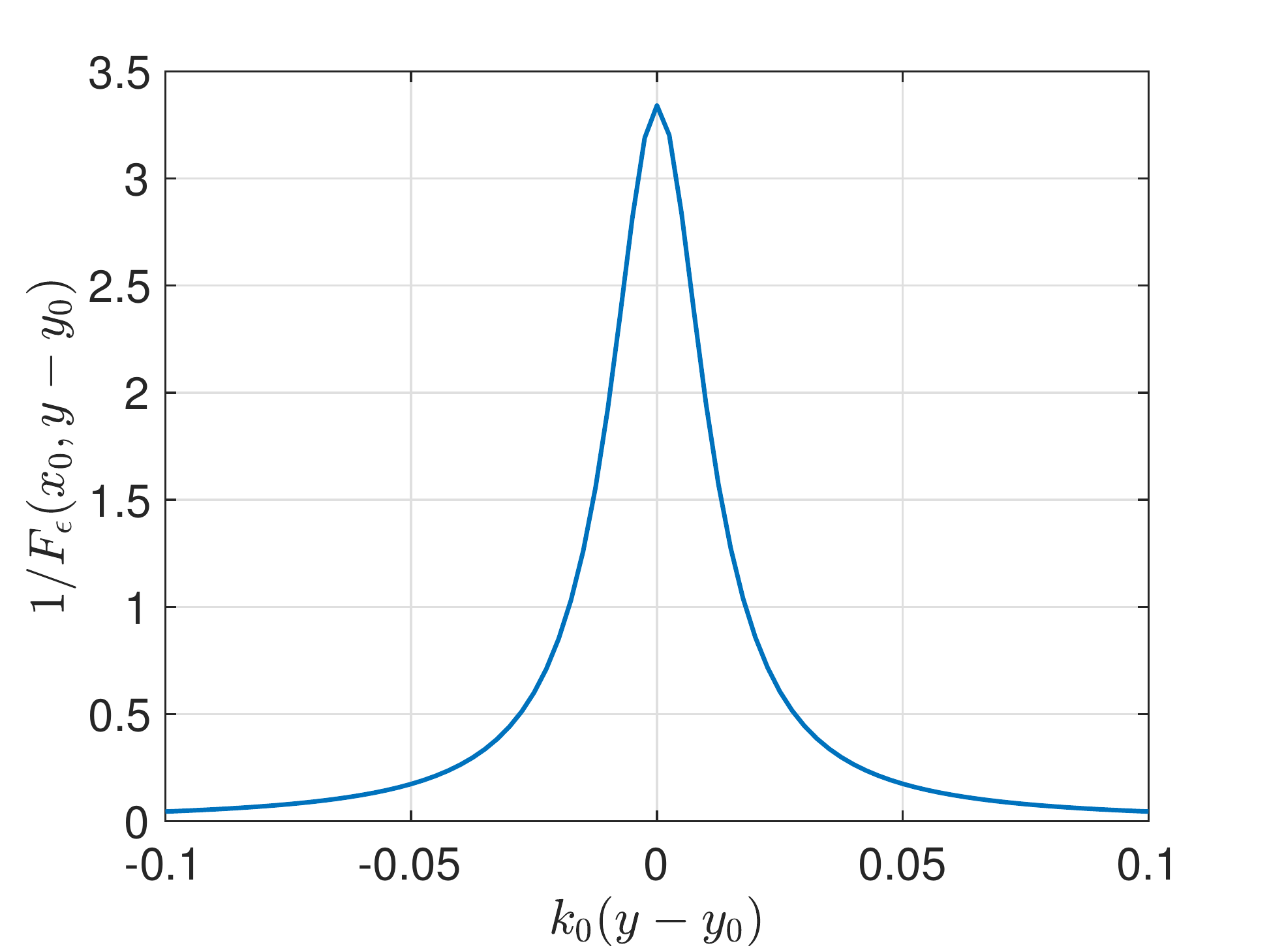}
  \caption{Image formed through evaluation of $1/F_{\epsilon}(\ypos)$
    with $F_{\epsilon}(\ypos)$ given in \eqref{eq:F-function} for a
    point target located at $x_{0} = y_{0} = 1\, \text{m}$ on the
    planar imaging region. Measurement noise was added so that
    $\text{SNR} = 44.1339 \text{dB}$. For this image,
    $\epsilon = 10^{-8}$.  }
  \label{fig:point-target}
\end{figure}

\subsection{Single point target}

We first consider imaging a single point target located at
$x_{0} = y_{0} = 1\, \text{m}$ with complex reflectivity
$\rho_{0} = 3.4 \mathrm{i}$ on the planar imaging region. Figure
\ref{fig:point-target} shows the image formed through evaluation of
$1/F_{\epsilon}(\ypos)$ with $F_{\epsilon}(\ypos)$ given in
\eqref{eq:F-function} with $\epsilon = 10^{-10}$. Measurement noise
was added so that the signal-to-noise ratio (SNR) is
$\text{SNR} = 44.1339\, \text{dB}$. The left plot of
Fig.~\ref{fig:point-target} shows the color contour plot of the image
in a region about the target location. The center plot of
Fig.~\ref{fig:point-target} shows the image on $y = y_{0}$ as a
function of $x$ (cross-range), and the right plot shows the image on
$x = x_{0}$ as a function of $y$ (range).  These results shown in
Fig.~\ref{fig:point-target} show that the image attains its maximum
value of $3.4$ corresponding to $|\rho_{0}|$ at the correct target
location.  The image attains a high resolution due to choice of
$\epsilon$. Because $L/R = 2.29$ and $L/a = 62.46$, we expect from the
resolution estimates given in Theorem \ref{thm:resolution} that the
range resolution should be better than the cross-range
resolution. This difference in resolution can be observed by noting
the values of $k_{0} (x - x_{0})$ in the center plot compared to the
values of $k_{0} (y - y_{0})$ in the right plot of
Fig.~\ref{fig:point-target}.

\begin{figure}[htb]
  \centering
  \includegraphics[width=0.40\linewidth]{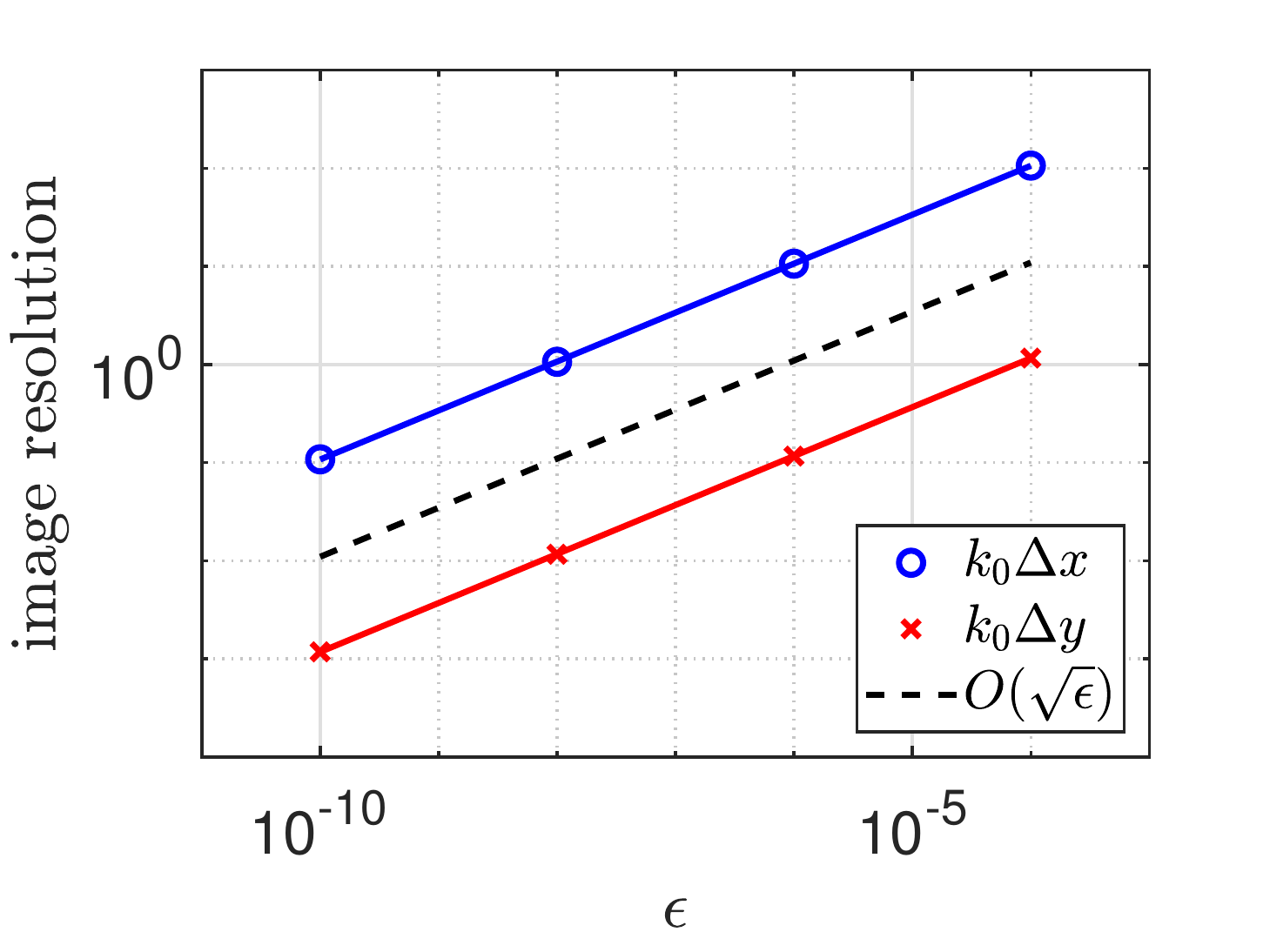}
  \includegraphics[width=0.40\linewidth]{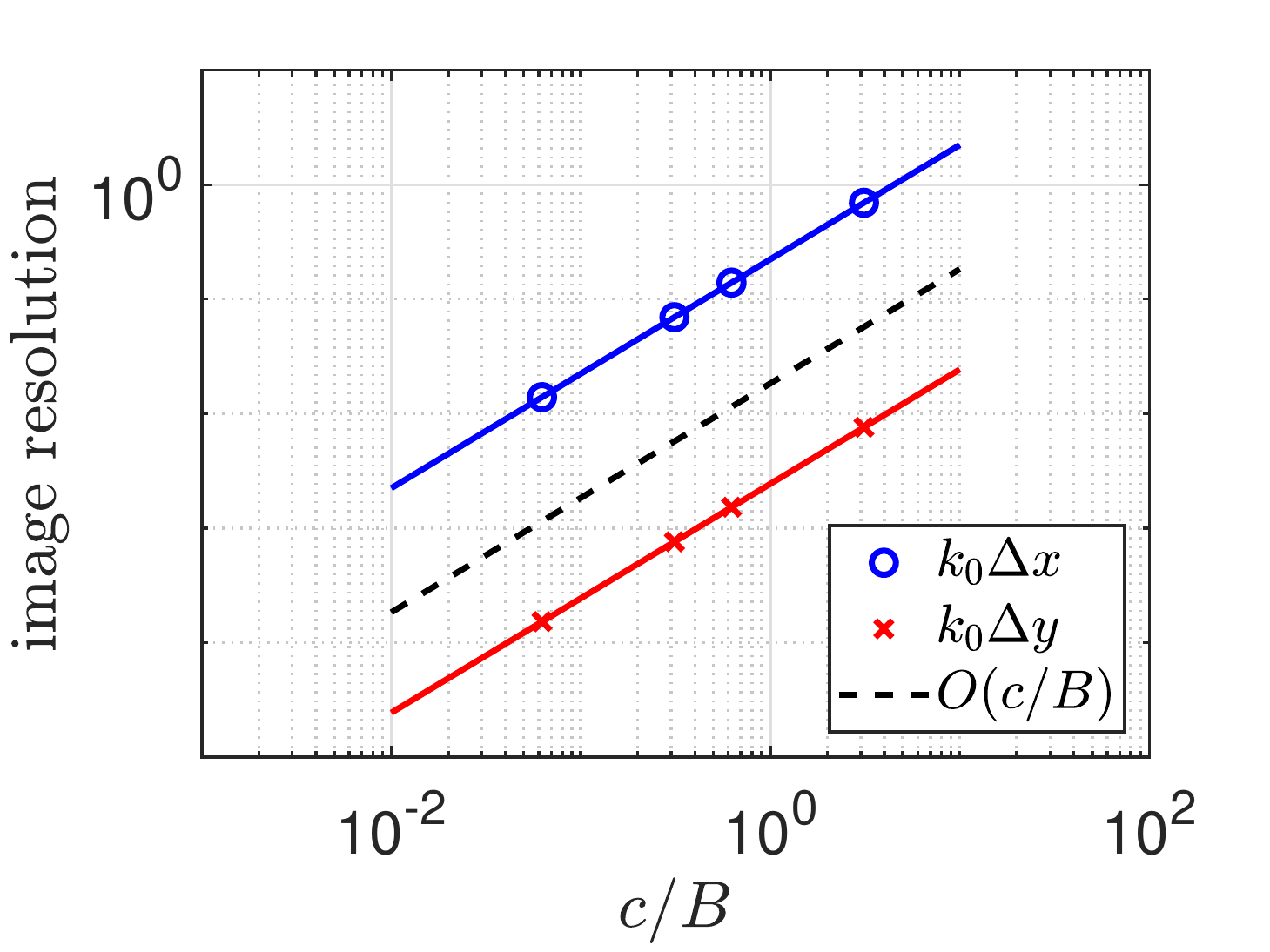}
  \caption{Numerically computed image resolutions with respect to
    $\epsilon$ (left) and $c/B$ (right). The range resolutions,
    $k_{0} \Delta x$, are plotted as ``$\circ$'' symbols, and the
    cross-range resolution, $k_{0} \Delta y$, are plotted as red
    ``$\times$'' symbols. The blue and red curves are the
    least-squares fit to lines through the range and cross-range
    resolution data, respectively.}
  \label{fig:epsilon+bandwidth}
\end{figure}

In Fig.~\ref{fig:epsilon+bandwidth} we show numerically computed FWHM
values of $k_{0} \Delta x$ (cross-range resolution) and
$k_{0} \Delta y$ (range resolution) for a single point target when
varying $\epsilon$ (left plot) and $c/B$ (right plot). The blue
``$\circ$'' symbols are the computed values of $k_{0} \Delta x$ and
the red ``$\times$'' symbols are the computed values of
$k_{0} \Delta y$ both found by numerically determining the FWHM. The
solid blue and red curves are the least-squares linear fit through the
$\Delta x$ and $\Delta y$ data, respectively. For the results shown in
the left plot of Fig.~\ref{fig:epsilon+bandwidth}, all parameters are
set to the same values used for Fig.~\ref{fig:point-target}, except
that $\text{SNR} = \infty$, so there is no noise. For the right plot of
Fig.~\ref{fig:epsilon+bandwidth}, we have varied the value of $B$, but
all other parameter values are the same as those used for
Fig.~\ref{fig:point-target}.  In these results, we find that
$k_{0} \Delta x > k_{0} \Delta y$ for all values of $\epsilon$ and
$c/B$ which is due to the fact that $L/R < L/a$.

The results for cross-range and range resolutions with respect to
$\epsilon$ given in the left plot of Fig.~\ref{fig:epsilon+bandwidth}
clearly show an $O(\sqrt{\epsilon})$ behavior which is plotted as a
dashed-black curve in the left plot of
Fig.~\ref{fig:epsilon+bandwidth}. The computed least-squares fits are
$\log(\Delta x) \approx 3.9593 + 0.4991 \log(\epsilon)$ and
$\log(\Delta y) \approx -0.5565 + 0.4992 \log(\epsilon)$ which
numerically validate this $O(\sqrt{\epsilon})$ behavior.

The results for cross-range and range resolutions with respect to
$c/B$ given in the right plot of Fig.~\ref{fig:epsilon+bandwidth}
clearly show an $O(c/B)$ behavior which is plotted as a dashed-black
curve. The least-squares fits are
$\log(\Delta x) \approx -6.8067 + 0.9997 \log(c/B)$ and
$\log(\Delta y) \approx -11.3247 + 0.9999 \log(c/B)$ which numerically
validate the $O(c/B)$ behavior.

\begin{figure}[htb]
  \centering
  \includegraphics[width=0.4\linewidth]{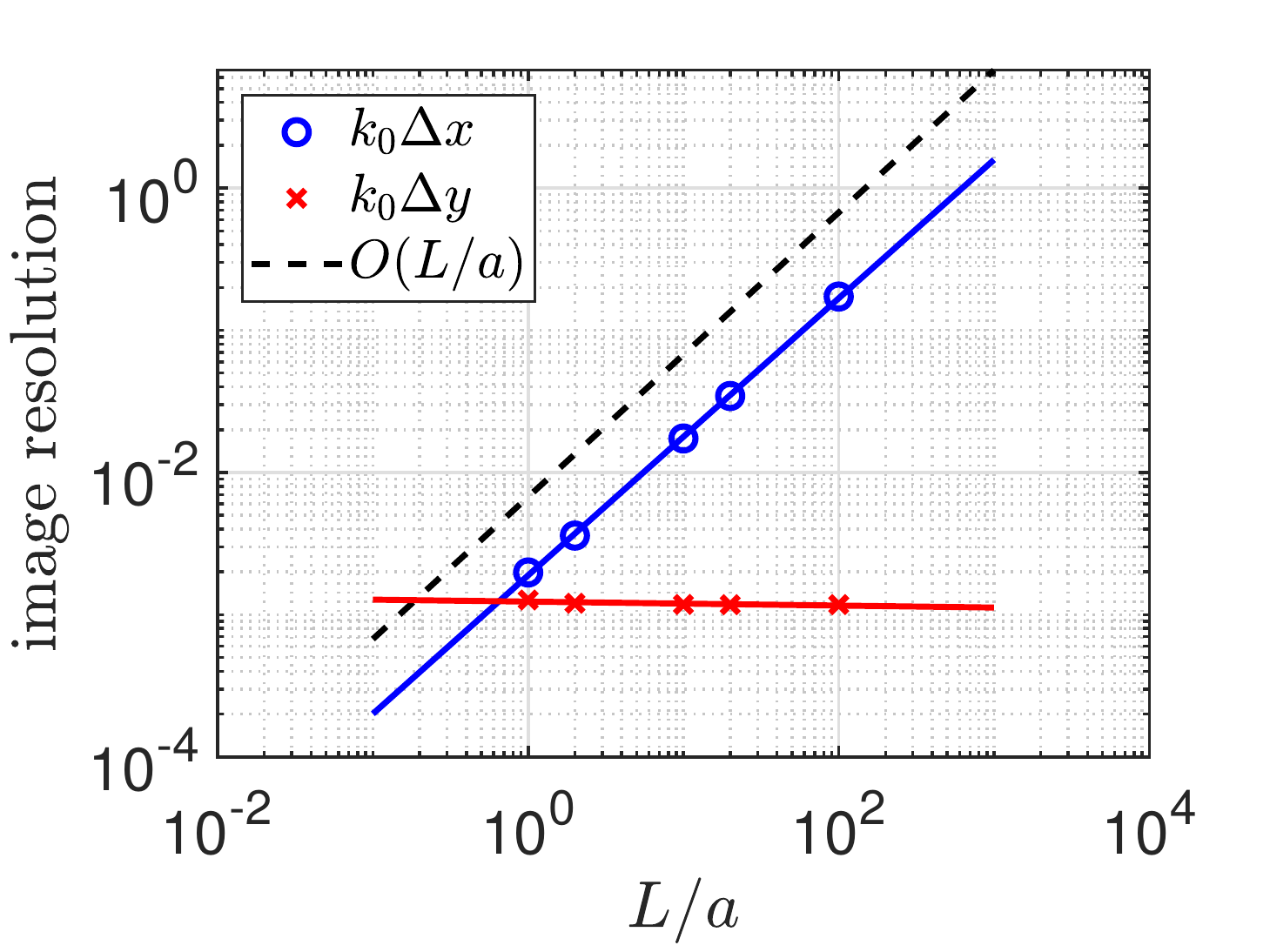}
  \includegraphics[width=0.4\linewidth]{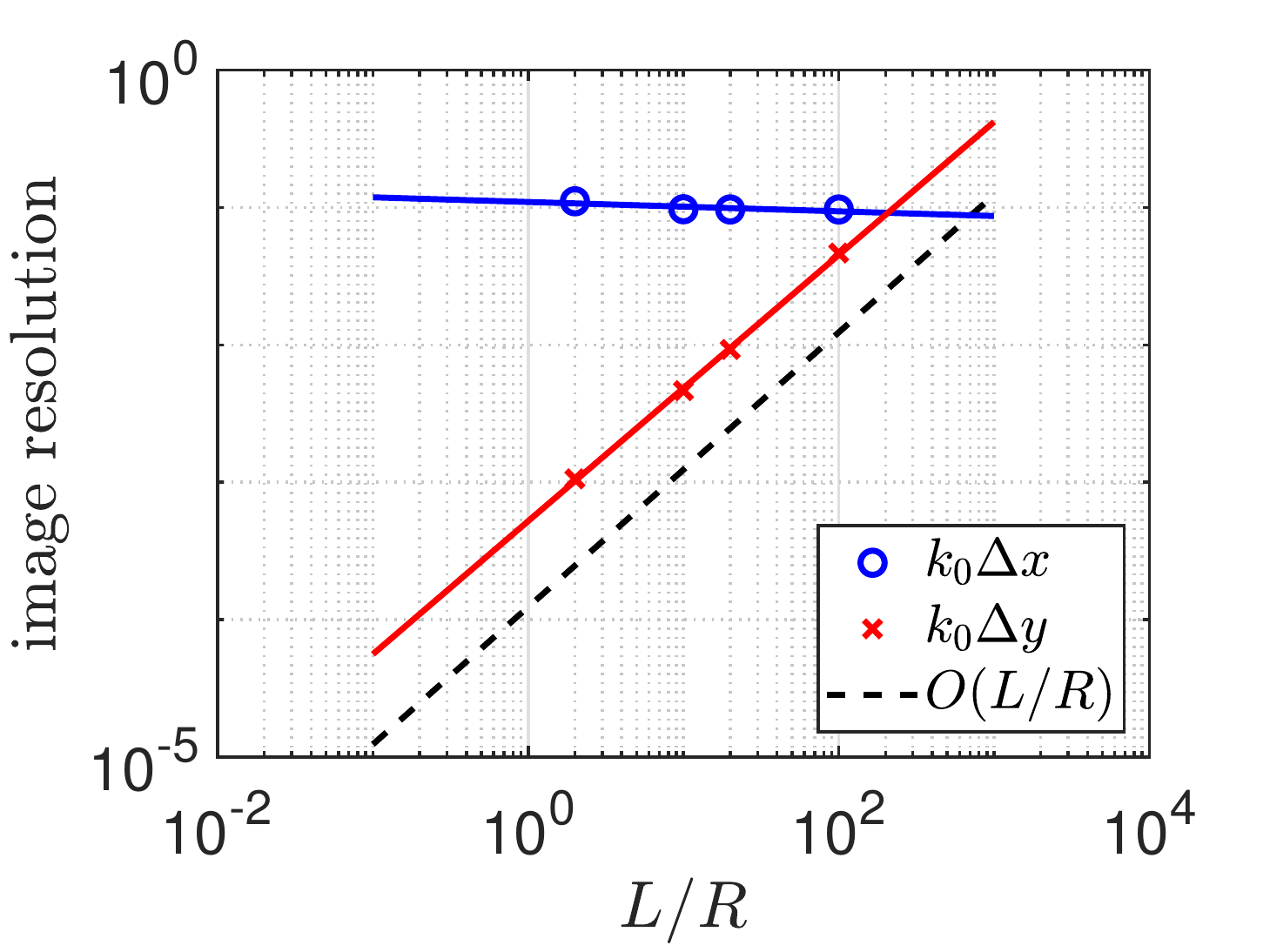}
  \caption{Numerical computed image resolutions with respect to $L/a$
    (left) and $L/R$ (right). The cross-range resolutions,
    $k_{0} \Delta x$, are plotted as ``$\circ$'' symbols, and the
    range resolutions, $k_{0} \Delta y$, are plotted as red
    ``$\times$'' symbols. The blue and red curves are the
    least-squares fit to lines through the cross-range and range
    resolution data, respectively.}
  \label{fig:aperture+radius}
\end{figure}

The behaviors of computed image resolution with respect to $L/a$ and
$L/R$ are shown in Fig.~\ref{fig:aperture+radius}. For these results,
all parameter values are the same as those used for
Fig.~\ref{fig:point-target} except that $\text{SNR} = \infty$, so
there is no noise and $a$ is varied in the left plot and $R$ is varied
in the right plot. The computed range and cross-range FWHM values,
$k_{0} \Delta x$ and $k_{0} \Delta y$, respectively, are plotted just
as in Fig.~\ref{fig:epsilon+bandwidth} including the corresponding
least-squares fit to lines.

The results for cross-range resolution with respect to $L/a$ shown in
the left plot of Fig.~\ref{fig:aperture+radius} clearly show an
$O(L/a)$ behavior, which is plotted as a dashed-black curve. The
computed least-squares fit to a line is
$\log(\Delta x) \approx -11.5748 + 0.9741 \log(L/a)$ which numerically
validates the $O(L/a)$ behavior. In contrast, the range resolution
does not vary significantly with $L/a$. The computed least-squares fit
to a line is $\log(\Delta y) \approx -12.0007 - 0.0139 \log(L/a)$ which
quantifies the weak dependence that range resolution has on aperture.

The results for range resolution with respect to $L/R$ shown in the
right plot of Fig.~\ref{fig:aperture+radius} clearly show an $O(L/R)$
behavior, which is plotted as a dashed-black curve. The computed
least-squares fit to a line is
$\log(\Delta y) \approx -12.8645 + 0.9690 \log(L/R)$ which numerically
validates this $O(L/R)$ behavior. In contrast, the cross-range
resolution shows a much weaker dependence on $L/R$. The computed
least-squares fit to a line is
$\log(\Delta x) \approx -7.5134 - 0.0340 \log(L/R)$ which quantifies
this weak dependence on $L/R$.

\begin{figure}[htb]
  \centering
  \includegraphics[width=0.42\linewidth]{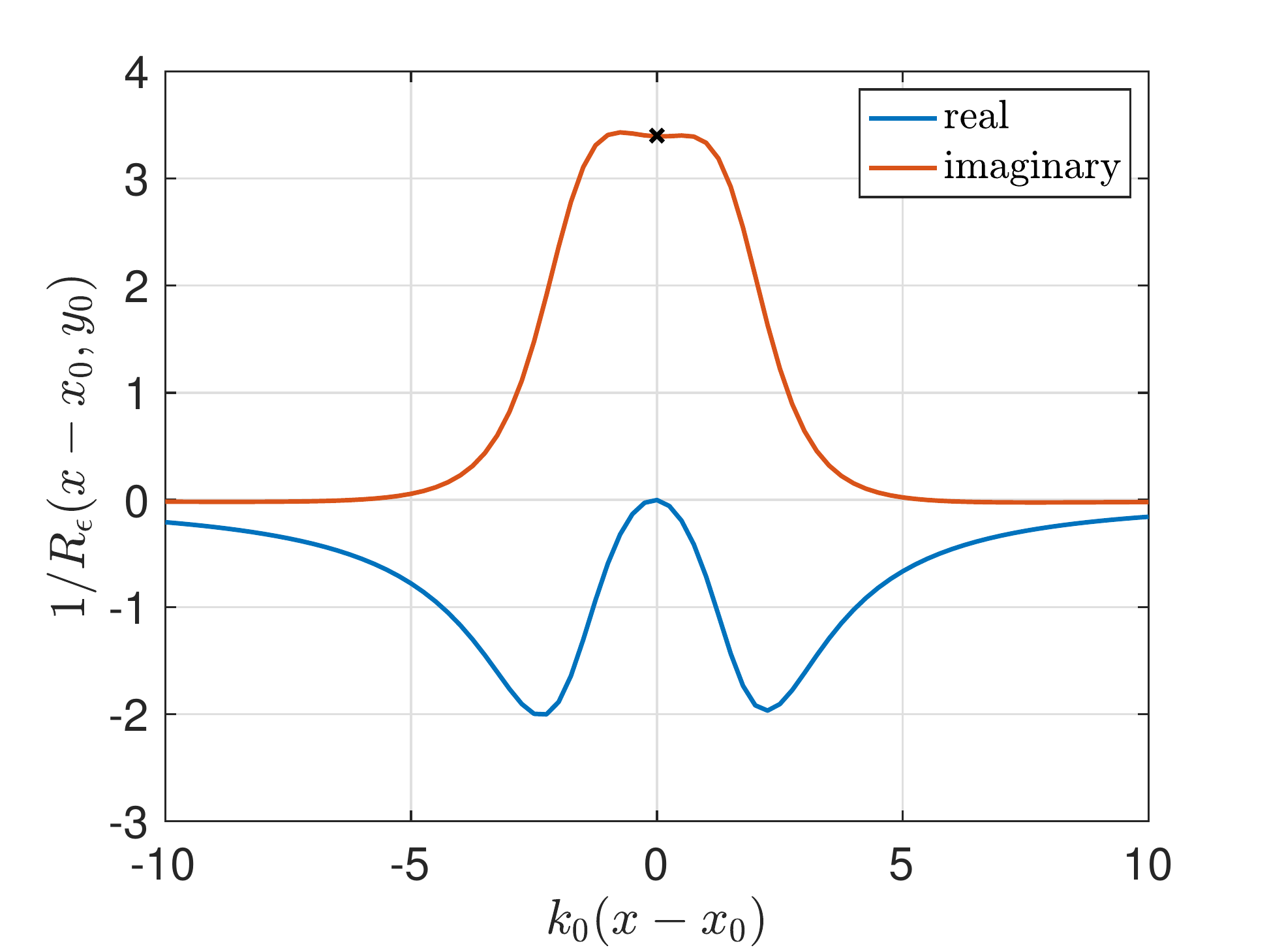}
  \includegraphics[width=0.42\linewidth]{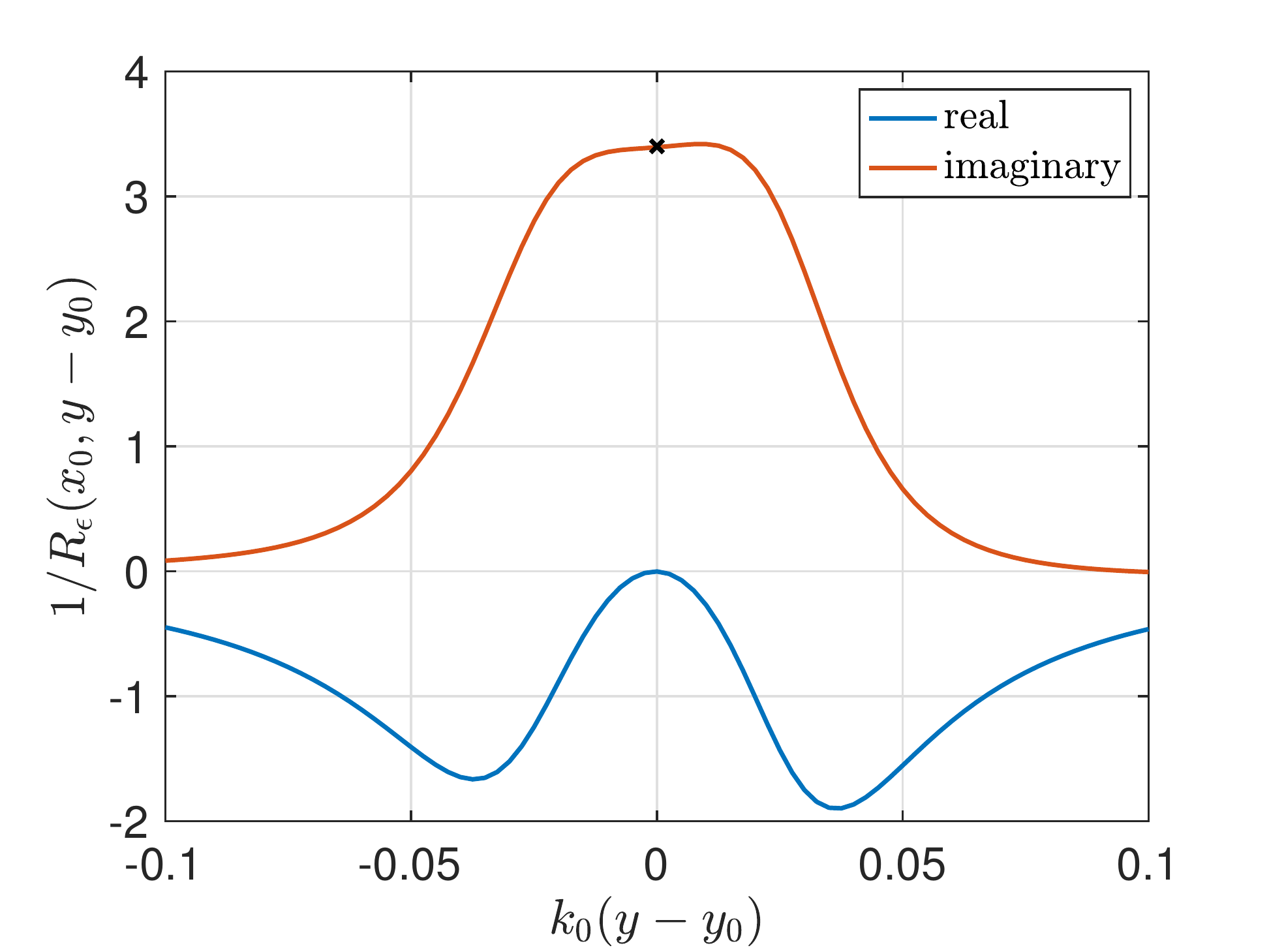}
  \caption{Recovering the complex reflectivity through evaluation of
    $1/R_{\epsilon}(\ypos)$ with $R_{\epsilon}(\ypos)$ given in
    \eqref{eq:R-function} with $\epsilon = 10^{-10}$. Measurement
    noise was added so that $\text{SNR} = 44.1339 \text{dB}$. The left
    plot shows $1/R_{\epsilon}$ on $y = y_{0}$ as a function of
    $x - x_{0}$ and the right plot shows $1/R_{\epsilon}$ on
    $x = x_{0}$ as a function of $y - y_{0}$. The blue curves
    give the real part of $1/R_{\epsilon}(\ypos)$ and the red curves
    give the imaginary part. The black ``$\times$'' symbol gives the
    exact value of the complex reflectivity,
    $\rho_{0} = 3.4 \mathrm{i}$.}
  \label{fig:complex-reflectivity}
\end{figure}

We now show results from evaluating $1/R_{\epsilon}(\ypos)$ with
$R_{\epsilon}(\ypos)$ given in \eqref{eq:R-function}. These results
use the same parameter values as those used in
Fig.~\ref{fig:point-target}. When plotting $1/R_{\epsilon}$, there is
no local behavior to indicate the location of the target. For this
reason these images do not provide useful information about the
location of targets.  However, when we evaluate
$1/R_{\epsilon}(\ypos)$ in a region near the target location, we are
able to recover the complex reflectivity. In
Fig.~\ref{fig:complex-reflectivity} we show the real and imaginary
parts of $1/R_{\epsilon}(x-x_{0},y_{0})$ in the left plot and of
$1/R_{\epsilon}(x_{0},y-y_{0})$ in the right plot.  In both plots the
actual value $\rho_{0} = 3.4 \mathrm{i}$ is plotted as a black
``$\times$'' symbol. These plots show that when the location of the
point target is known, evaluating $1/R_{\epsilon}(\ypos)$ at the
recovered target location provides a method for recovering the complex
reflectivity. At the target location, we find
$1/R_{\epsilon}(\ypos_{0}) = -1.3059 \times 10^{-3} + 3.3928
\mathrm{i}$ which demonstrates a very high accuracy in recovering the
complex reflectivity.  Provided that the target location is reasonably
accurate, the user-defined parameter $\epsilon$ can be used to
regularize these results to enable stable recovery of the complex
reflectivity.

\begin{figure}[htb]
  \centering
  \includegraphics[width=0.45\linewidth]{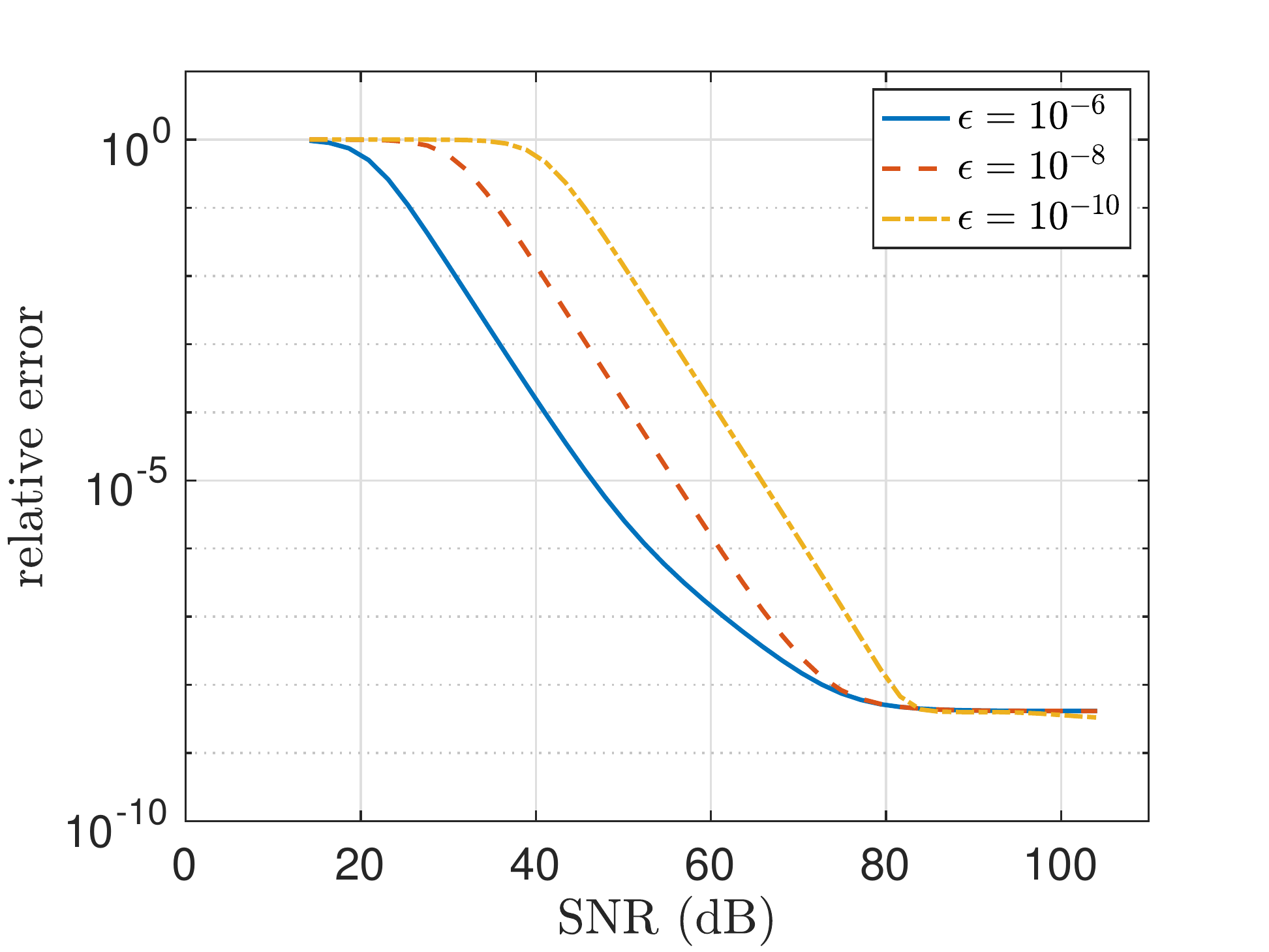}
  \caption{Relative error in the recovery of the complex reflectivity
    through evaluation of $1/R_{\epsilon}(\ypos_{0})$ with
    $\epsilon = 10^{-8}$ for a point target as a function of SNR
    (dB). All parameters are the same as those used in
    Fig.~\ref{fig:point-target}.}
  \label{fig:error}
\end{figure}

In both Theorems \ref{thm:resolution} and \ref{thm:reflectivity}, it
is assumed that the SNR is sufficiently high that one can separate the
signal subspace from the noise subspace. To investigate the effect of
SNR on the recovery of the complex reflectivity, we evaluate
$1/R_{\epsilon}(\ypos_{0})$ for different SNR values and compute the
relative error,
$E_{\text{rel}} = | \rho_{0} -
1/R_{\epsilon}(\ypos_{0})|/|\rho_{0}|$. These relative error results
are shown with $\epsilon = 10^{-6}$ as a solid blue curve,
$\epsilon = 10^{-8}$ as a dashed red curve, and $\epsilon = 10^{-10}$
as a dot-dashed yellow curve in Fig.~\ref{fig:error}. The results in
Fig.~\ref{fig:error} show that sufficiently high SNR is needed to
achieve a high accuracy. Additionally, we observe that larger
$\epsilon$ values achieve higher accuracy for any fixed SNR. This
higher accuracy occurs because $\epsilon$ regularizes
$1/R_{\epsilon}(\ypos)$ thereby stabilizing the recovery of the
complex reflectivity.  The role of SNR on the resolution becomes more
of an issue when imaging multiple targets which we discuss below.

\subsection{Multiple point targets}

We now consider multiple point targets in the imaging region. We set
the origin of the coordinate system to lie at the center of a
$5\, \text{m} \times 5\, \text{m}$ planar imaging region on $z =
0$. The first target is located at
$(x_{1}, y_{1},0) = (0.01\, \text{m}, 0.1\, \text{m},0)$ with complex
reflectivity $\rho_{1} = 3.4 \mathrm{i}$. The second target is located
at $(x_{2}, y_{2}, 0) = (-0.30\, \text{m}, -0.50\, \text{m},0)$ with
complex reflectivity $\rho_{2} = 4.2 \mathrm{i}$. The third target is
located at $(x_{3}, y_{3}, 0) = (-0.50\, \text{m}, 0.50\, \text{m},0)$
with complex reflectivity $\rho_{3} = 3.1 \mathrm{i}$.

\begin{figure}[t]
  \centering
  \begin{subfigure}[t]{0.32\linewidth}
    \includegraphics[width=\linewidth]{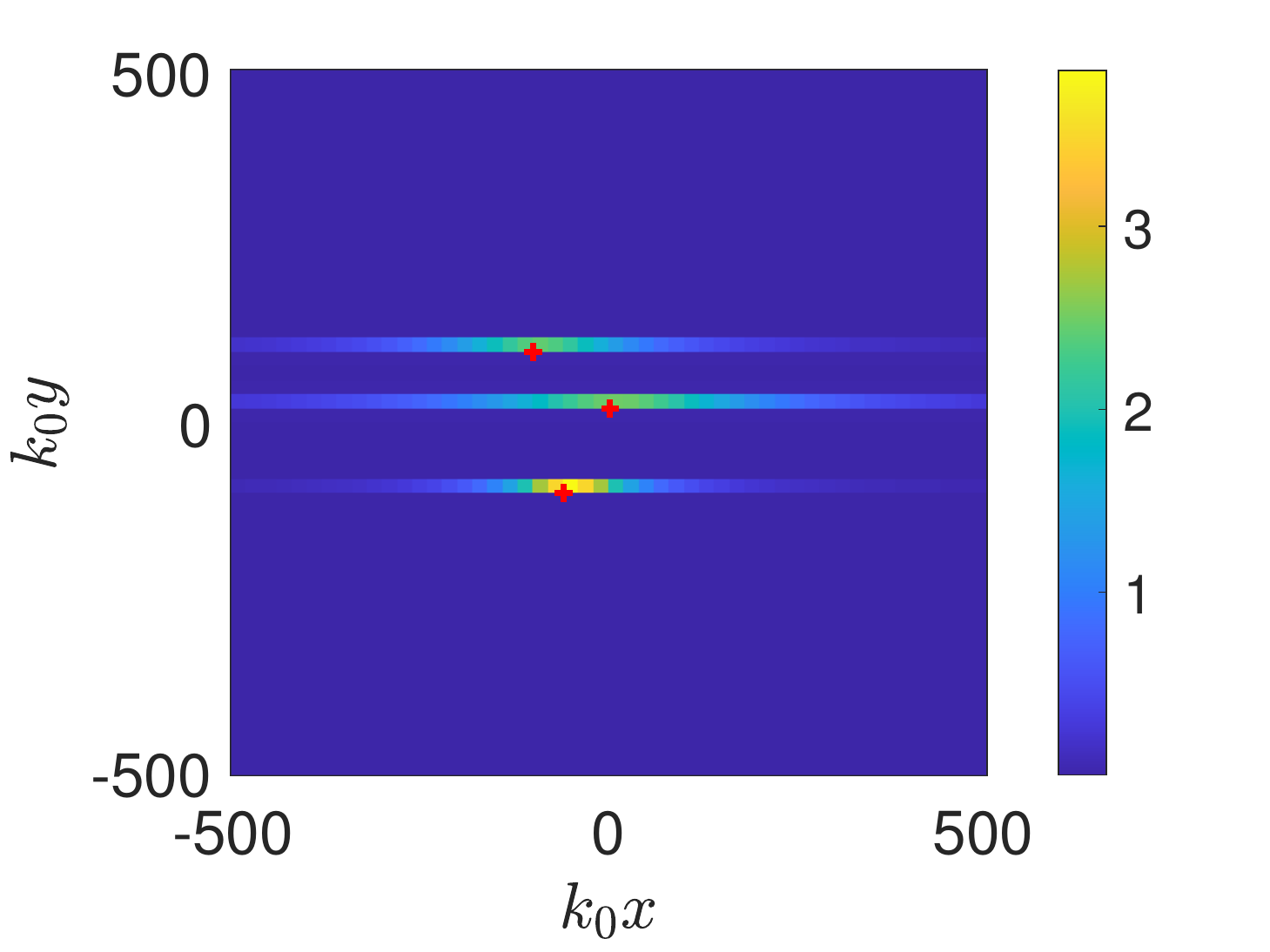}
    \caption{$\epsilon = 10^{-6}$}
  \end{subfigure}
  \begin{subfigure}[t]{0.32\linewidth}
    \includegraphics[width=\linewidth]{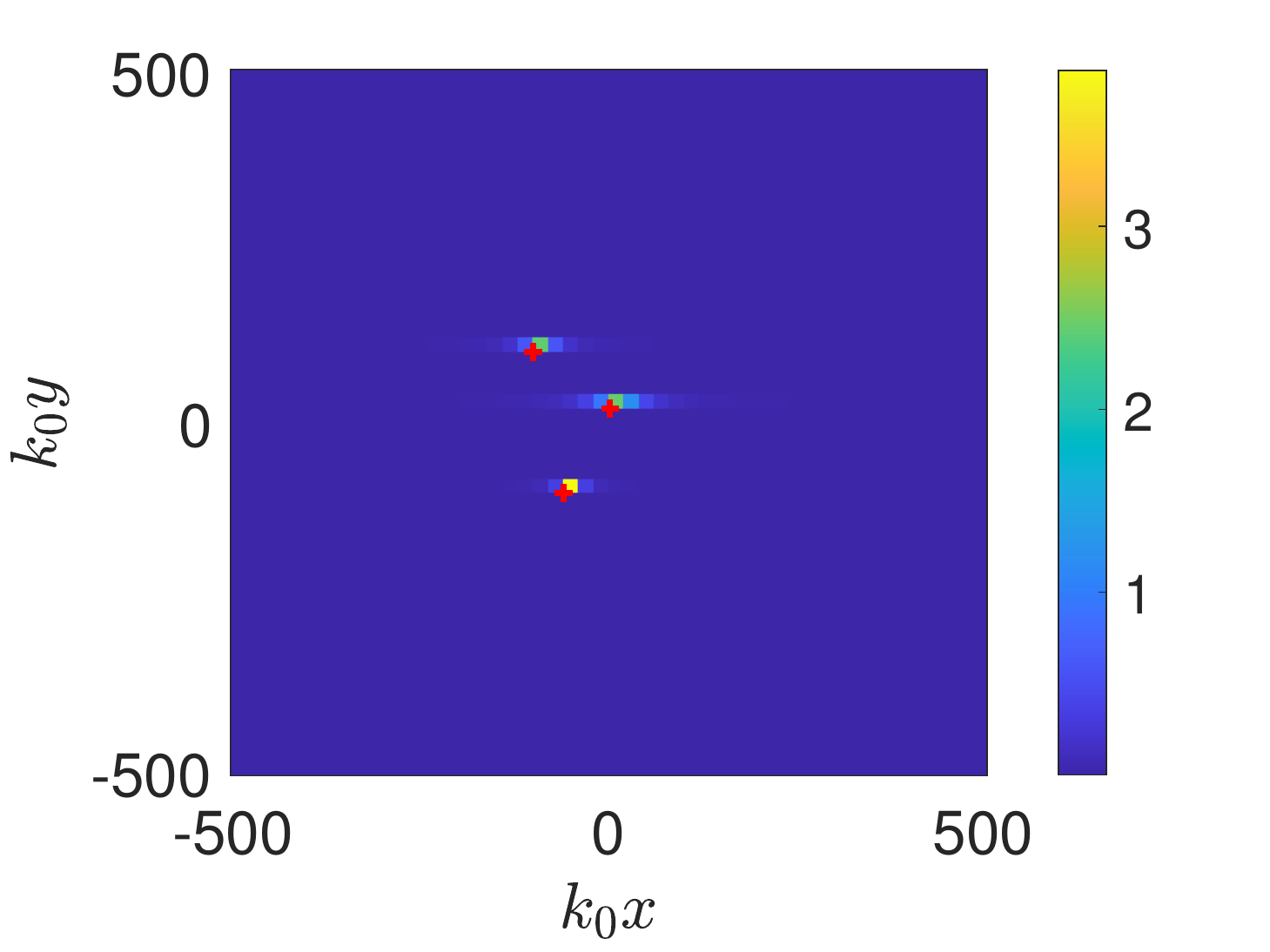}
    \caption{$\epsilon = 10^{-8}$}
  \end{subfigure}
  \begin{subfigure}[t]{0.32\linewidth}
    \includegraphics[width=\linewidth]{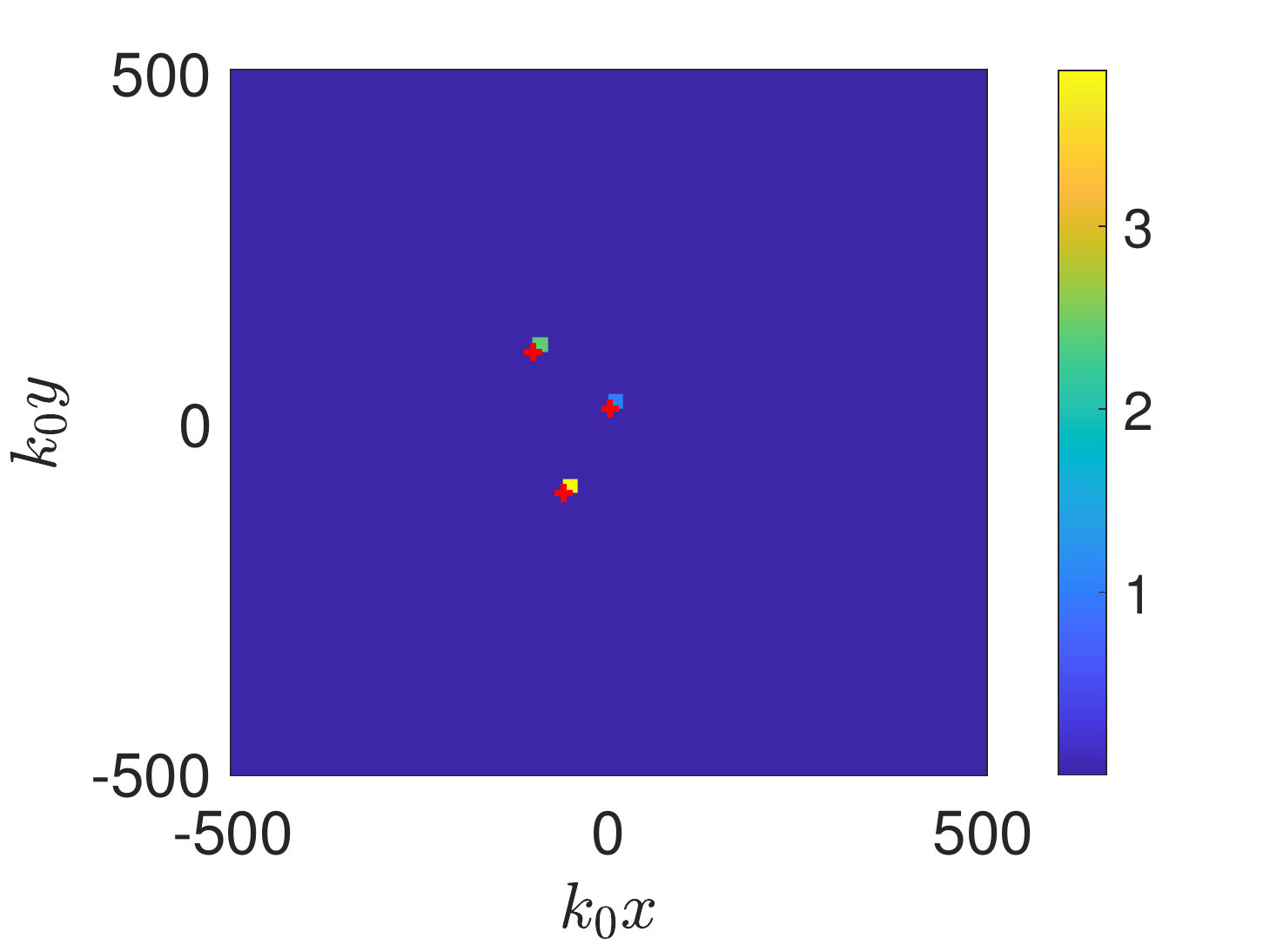}
    \caption{$\epsilon = 10^{-10}$}
  \end{subfigure}
  \caption{Imaging through evaluation of $1/F_{\epsilon}(\ypos)$ with
    $F_{\epsilon}(\ypos)$ given in \eqref{eq:F-function} and
    $\epsilon = 10^{-6}$ (left), $\epsilon = 10^{-8}$ (center), and
    $\epsilon = 10^{-10}$ (right) for 3 point targets with
    $\text{SNR} = 64.1695\, \text{dB}$. The first target is located at
    $(x_{1}, y_{1},0) = (0.01\, \text{m}, 0.1\, \text{m},0)$ with
    complex reflectivity $\rho_{1} = 3.4 \mathrm{i}$. The second
    target is located at
    $(x_{2}, y_{2}, 0) = (-0.3\, \text{m}, -0.5\, \text{m},0)$ with
    complex reflectivity $\rho_{2} = 4.2 \mathrm{i}$. The third target
    is located at
    $(x_{3}, y_{3}, 0) = (-0.5\, \text{m}, 0.5\, \text{m},0)$ with
    complex reflectivity $\rho_{3} = 3.1 \mathrm{i}$.}
  \label{fig:3targets}
\end{figure}

In Fig.~\ref{fig:3targets} we show the image produced through
evaluation of $1/F_{\epsilon}(\ypos)$ with $F_{\epsilon}(\ypos)$ given
in \eqref{eq:F-function} with $\epsilon = 10^{-6}$ (left),
$\epsilon = 10^{-8}$ (center), and $\epsilon = 10^{-10}$ (right). The
imaging region is discretized using a $51 \times 51$ equi-spaced mesh
corresponding to approximately a $10\, \text{cm}$
meshwidth. Measurement noise was added so that
$\text{SNR} = 64.1695\, \text{dB}$. We see that the value of
$\epsilon$ affects the overall resolution of the three targets,
especially with respect to cross-range since $L/a < L/R$.  With
$\epsilon = 10^{-10}$, the image produced through evaluation of
$1/F_{\epsilon}(\ypos)$ clearly indicates the locations of the three
targets. Even though we do not have direct interpretation of the
magnitude of the peaks in this plot, we do find that
$\| 1/F_{\epsilon}(\ypos) \|_{\infty} = 3.8641$ for
$\epsilon = 10^{-10}$, which is close to the values of $|\rho_{1}|$,
$|\rho_{2}|$, and $|\rho_{3}|$.

\begin{figure}[t]
  \centering
  \includegraphics[width=0.32\linewidth]{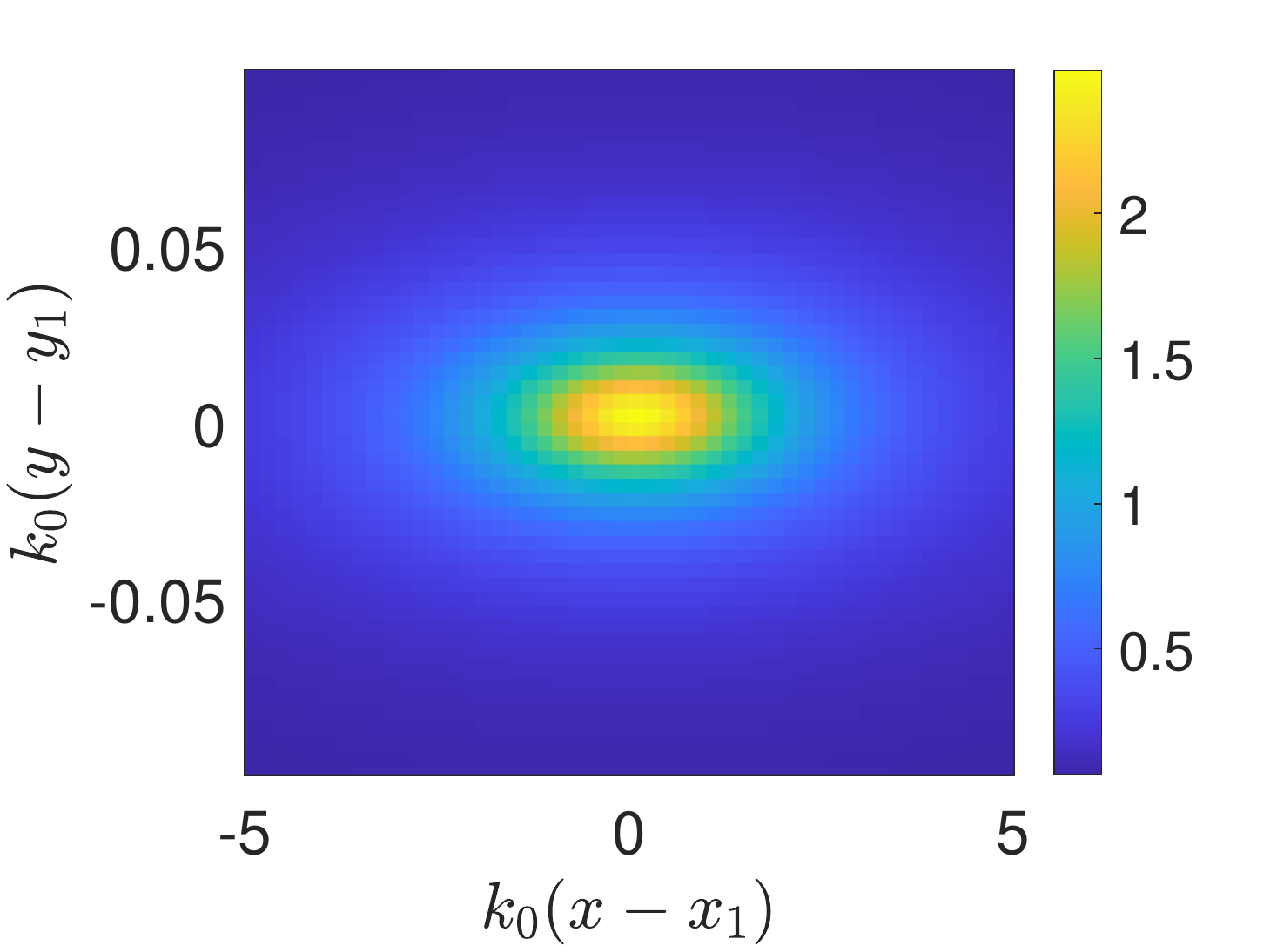}
  \includegraphics[width=0.32\linewidth]{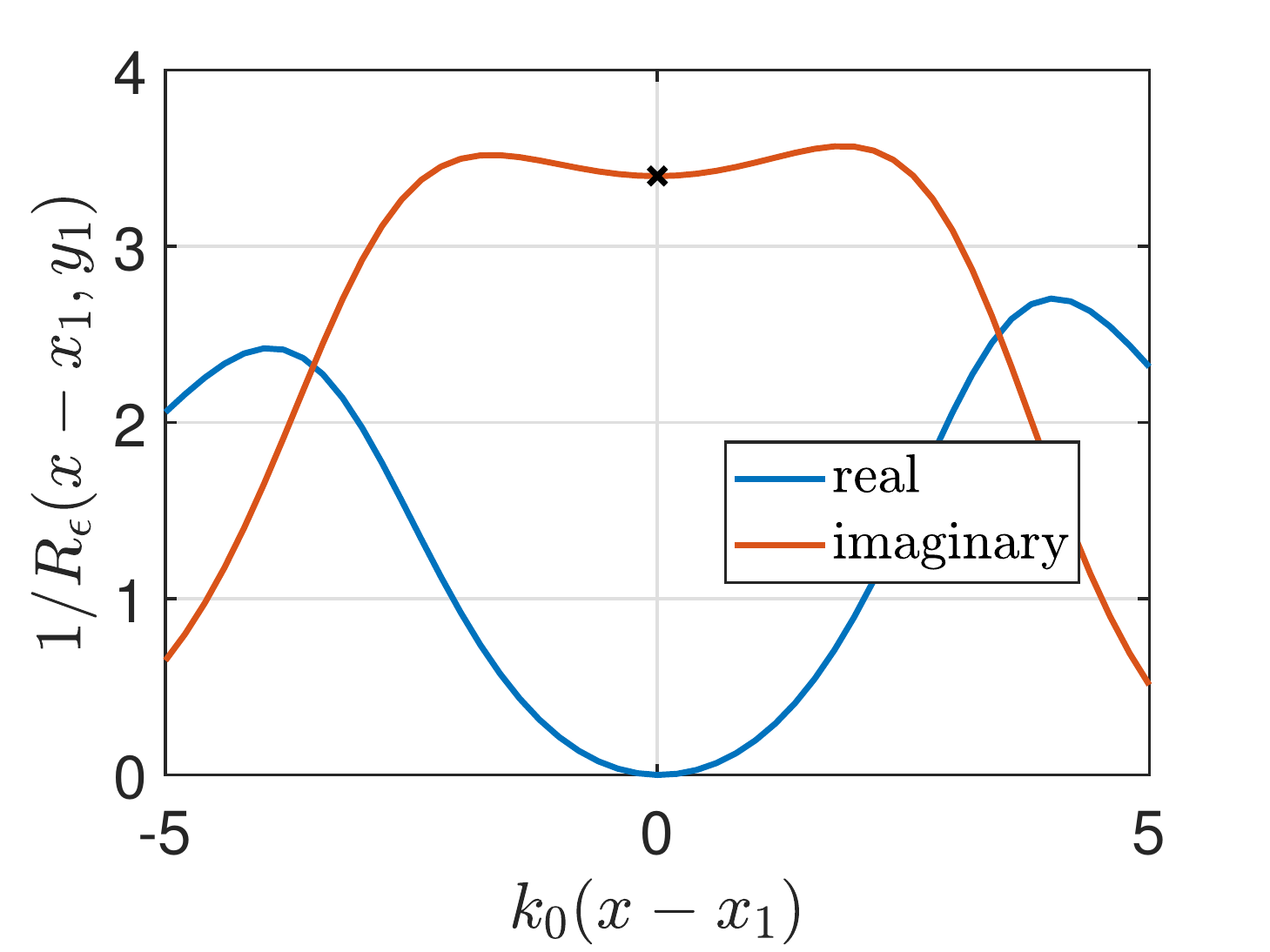}
  \includegraphics[width=0.32\linewidth]{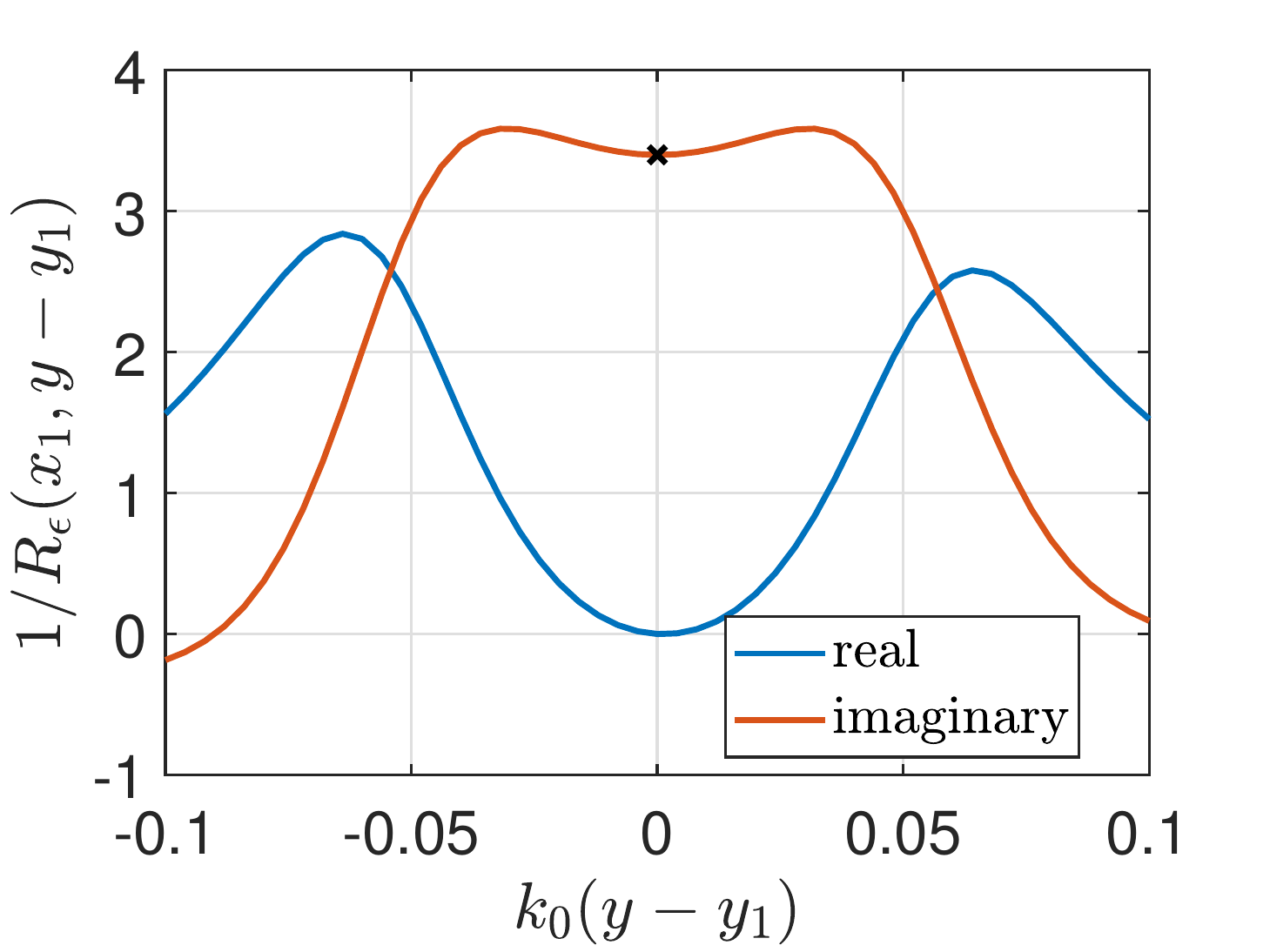}
  \caption{(left) Imaging using $1/F_{\epsilon}(\ypos)$ in a region
    about target 1 located at
    $(x_{1}, y_{1},0) = (0.01\, \text{m}, 0.1\, \text{m},0)$ with
    complex reflectivity $\rho_{1} = 3.4 \mathrm{i}$. Recovering the
    complex reflectivity through evaluation of
    $1/R_{\epsilon}(x-x_{1},y_{1})$ (center) and
    $1/R_{\epsilon}(x_{1},y-y_{1})$ (right). The blue curves give the
    real part of $1/R_{\epsilon}(\ypos)$ and the red curves give the
    imaginary part. The black ``$\times$'' symbol gives the exact
    value of the complex reflectivity, $\rho_{1} = 3.4 \mathrm{i}$.}
  \label{fig:3targets-target1}
  \includegraphics[width=0.32\linewidth]{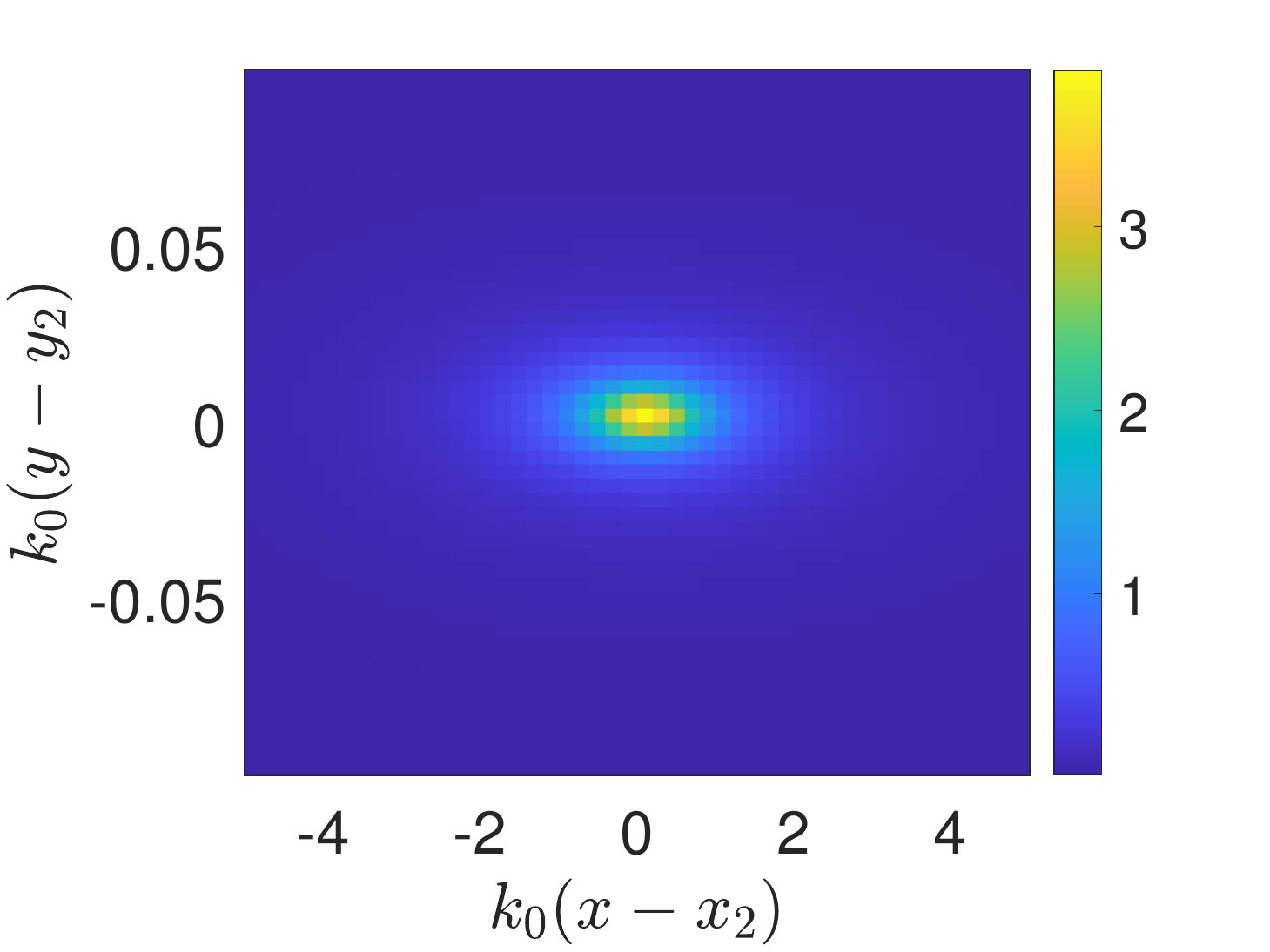}
  \includegraphics[width=0.32\linewidth]{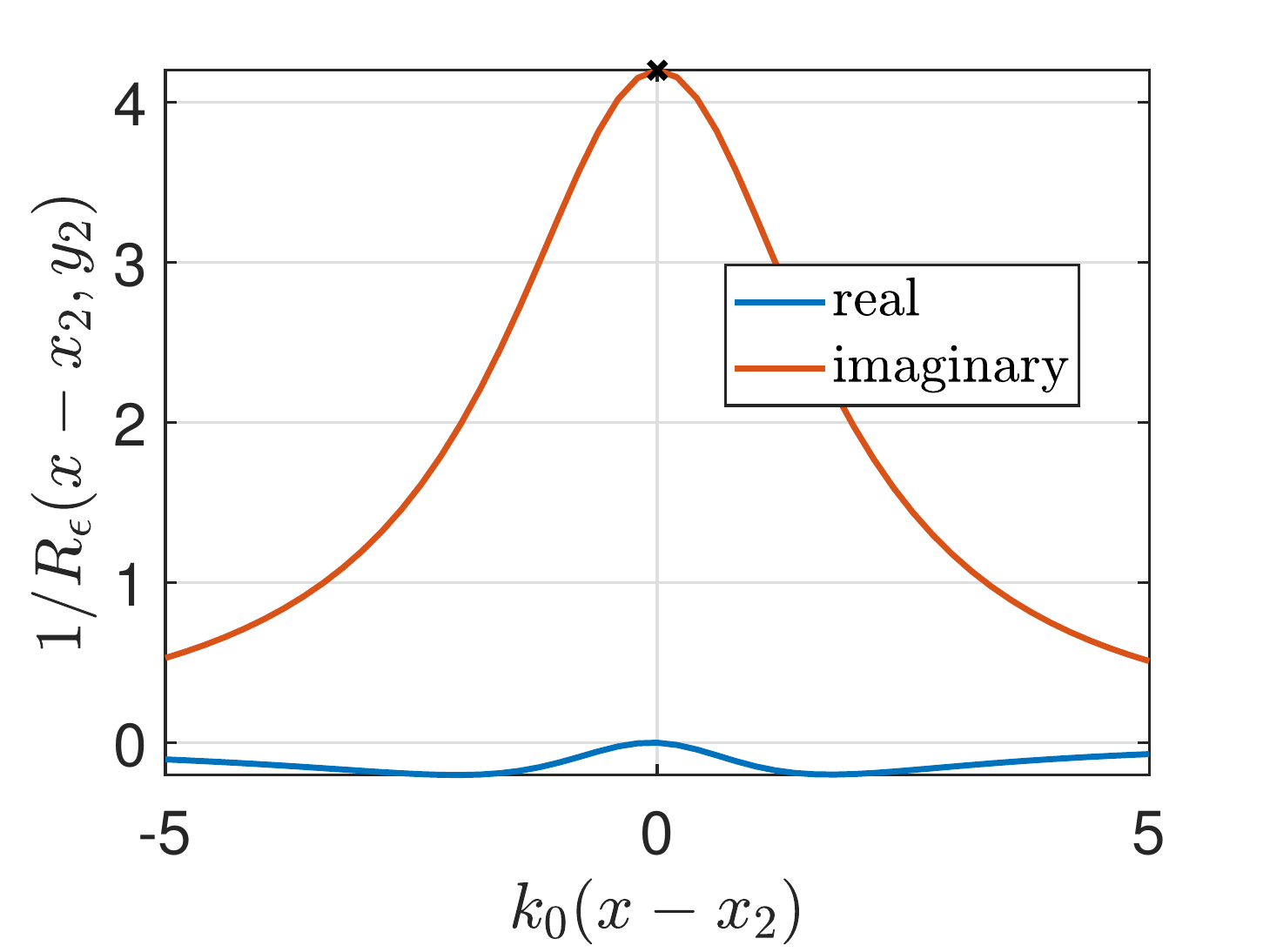}
  \includegraphics[width=0.32\linewidth]{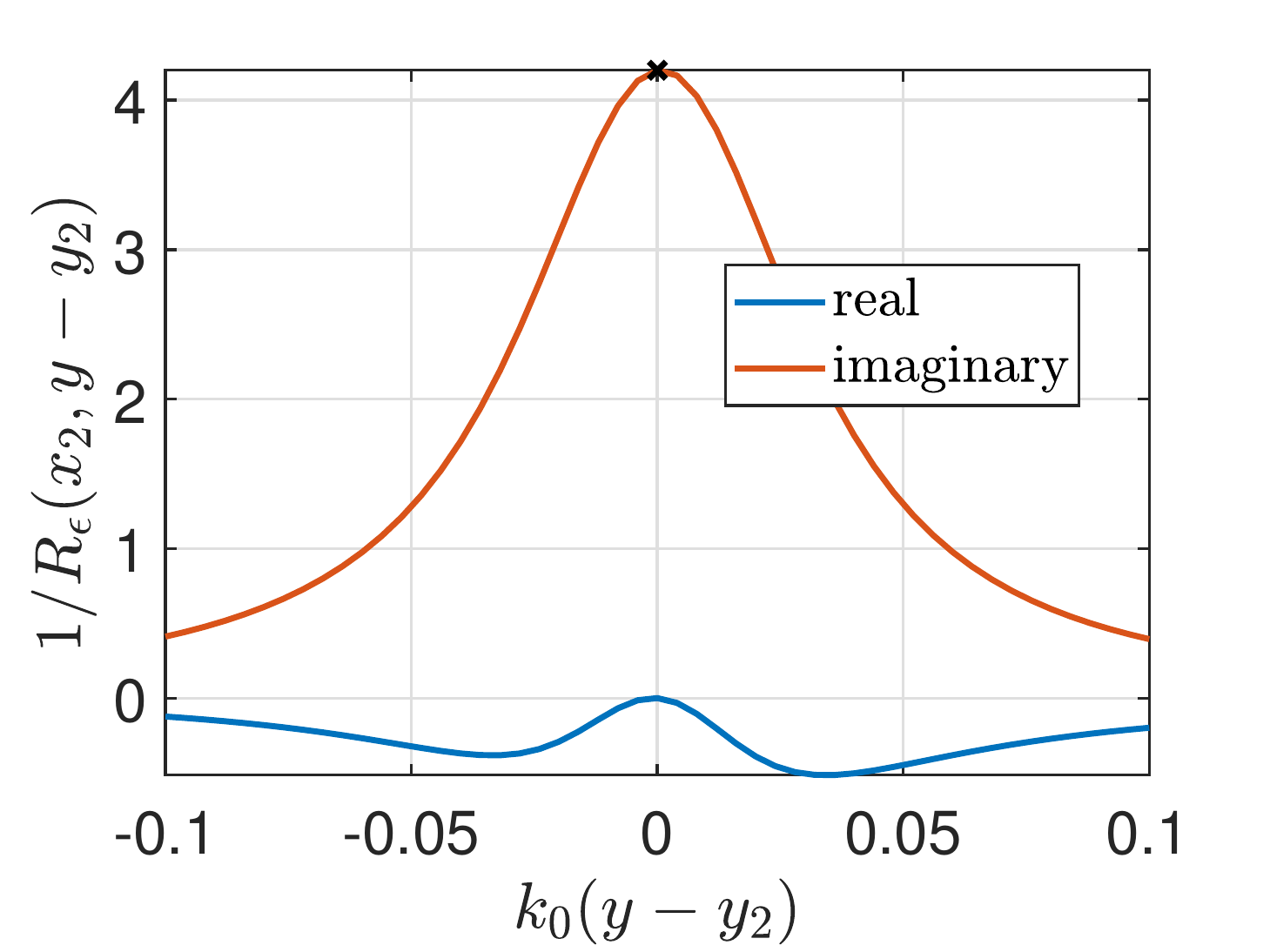}
  \caption{The same as Fig.~\ref{fig:3targets-target1}, except for
    target 2 located at
    $(x_{2}, y_{2}, 0) = (-0.30\, \text{m}, -0.50\, \text{m},0)$ with
    complex reflectivity $\rho_{2} = 4.2 \mathrm{i}$.}
  \label{fig:3targets-target2}
  \includegraphics[width=0.32\linewidth]{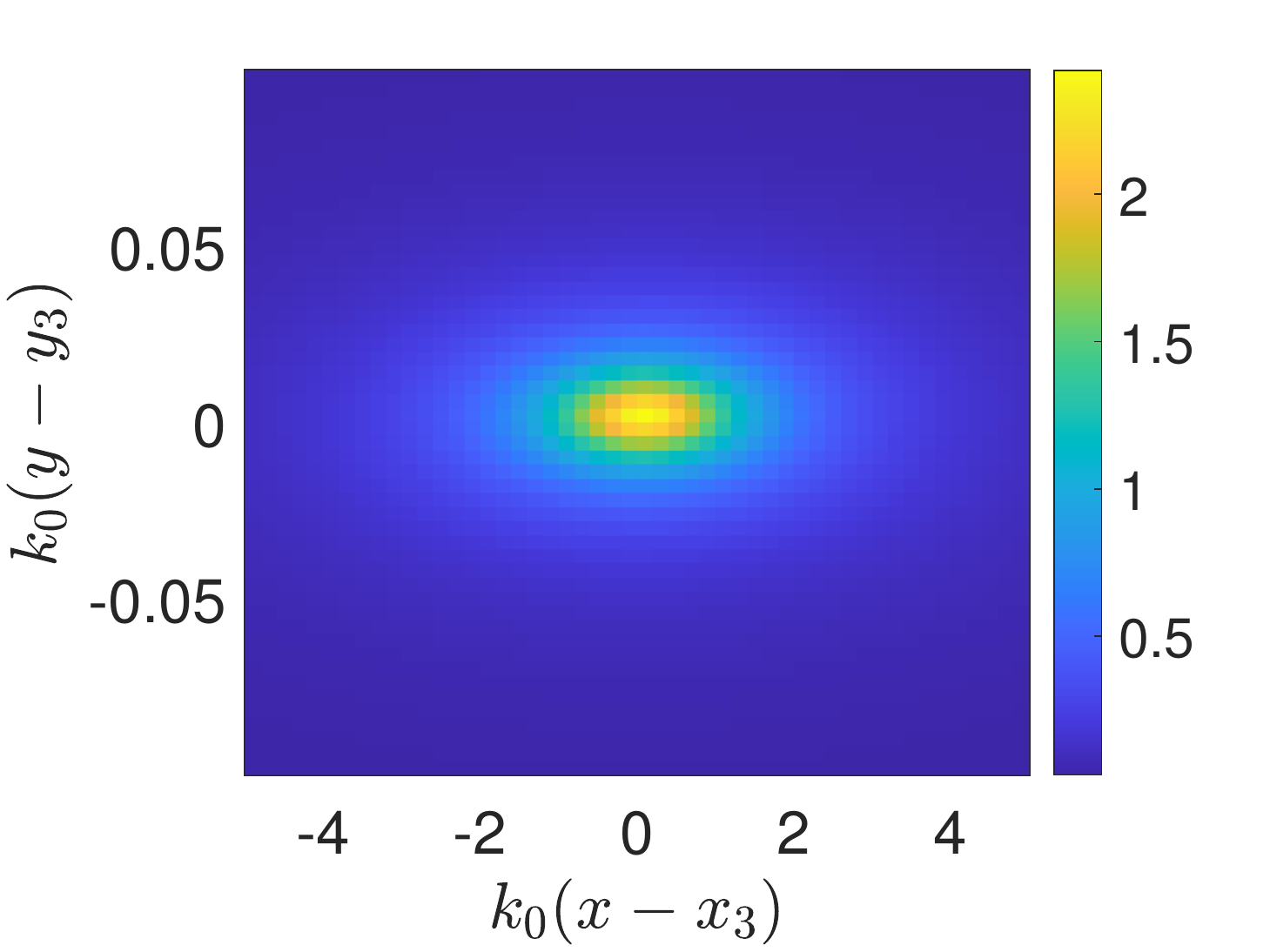}
  \includegraphics[width=0.32\linewidth]{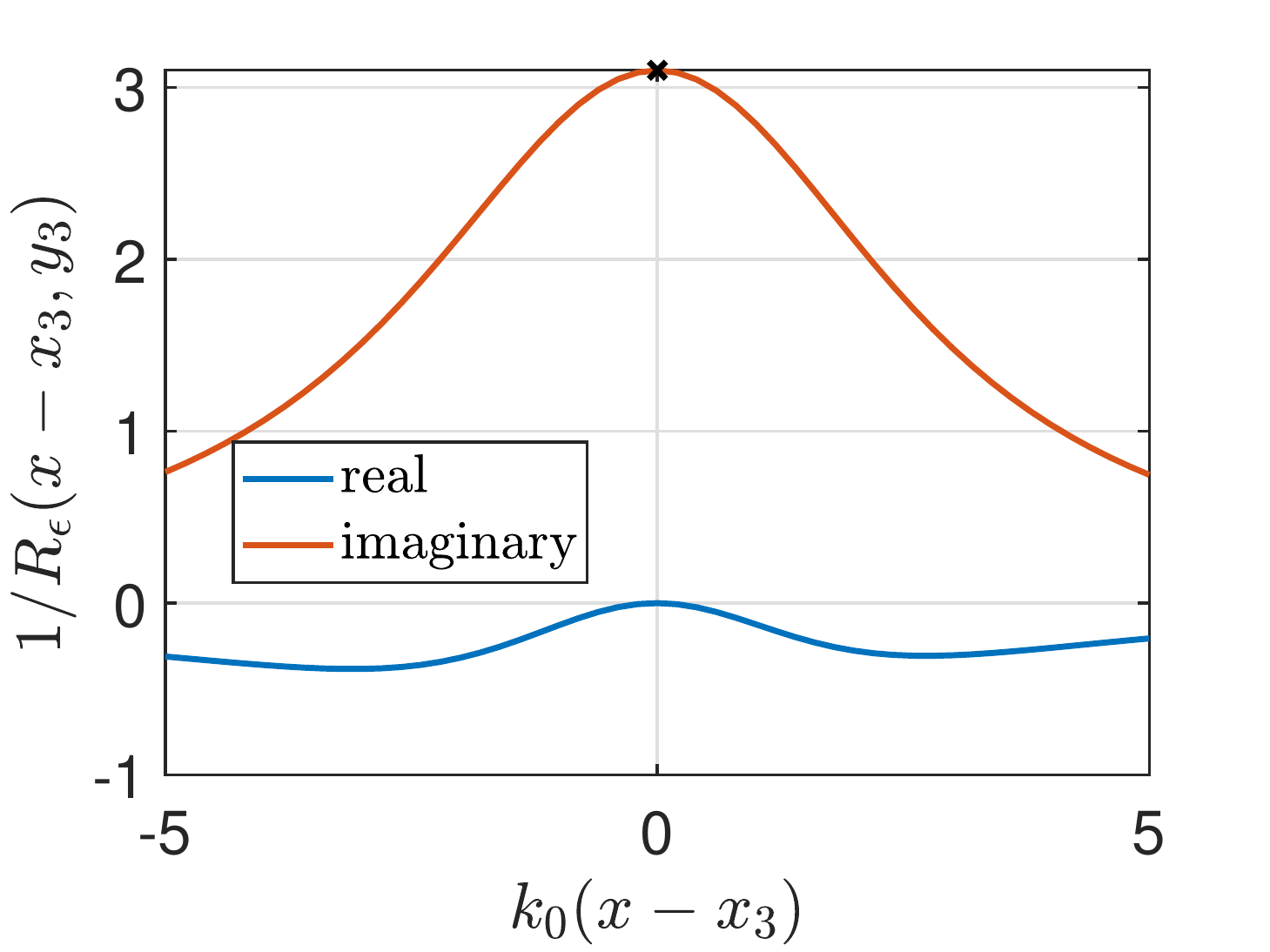}
  \includegraphics[width=0.32\linewidth]{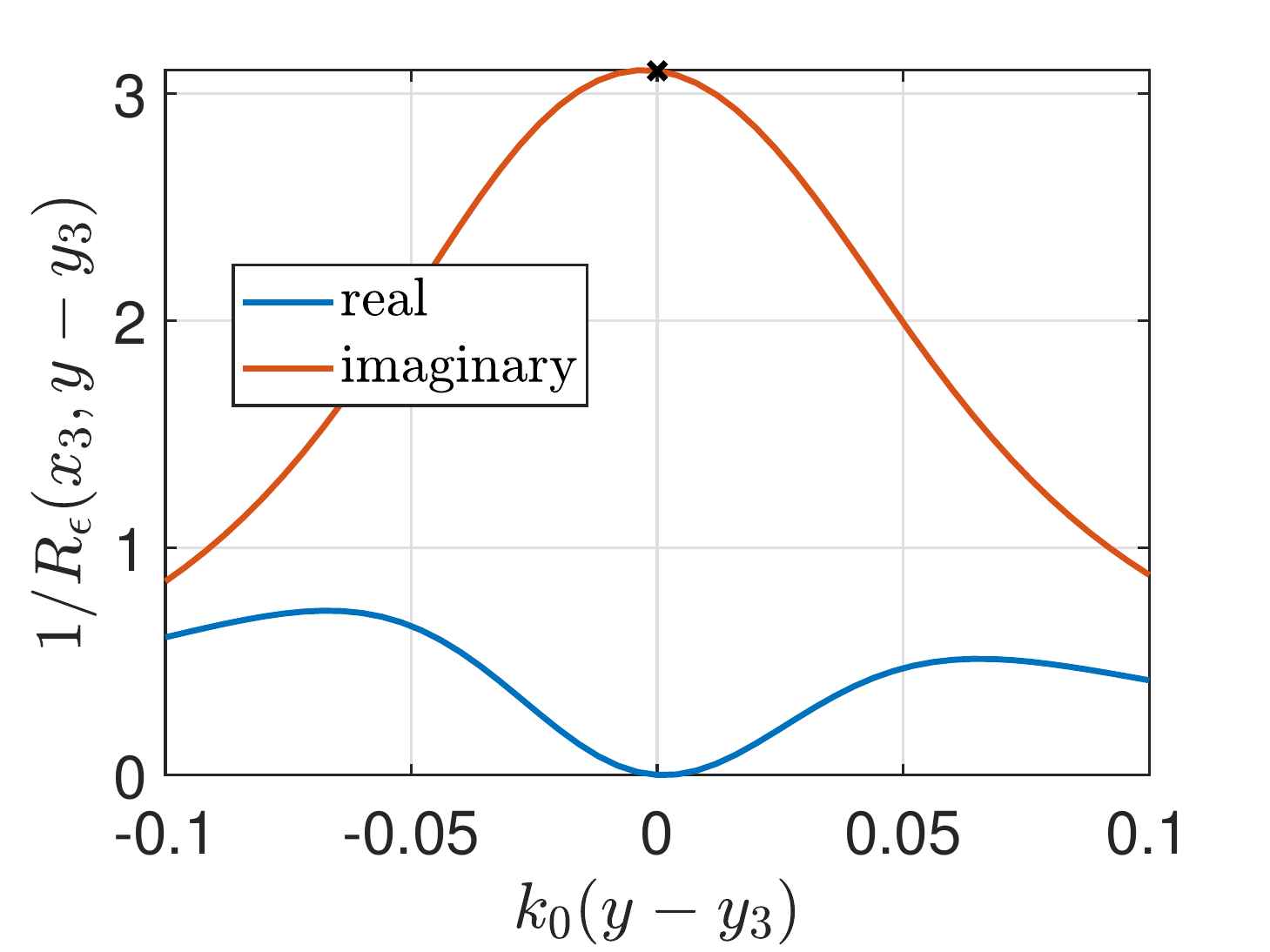}
  \caption{The same as Fig.~\ref{fig:3targets-target1}, except for
    target 3 located at
    $(x_{3}, y_{3},0) (-0.50\, \text{m}, 0.50\, \text{m},0)$ with
    complex reflectivity $\rho_{3} = 3.1 \mathrm{i}$.}
  \label{fig:3targets-target3}
\end{figure}

Using Fig.~\ref{fig:3targets} to determine regions about each of the
target locations, we then evaluate $1/F_{\epsilon}(\ypos)$ using the
same measurements to obtain the location more precisely and
$1/R_{\epsilon}(\ypos)$ to recover the complex reflectivities. In
particular, we plotted the evaluation of $1/F_{\epsilon}(\ypos)$ in a
window of size $10\, k_{0}$ in cross-range ($x$) and $0.2\, k_{0}$ in
range ($y$) about each target using a $51 \times 51$ mesh.  The
results of doing this are shown in Fig.~\ref{fig:3targets-target1} for
target 1, Fig.~\ref{fig:3targets-target2} for target 2, and
Fig.~\ref{fig:3targets-target3} for target 3. The left plots in
Figs.~\ref{fig:3targets-target1} -- \ref{fig:3targets-target3} show
results of evaluating $1/F_{\epsilon}(\ypos)$ in regions about the
respective targets. The center plots in
Figs.~\ref{fig:3targets-target1} -- \ref{fig:3targets-target3} shows
results of evaluating $1/R_{\epsilon}(\ypos)$ on $y = y_{1}$, $y_{2}$,
and $y_{3}$, respectively, and the left plots in
Figs.~\ref{fig:3targets-target1} -- \ref{fig:3targets-target3} shows
results of evaluating $1/R_{\epsilon}(\ypos)$ on $x = x_{1}$, $x_{2}$,
and $x_{3}$, respectively.

When plotting $1/F_{\epsilon}(\ypos)$ in a small region about each
target location, we are readily able to determine the target location
corresponding to where this function attains its local maximum thereby
demonstrating the high-resolution of this imaging method. With the
location of each target determined using these results, we then
evaluate $1/R_{\epsilon}(\ypos)$ in these regions which allows for
recovery of the complex reflectivity of each target. For these
results, these evaluations yielded
$\rho_{1} = 3.4 \mathrm{i} \approx 4.0096 \times 10^{-4} + 3.3990
\mathrm{i}$,
$\rho_{2} = 4.2 \mathrm{i} \approx 1.4427 \times 10^{-4} + 4.2000
\mathrm{i}$, and
$\rho_{3} = 3.1 \mathrm{i} \approx 1.3969 \times 10^{-4} + 3.0998
\mathrm{i}$ thereby demonstrating the high quantitative accuracy
achieved using this method.

\begin{figure}[htb]
  \centering
  \includegraphics[width=0.4\linewidth]{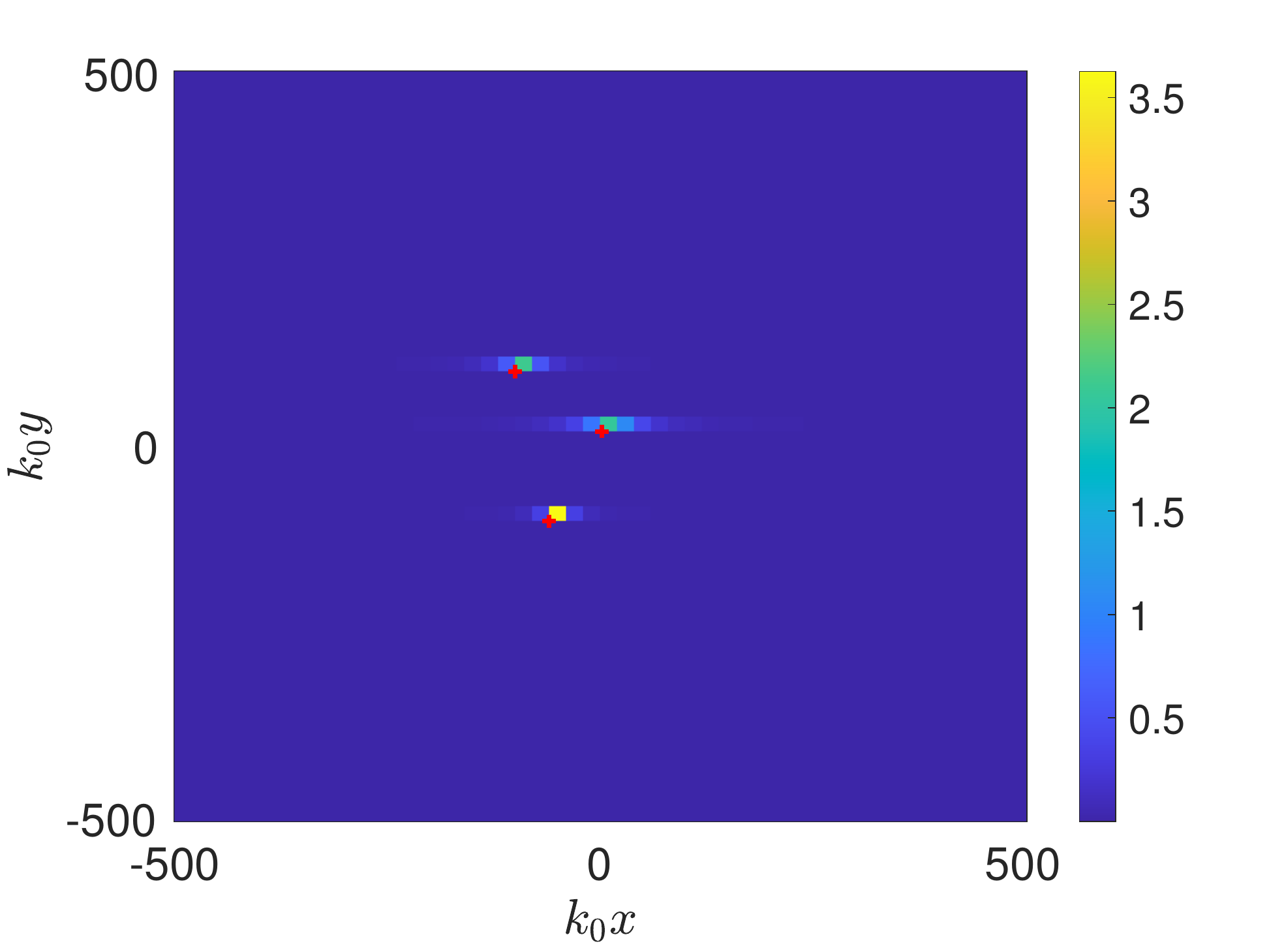}
  \includegraphics[width=0.38\linewidth]{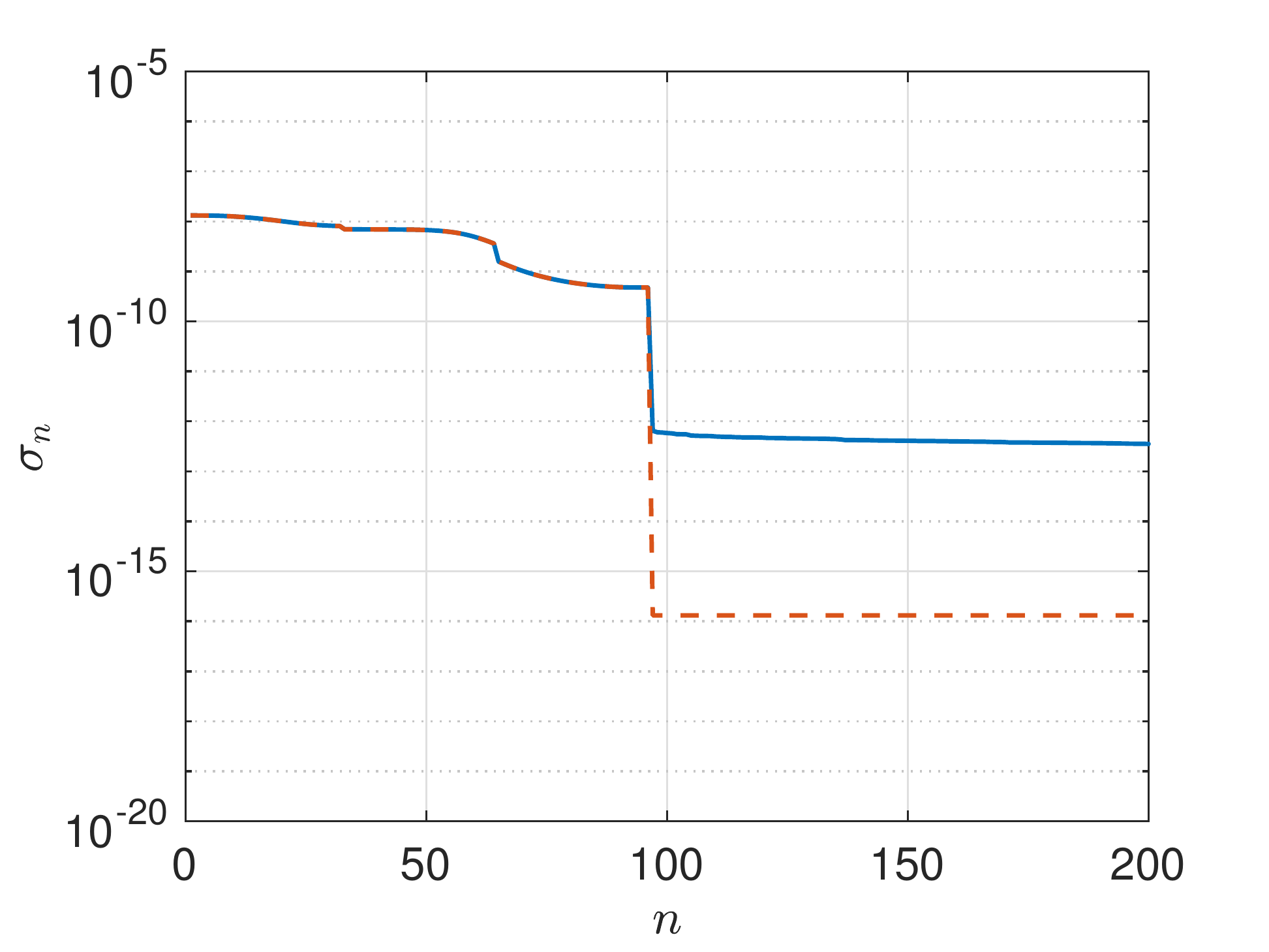}
  \vspace{2pt}\\
  \includegraphics[width=0.4\linewidth]{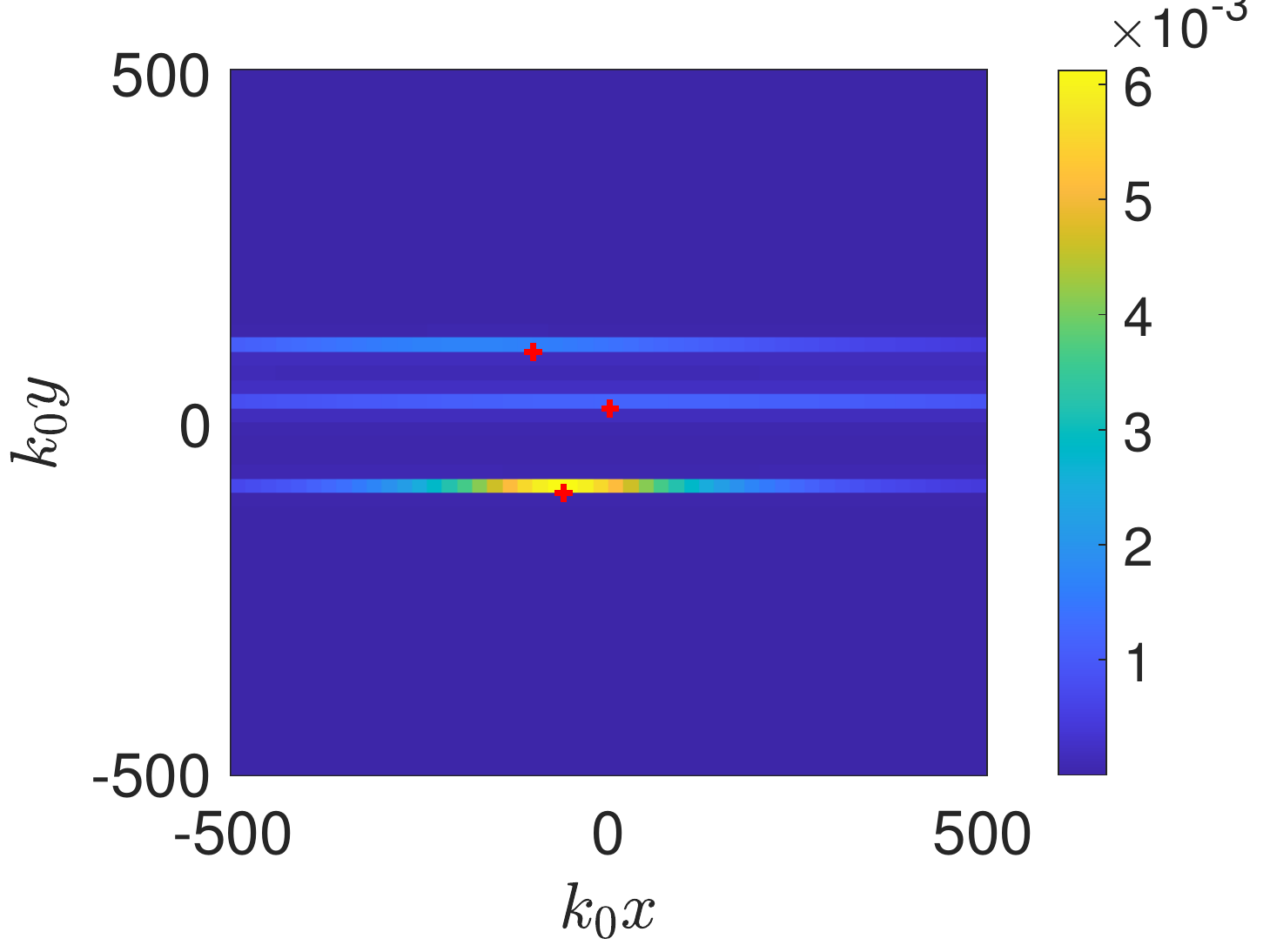}
  \includegraphics[width=0.38\linewidth]{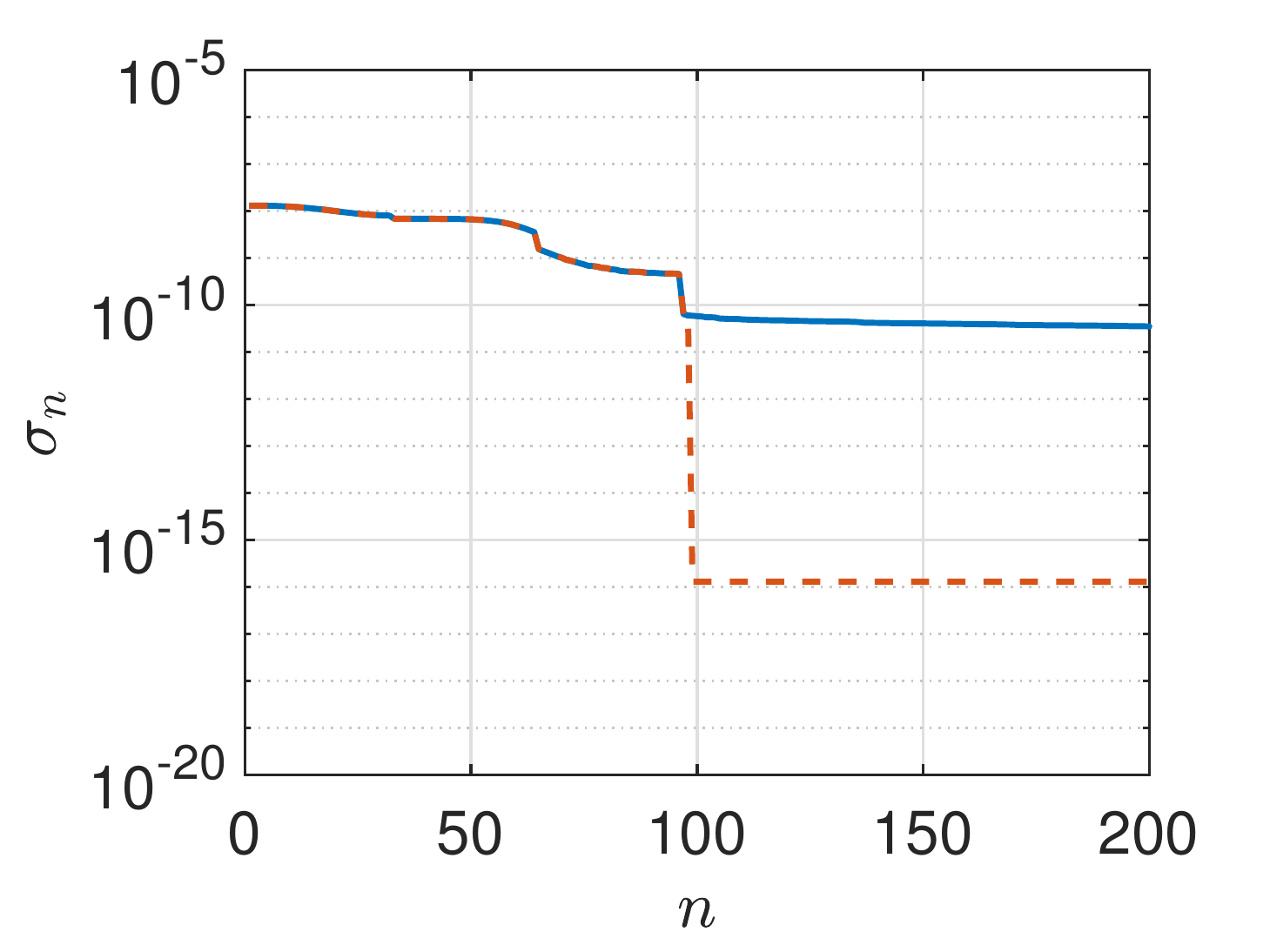}
  \caption{Image produced through evaluation of
    $1/F_{\epsilon}(\ypos)$ with $F_{\epsilon}(\ypos)$ given in
    \eqref{eq:F-function} and $\epsilon = 10^{-8}$ for
    $\text{SNR} = 44.1695\, \text{dB}$ (top left), and the
    corresponding singular value spectrum (top right). The lower left
    and right plots are for $\text{SNR} = 24.1695\, \text{dB}$.  All
    parameter values are the same as those used in
    Fig.~\ref{fig:3targets}.}
  \label{fig:3targets-SNR}
\end{figure}

We showed the effect of SNR on the recovery of the complex
reflectivity for a single point target in Fig.~\ref{fig:error}.  To
study the effect of SNR of imaging multiple point targets, we consider
images for the same scenario produced through evaluation of
$1/F_{\epsilon}(\ypos)$ with $F_{\epsilon}(\ypos)$ given in
\eqref{eq:F-function} and $\epsilon = 10^{-8}$ for
$\text{SNR} = 44.1695\, \text{dB}$ (top row of
Fig.~\ref{fig:3targets-SNR}) and $24.1695\, \text{dB}$ (bottom row of
Fig.~\ref{fig:3targets-SNR}). Except for the SNR values, all parameter
values are the same as those in Fig.~\ref{fig:3targets}(b). Included
with each of those images are the corresponding singular value spectra
for $D_{\text{Prony}}$ given in \eqref{eq:D-Prony} plotted as blue
curves. The dashed red curves show the thresholded singular values in
which $\sigma_{n}$ is replaced with $\epsilon \sigma_{1}$ for all
$\sigma_{n} < 0.01 \sigma_{1}$.

In Fig.~\ref{fig:3targets-SNR}, the image for
$\text{SNR} = 44.1695\, \text{dB}$ (top left) shows the three targets
clearly, but the image for $\text{SNR} = 24.1695\, \text{dB}$ (bottom
left) has much poorer resolution, especially in range. The singular
value spectra in Fig.~\ref{fig:3targets-SNR} (top right and bottom
left) provide valuable insight into the difference between these two
images. A signal subspace method is predicated on the assumption that
one can distinguish the signal and noise subspaces from one
another. With regards to the singular value spectrum, one would like
to have a large ``gap'' between the singular values corresponding to
the signal subspace and those corresponding to the noise subspace. We
observe in Fig.~\ref{fig:3targets-SNR} that when the SNR decreases, so
does the gap separating the singular values for the signal subspace
from those of the noise subspace. Even though the thresholding
criterion of replacing $\sigma_{n}$ with $\epsilon \sigma_{1}$ when
$\sigma_{n} < 0.01 \sigma_{1}$ captures the location of the gap
correctly for both SNR values, a consequence of the narrowing of this
gap is a loss of image resolution. Because $L/a < L/R$, we see a more
severe loss in resolution in range than in cross-range. These results
demonstrate that this imaging method requires a sufficiently high SNR
to be effective and accurate.

\subsection{Imaging in random media}

We consider perturbations in travel times resulting from wave
propagation in random media. Assuming an inhomogeneous velocity
profile of the form
\begin{equation}
  \label{eq:random_wave_speed}
  \frac{1}{c^2(\xpos)} = \frac{1}{c_0^2} \left( 1 +\sigma\mu
    \left(\frac{\xpos}{\ell} \right) \right),
\end{equation}
we approximate the Green's function between points $\xpos$ and $\ypos$
at frequency $\omega$ by
\begin{equation}
  \label{eq:random_green_func}
  G(\xpos,\ypos; \omega)=G_0(\xpos,\ypos; \omega)
  \exp{\left[i \omega \nu (\xpos,\ypos)\right]}, 
\end{equation}
with $\nu (\xpos,\ypos)$ denoting the random travel time function
\begin{equation}
  \nu (\xpos,\ypos) = \frac{\sigma |\xpos-\ypos|}{2c_0} \int_{0}^{1}
  \mu \left( \frac{\ypos}{\ell} + (\xpos-\ypos)\frac{s}{\ell}
  \right) \mathrm{d}s.
\label{eq:rtt}
\end{equation}
Here $c_0$ denotes the average propagation speed, assumed constant,
$\ell$ is the correlation length and $\sigma$ is the strength of the
fluctuations. The stationary random process $\mu(\cdot)$ has mean zero
and normalized auto-correlation function
$R(|\xpos-\xpos'|)= \mathbb{E}(\mu(\xpos)\mu(\xpos'))$, so that
$R(0)=1$, and $\int_0^\infty R(r) r^2 dr < \infty$.  In
\eqref{eq:random_green_func}, $G_0$ denotes the Green's function in
the homogeneous medium with propagation speed $c_0$.

\begin{figure}[t]
  \centering
  \begin{subfigure}[t]{0.3\linewidth}
    \includegraphics[width=\linewidth]{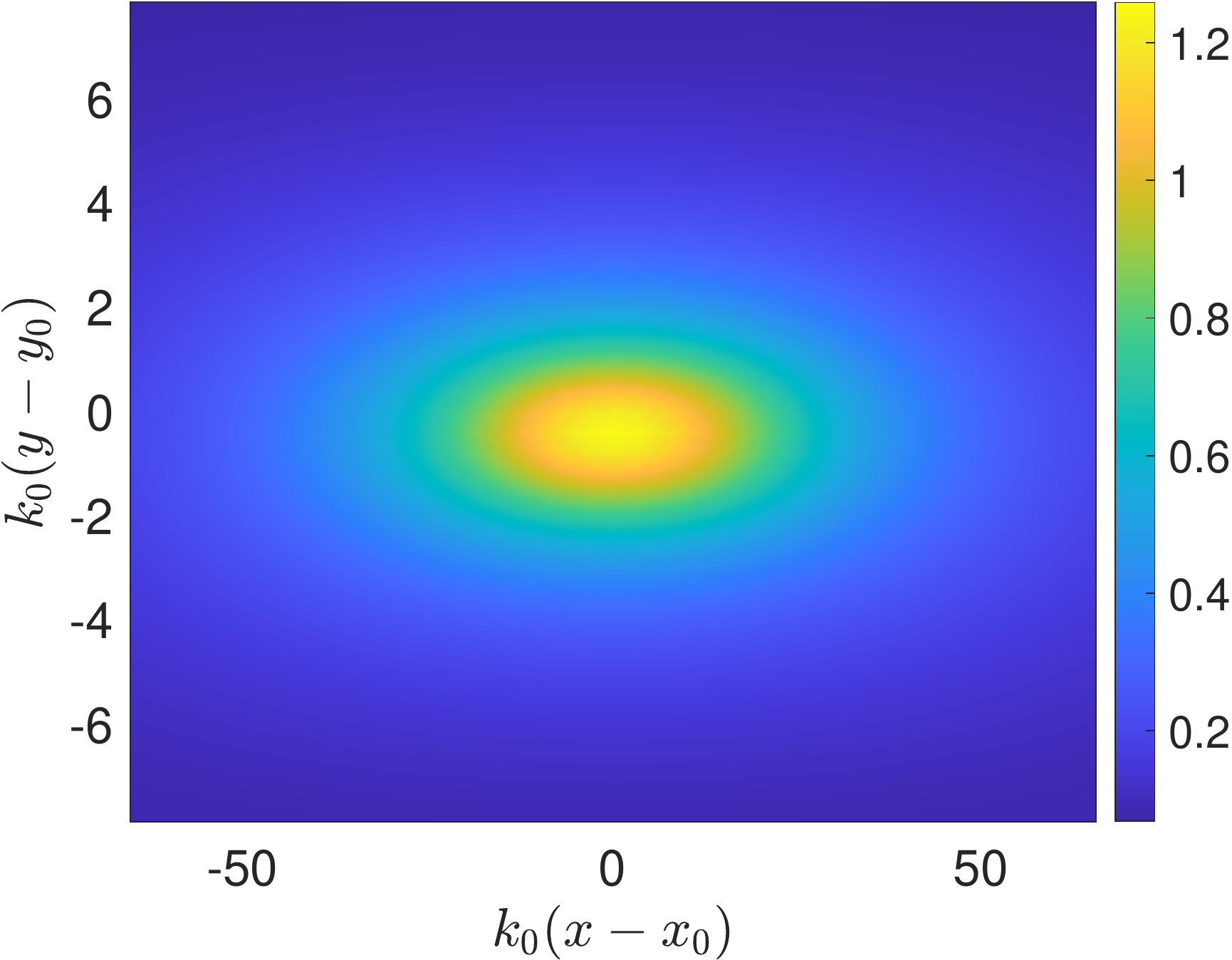} 
    \caption{$\tilde{\sigma}/\sqrt{\epsilon} = 0$}
  \end{subfigure}
  \begin{subfigure}[t]{0.3\linewidth}
    \includegraphics[width=\linewidth]{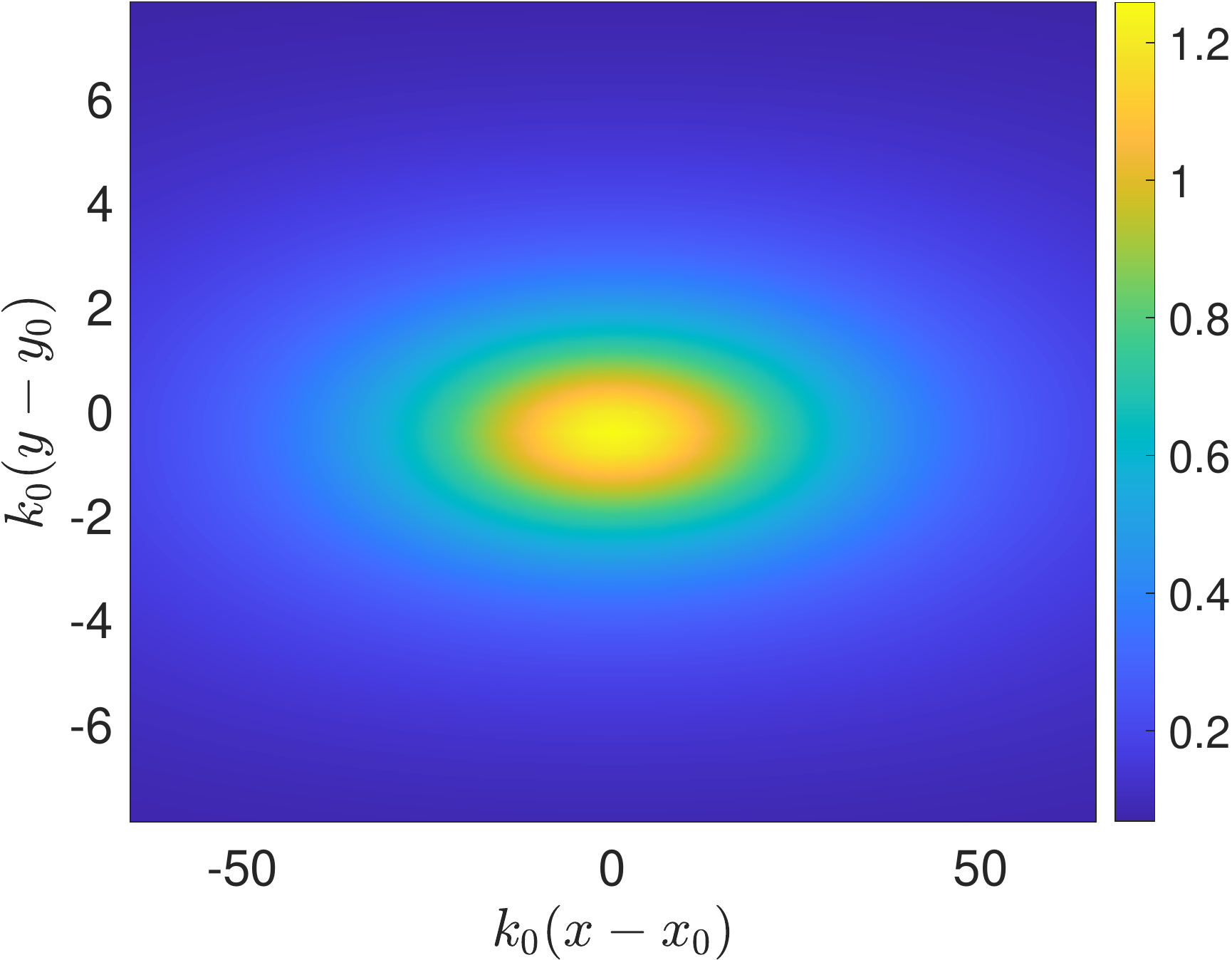} 
    \caption{$\tilde{\sigma}/\sqrt{\epsilon} = \epsilon$}
  \end{subfigure}
  % %
  % \begin{subfigure}[t]{0.4\linewidth}
  %   \includegraphics[width=\linewidth]{./figures_rtt/image_sigmaEqepsilon} 
  %   \caption{$\tilde{\sigma}/\sqrt{\epsilon} = \sqrt{\epsilon}$}
  % \end{subfigure}
  % %
  \begin{subfigure}[t]{0.3\linewidth}
    \includegraphics[width=\linewidth]{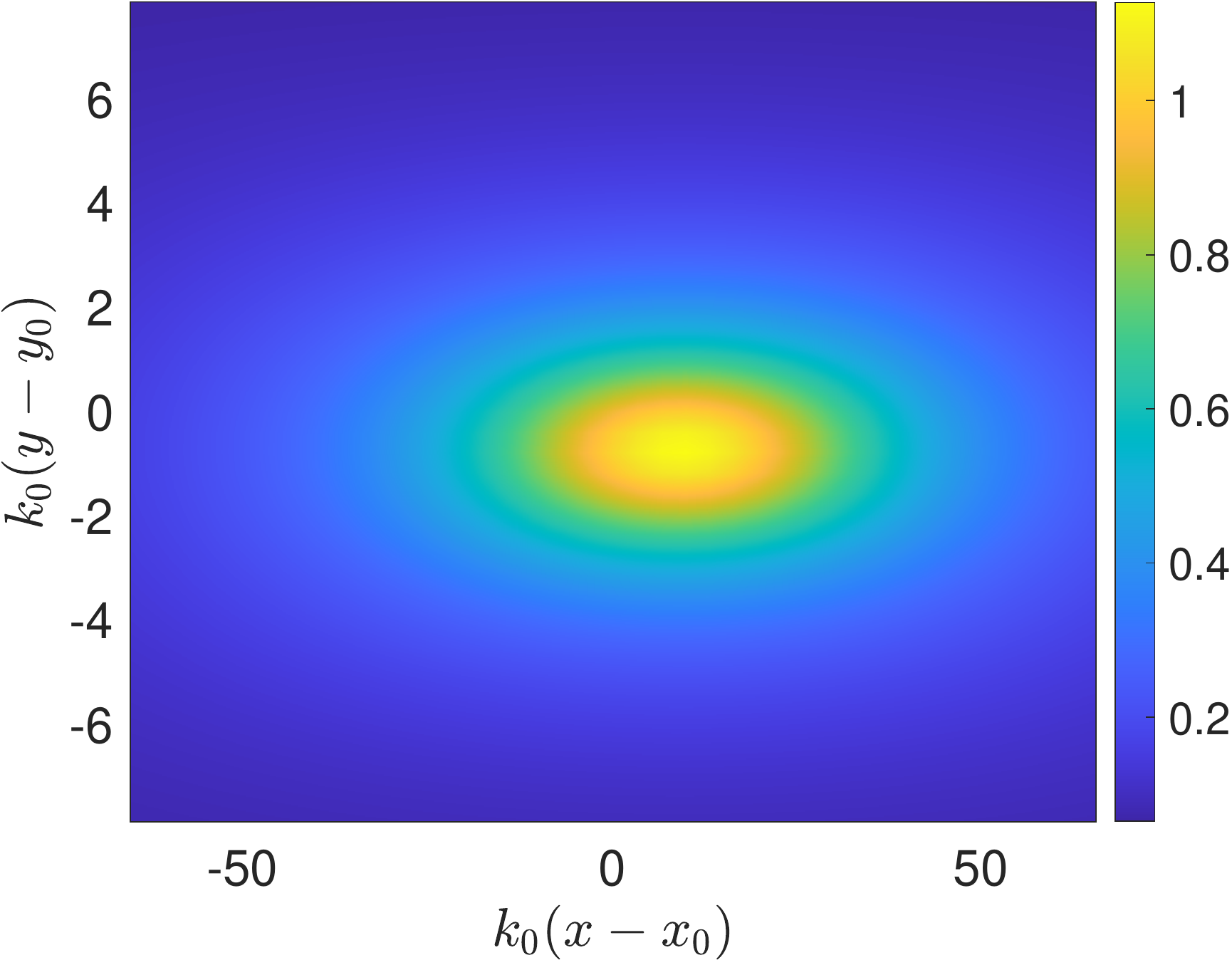} 
    \caption{$\tilde{\sigma}/\sqrt{\epsilon} = 1$}
  \end{subfigure}
  \caption{Imaging through evaluation of $1/F_{\epsilon}(\ypos)$ with
    $F_{\epsilon}(\ypos)$ given in \eqref{eq:F-function} and
    $\epsilon = 0.02$ for a point target located at range 
    $L=100 \ell$ in a random medium with $\ell=100 \lambda$. 
    The strength of the fluctuations in the random medium is such that: (a)
    $\tilde{\sigma}/\sqrt{\epsilon} = 0$, (b)
    $\tilde{\sigma}/\sqrt{\epsilon} = \epsilon$, and (c)
    $\tilde{\sigma}/\sqrt{\epsilon} = 1$.}
  \label{fig:iid-rtt}
\end{figure}

The random travel time model provides an approximation of the Green's
function in the high-frequency regime in random media with weak
fluctuations $\sigma \ll 1$ and large correlation lengths $\ell$
compared to the wavelength $\lambda$. The propagation distance $L$ is
assumed to be large with respect to the correlation length,
$L \gg \ell$, so that the scattering induced by the random medium
perturbations has an order one effect on the phase of the Green's
function.  This is true when \cite{BGPT-11}
\begin{equation}
\label{cond}
\frac{\sigma^2L^3}{\ell^3}\ll\frac{\lambda^2}{\sigma^2\ell L} \sim 1\, ,
\end{equation}
Following \cite{MNPT-16} we introduce the dimensionless parameter 
\begin{eqnarray}
\label{eq:sigma0}
\sigma_0= \lambda/\sqrt{\ell L},
\end{eqnarray}
and scale the fluctuations of the random medium so that 
\begin{eqnarray}
\label{eq:sigmat}
\tilde{\sigma} = \frac{\sigma}{\sigma_0} 
\end{eqnarray}
is order one according to \eqref{cond}.  

In contrast to the previous results, we consider here a ``flat''
geometry for which $H = 0$. The propagation distance is $L=100 \ell$
and the correlation length in the random medium is $\ell=100
\lambda$. The synthetic array aperture is $a=24 \ell$ and the
bandwidth parameter $\beta=0.5$.

In Fig.~\ref{fig:iid-rtt}, we show images of a single point target
formed through evaluation of $1/F_{\epsilon}(\ypos)$ for a single
realization of the random medium with different values of
$\tilde{\sigma}$. The magnitude of the complex reflectivity of the
target is $|\rho_{0}| = 1.2584$. For Fig.~\ref{fig:iid-rtt}(a),
$\tilde{\sigma} = 0$ corresponding to a homogeneous medium. For
Figs.~\ref{fig:iid-rtt}(b) and (c), the values of of $\tilde{\sigma}$
are set so that $\tilde{\sigma}/\sqrt{\epsilon} = \epsilon$ and $1$,
respectively.  As predicted by Theorem \ref{thm:random}, the image
with $\tilde{\sigma}/\sqrt{\epsilon} = \epsilon \ll 1$ is stable and
qualitatively and quantitatively similar to the one obtained for the
homogeneous medium.  For $\tilde{\sigma}/\sqrt{\epsilon} = 1$ the
image is not focused on the true target location, the resolution is
decreased, and the reconstructed absolute value of the reflectivity is
less accurate.

\begin{figure}[t]
  \centering
  \begin{subfigure}[t]{0.4\linewidth}
    \includegraphics[width=\linewidth]{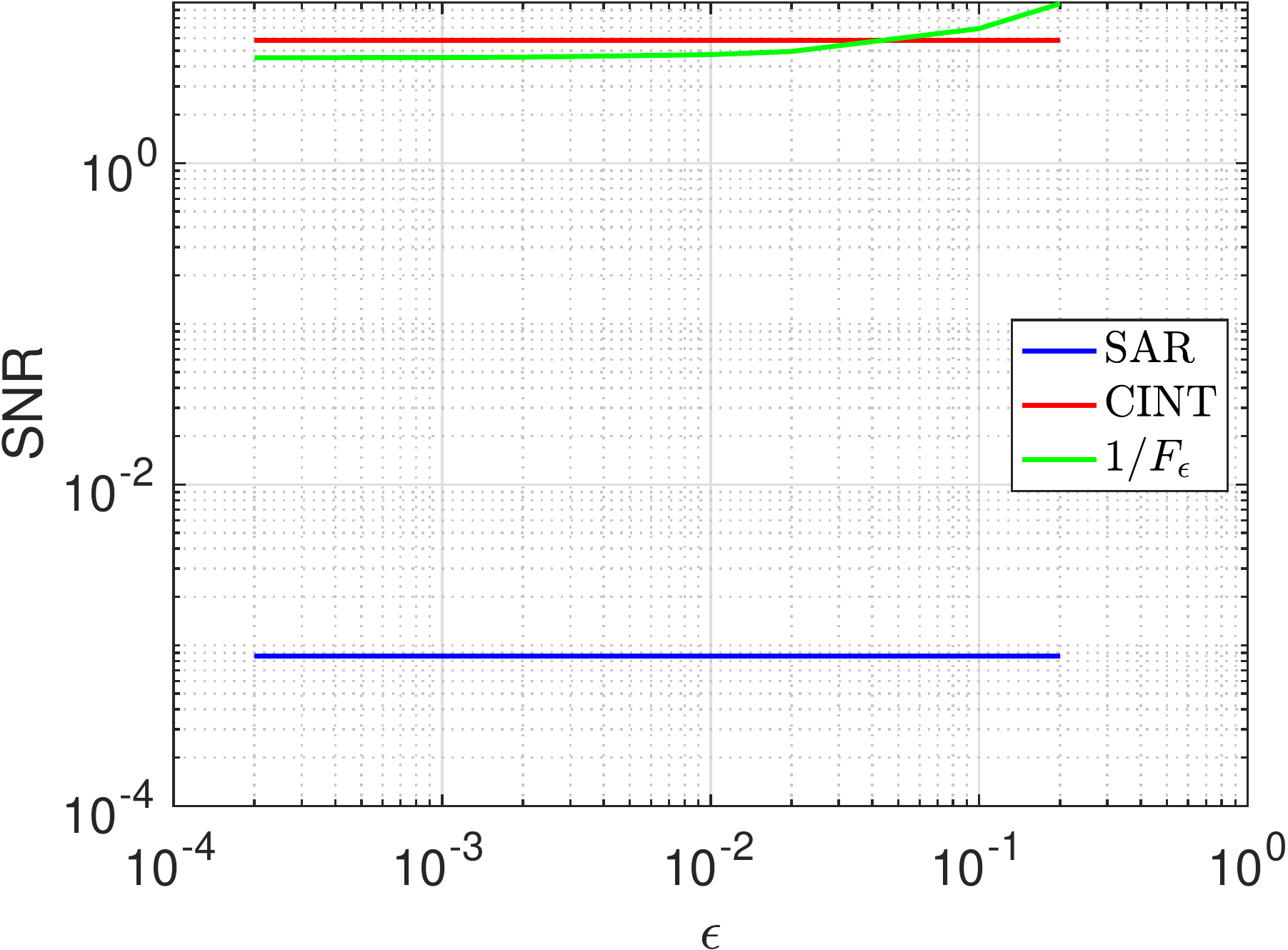}
    \caption{$\tilde{\sigma} = 0.4$}
  \end{subfigure}
  \begin{subfigure}[t]{0.4\linewidth}
    \includegraphics[width=\linewidth]{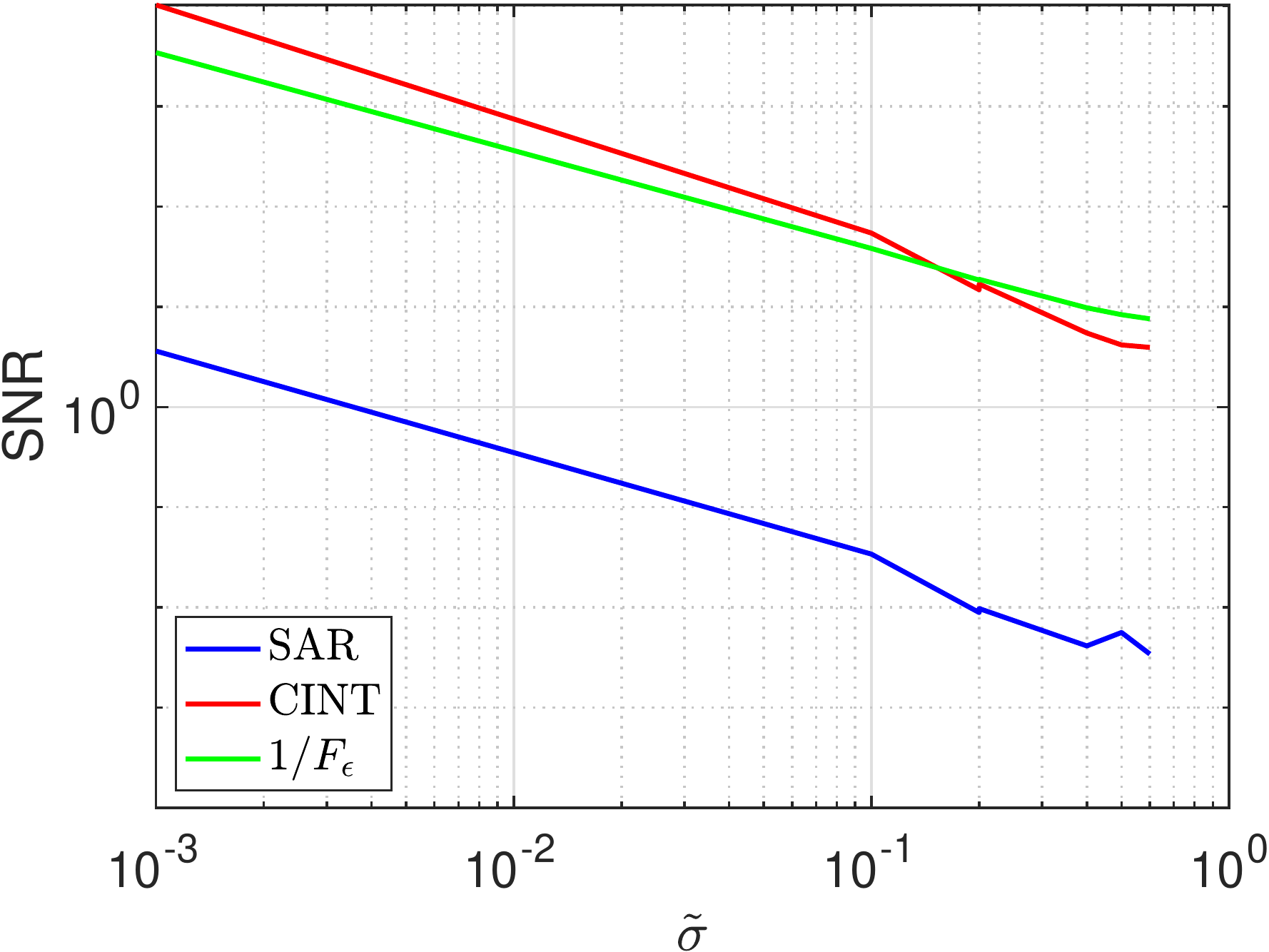}
    \caption{$\epsilon = 0.2$}
  \end{subfigure}
  \caption{Stability of the image as a function of the parameter
    $\epsilon$ with $\tilde{\sigma} = 0.4$ (a) and as a function of
    $\tilde{\sigma}$ with $\epsilon = 0.2$ (b). The CINT image is
    computed with $X_d=a/6$ and $\Omega_d=B/2$.}
  \label{fig:stab-rtt1}
\end{figure}

Following \cite{BGPT-11}, we compute the image's SNR defined as the
mean of the image divided by its standard deviation in a small area
around the true target location to estimate the stability of the
imaging method. The sample mean and the sample standard deviation are
computed using 100 realizations of the random medium with the same
characteristics (correlation length and strength of fluctuations).
For comparison, we also compute this SNR for the classical SAR image
and the CINT image. The CINT method requires specifying two key
parameters, the decoherence length $X_{d}$ and the decoherence frequency
$\Omega_{d}$. In the CINT results that follow, we have set
$X_{d} = a/6$ and $\Omega_{d} = B/2$. The results of these comparisons
are shown in Fig. \ref{fig:stab-rtt1} where we compare SNR as a
function of $\epsilon$ with $\tilde{\sigma} = 0.4$
(Fig.~\ref{fig:stab-rtt1}(a)) and as a function of $\tilde{\sigma}$
with $\epsilon = 0.2$ (Fig.~\ref{fig:stab-rtt1}(b)). Figure
\ref{fig:stab-rtt1}(a) and (b) illustrate the well-known result that
classical SAR imaging results are statistically unstable in random
media \cite{BGPT-11} since this SNR is very low. These results also
suggest similar stability for CINT and $1/F_\epsilon$, both with
comparably large SNRs. In Fig.~\ref{fig:stab-rtt1}(a), neither
classical SAR nor CINT depend on $\epsilon$, so we see no change in
behavior. However, as $\epsilon$ increases relative to
$\tilde{\sigma}$ such that $\tilde{\sigma}/\sqrt{\epsilon}$ becomes
small, we find that the SNR for $1/F_{\epsilon}$ becomes larger than
that for CINT. In Fig.~\ref{fig:stab-rtt1}(b), all images decrease in
SNR as $\tilde{\sigma}$ increases. However, the SNR for the CINT and
$1/F_\epsilon$ images is three orders of magnitude higher than the one
for classical SAR. It is important to note that the control offered by
$\epsilon$ is limited because one cannot set the value of $\epsilon$
to be larger than $1$. Otherwise, one cannot separate the signal from
the noise subspace.

\begin{figure}[t]
  \centering
  \begin{subfigure}[t]{0.4\linewidth}
    \includegraphics[width=\linewidth]{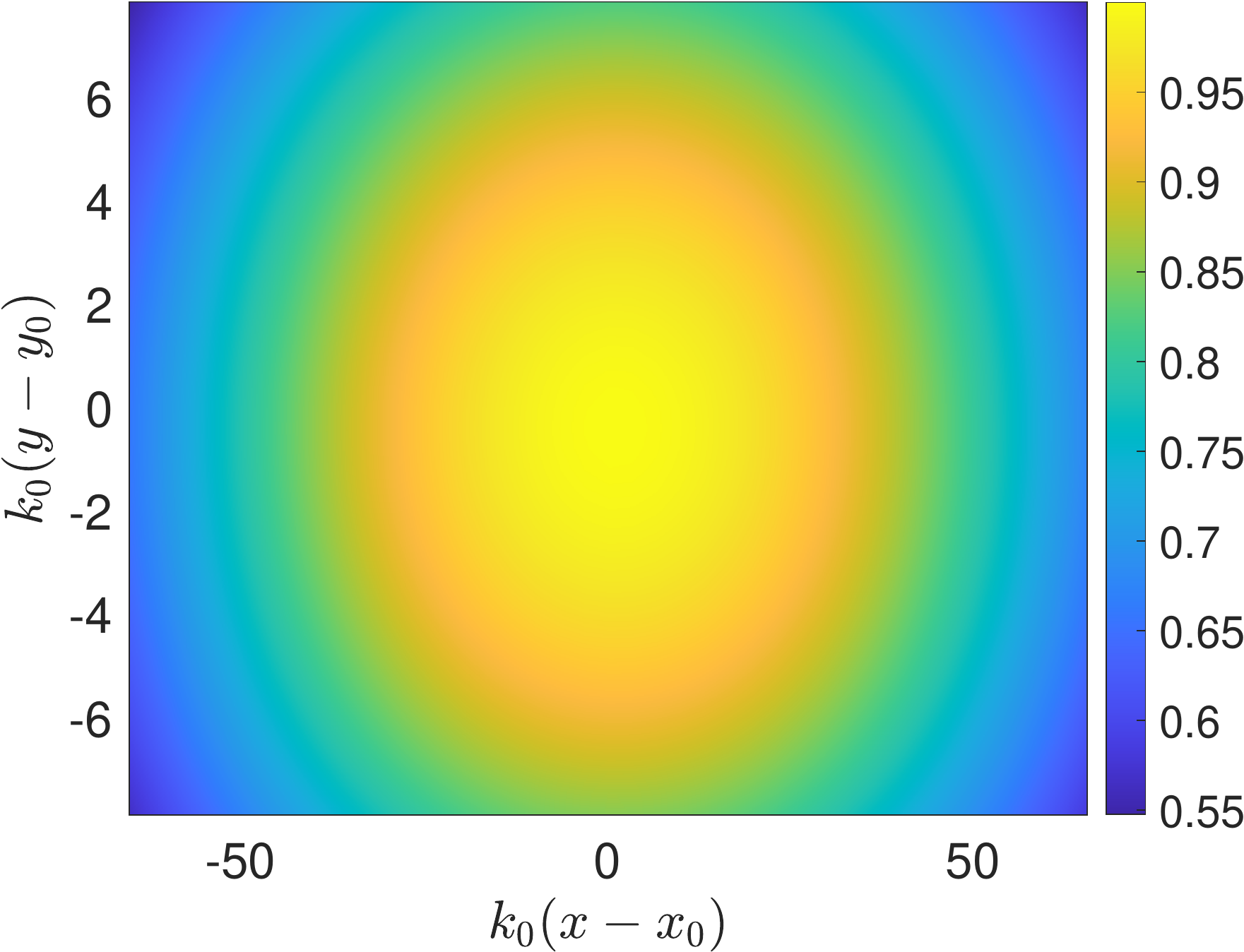}
    \caption{CINT}
  \end{subfigure}
  \begin{subfigure}[t]{0.4\linewidth}
    \includegraphics[width=\linewidth]{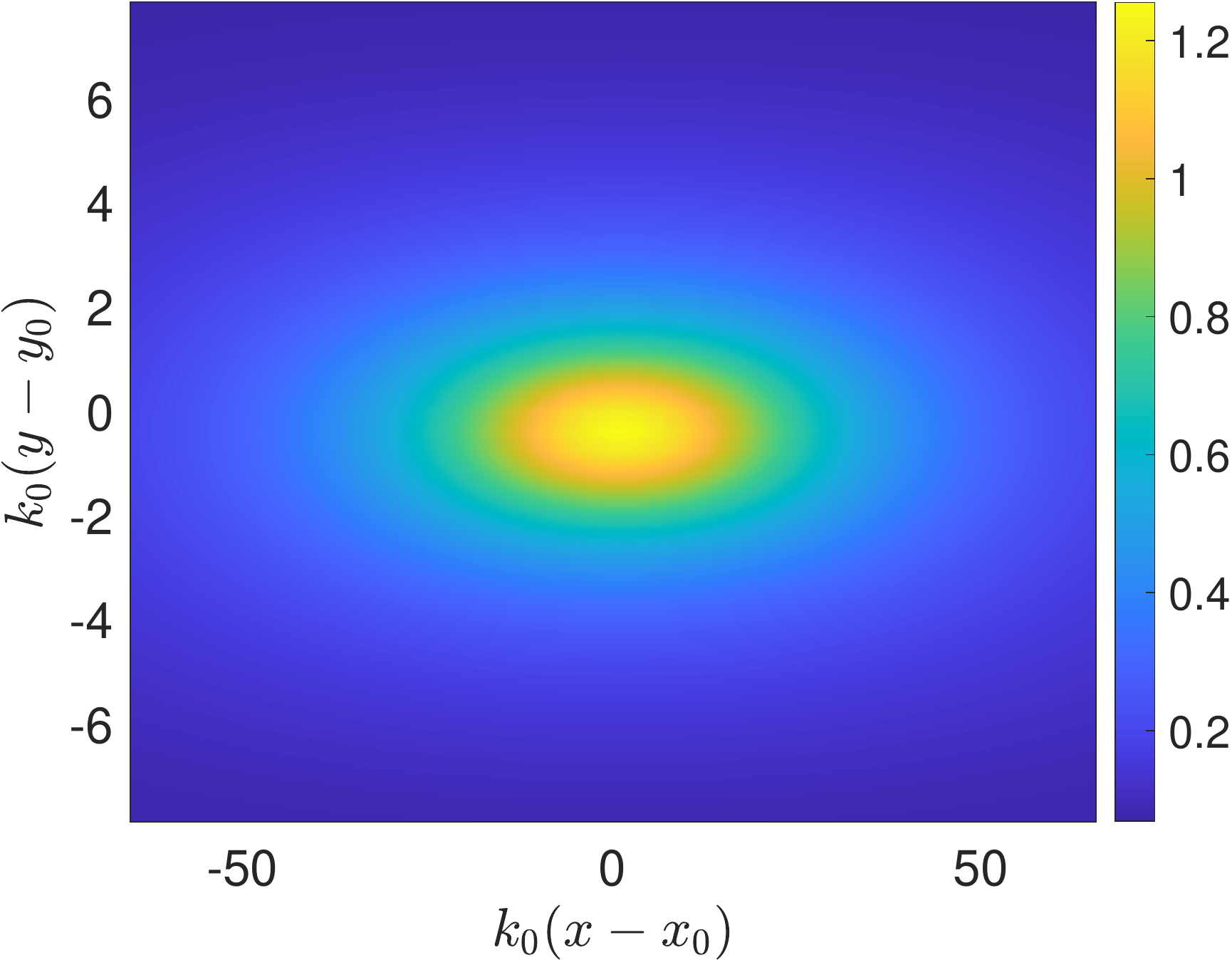}
    \caption{$1/F_{\epsilon}$}
  \end{subfigure}
  \caption{Comparison of images formed using CINT with $X_{d} = a/6$
    and $\Omega_{d} = B/2$ (a), and $1/F_{\epsilon}$ with $\epsilon$
    and $\tilde{\sigma}$ satisfying
    $\tilde{\sigma}/\sqrt{\epsilon} =\sqrt{\epsilon}$.}
  \label{fig:cint}  
\end{figure}

Although the images formed using CINT and $1/F_{\epsilon}$ have
similar stability behaviors, the image of $1/F_\epsilon$ has a much
better resolution. In Fig.~\ref{fig:cint} we compare images formed
using CINT and $1/F_{\epsilon}$ with $\epsilon$ and $\tilde{\sigma}$
set so that $\tilde{\sigma}/\sqrt{\epsilon} =\sqrt{\epsilon}$. These
results show that the image formed by $1/F_{\epsilon}$ is focused more
tightly on the target location in comparison to the image formed by
CINT. Additionally, there is quantitative information available from
the image formed by $1/F_{\epsilon}$. This tighter focus is especially
true for range although resolution is also better with cross-range.

To see the effect of this improved resolution, we compare images
formed using CINT and $1/F_{\epsilon}$ when the imaging region
contains four point targets situated closely to one another in
Fig.~\ref{fig:comp4a}. For all of these images,
$\tilde{\sigma} = 0.2$. Figures \ref{fig:comp4a}(b) and (c) are formed
using $1/F_{\epsilon}$ with $\epsilon = 0.01$ and $0.001$,
respectively. Here, the resolution of the CINT image does not allow
for identification of the four targets. The $1/F_{\epsilon}$ with
$\epsilon = 0.01$ image has a sharper resolution, but allows for
identification only three of the four targets.  In contrast, the
$1/F_{\epsilon}$ image with $\epsilon = 0.001$ shows four distinct
peaks indicating the target locations. These results show the
potential importance of being able to tune the resolution of an image
by varying the parameter $\epsilon$, even with random perturbations to
the travel times.

\section{Conclusions}
\label{sec:conclusion}

We have introduced and analyzed a quantitative signal subspace imaging
method for multi-frequency SAR measurements. The key to this method
involves a simple rearrangement of the frequency data at each spatial
location along the flight path where measurements are taken using the
Prony method. Using this rearranged frequency data, this method
involves two stages corresponding to two explicit imaging functionals,
\eqref{eq:F-function} and \eqref{eq:R-function}.

\begin{figure}[t]
  \centering
  \begin{subfigure}[t]{0.3\linewidth}
    \includegraphics[width=\linewidth]{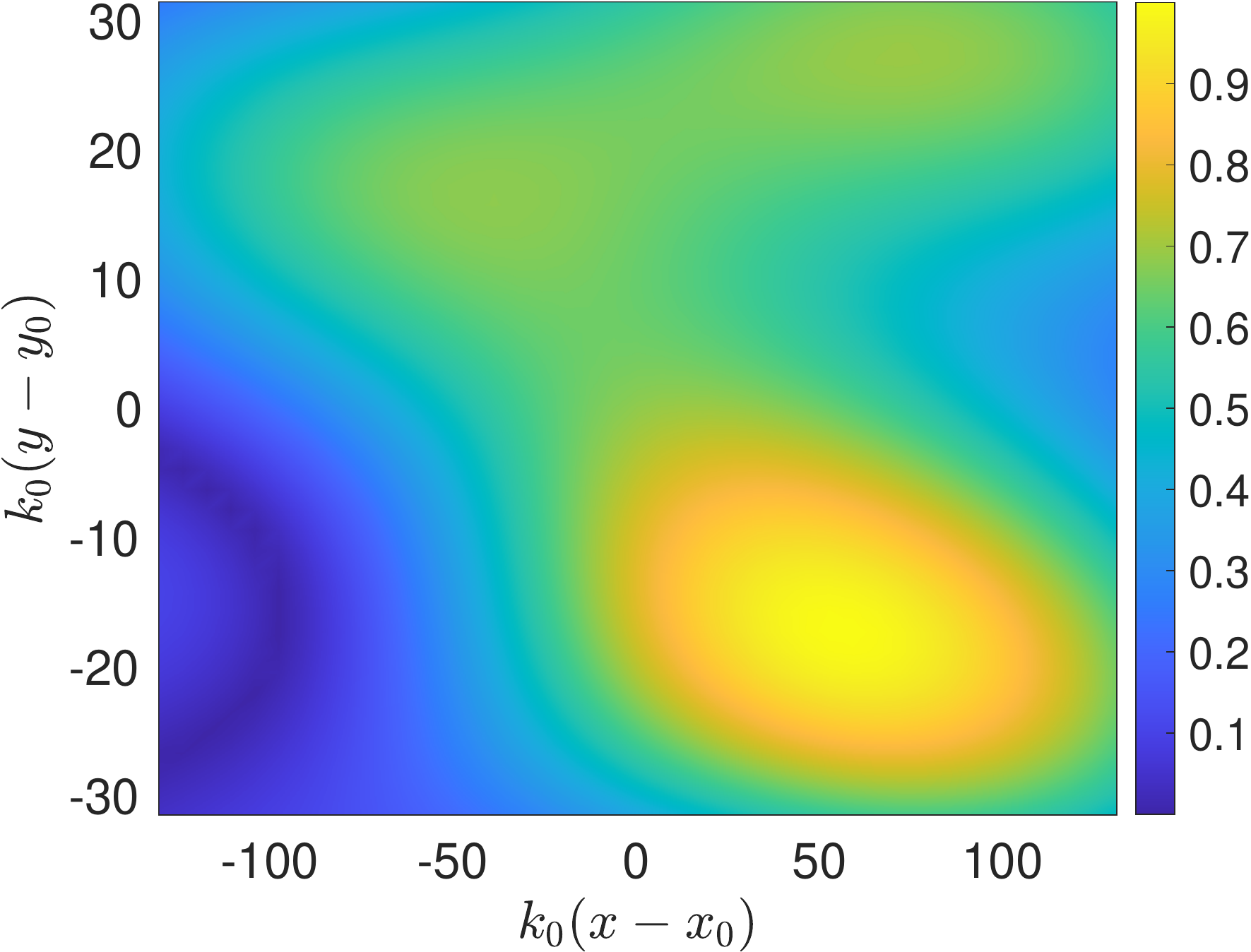}
    \caption{CINT}
  \end{subfigure}
  \begin{subfigure}[t]{0.3\linewidth}
    \includegraphics[width=\linewidth]{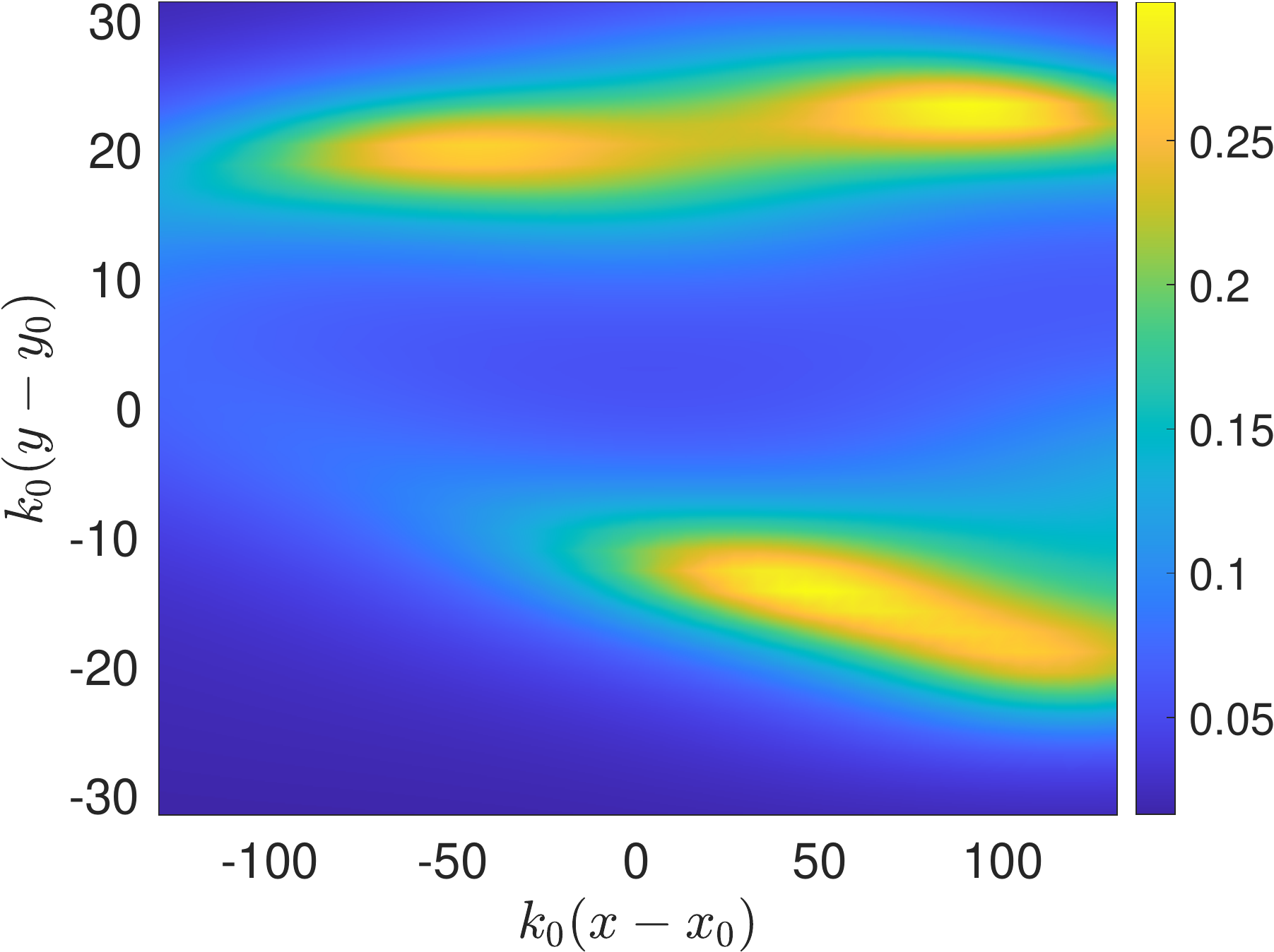}
    \caption{$1/F_\epsilon$ with $\epsilon=0.01$}
  \end{subfigure}
  \begin{subfigure}[t]{0.3\linewidth}
    \includegraphics[width=\linewidth]{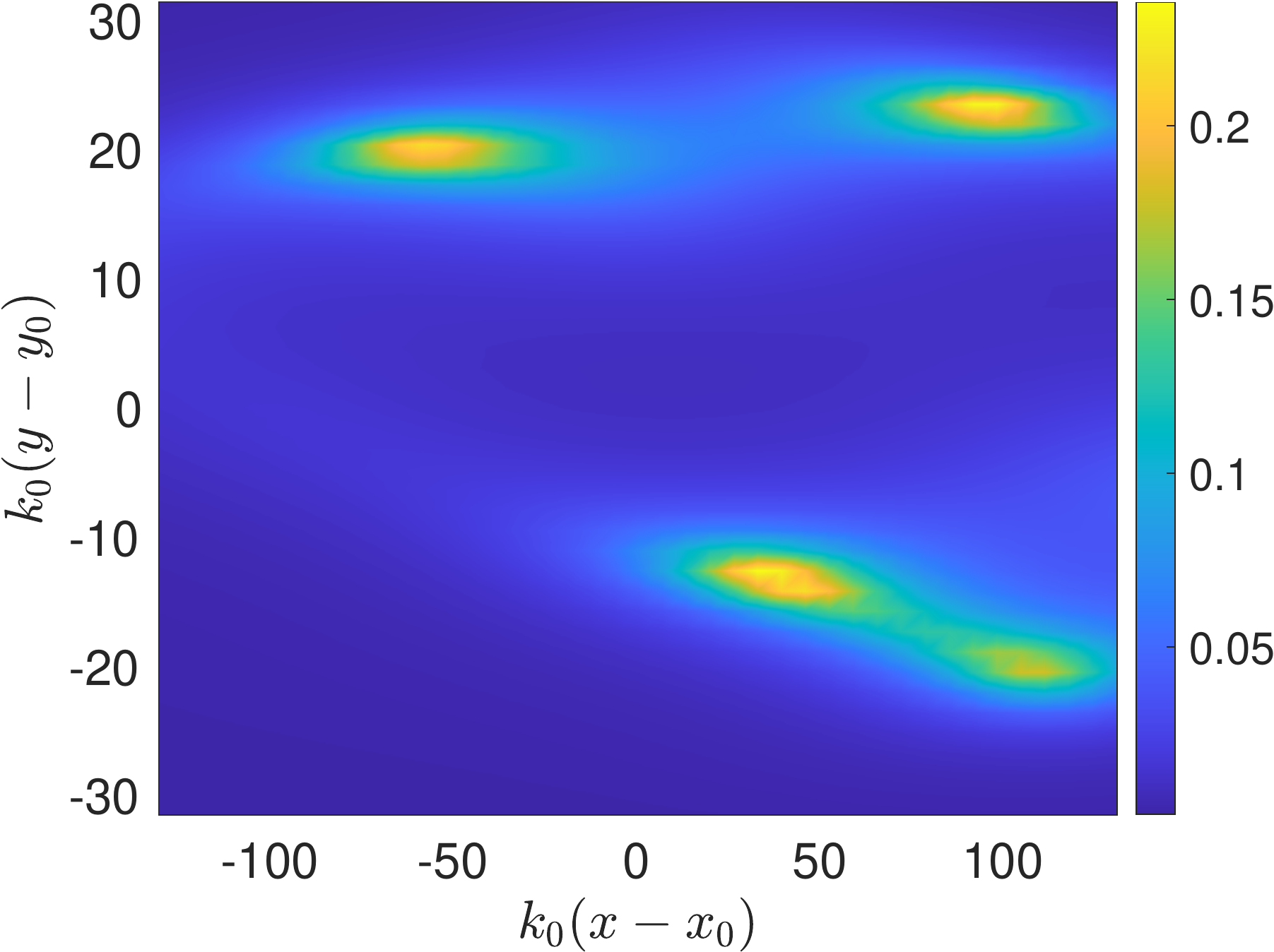}
    \caption{$1/F_\epsilon$ with $\epsilon=0.001$}
  \end{subfigure}
  \caption{Comparison of images formed using CINT (a) and
    $1/F_\epsilon$ with $\epsilon = 0.01$ (b) and $\epsilon = 0.001$
    (c) for four point targets. The CINT image is computed with
    $X_d=a/6$ and $\Omega_d=B/2$. The strength of the fluctuations is
    $\tilde{\sigma} = 0.2$.}
  \label{fig:comp4a}  
\end{figure}

Images produced through evaluation of $1/F_{\epsilon}(\ypos)$ over an
imaging region attain tunably high-resolution images of target
locations through a user-defined parameter $\epsilon$. Through a
resolution analysis for a linear flight path, we have determined that
the cross-range resolution of this imaging method is
$O(\sqrt{\epsilon} (c/B) (L/a) )$ where $c$ is the wave speed, $B$ is
the bandwidth, $L$ is the distance from the center of the flight path
to the center of the imaging region, and $a$ is the length of the
synthetic aperture. We have also determined that the resolution of
this imaging method in range is $O(\sqrt{\epsilon} (c/B) (L/R) )$
where $R$ is the range distance from the center of the flight path to
the center of the imaging region. With these resolution estimates, we
see how the user-defined parameter $\epsilon$ may be set to adjust the
image resolution for different settings. 

Images produced through evaluation of $1/R_{\epsilon}(\ypos)$ over an
imaging region do not reveal target locations. However, if the target
location is known, $1/R_{\epsilon}(\ypos)$ provides an accurate method
for recovering the complex reflectivity of a target. Again, the
user-defined parameter $\epsilon$ can be set to regularize the
function to enable stable recovery of the complex reflectivity. It is
for this reason that we have proposed a two-stage imaging method in
which $1/F_{\epsilon}(\ypos)$ is used to determine location of
target(s), and $1/R_{\epsilon}(\ypos)$ is evaluated at those locations
to recover the complex reflectivities. Additionally, the value of
$\epsilon$ used for evaluating $1/F_{\epsilon}$ need not be the same
used for evaluating $1/R_{\epsilon}$, so this parameter can be tuned
independently for these two different imaging functionals.

When there is uncertainty in the travel times, we have shown that
images formed by evaluating $1/F_{\epsilon}(\ypos)$ have an expected
value that is the same as the image formed in a homogeneous medium
provided that the variances of the random perturbations are
sufficiently small. Moreover, the variance of the image will be small
for that case indicating that this imaging method is statistically
stable to random perturbations in the travel times.

Both $F_{\epsilon}(\ypos)$ and $R_{\epsilon}(\ypos)$ are computed
using the SVD of the rearranged data. Consequently, their
effectiveness is understood to be related to how well the singular
values corresponding to signals scattered by the targets are separated
from noise. Provided that there is sufficient SNR for these singular
values to be separated, the parameter $\epsilon$ mitigates noise and
allows the user to control image resolution.  When there is
uncertainty in travel times, one can set the value of $\epsilon$ to
ensure image accuracy and statistical stability which, in turn, will
set the achievable image resolution.

Because this imaging method involves only elementary computations on
the data and allows for user-control to produce high-resolution,
quantitative images of targets, we believe that it is useful for a
broad variety of SAR imaging applications.

\section*{Acknowledgments}

The authors acknowledge support by the Air Force Office of Scientific
Research (FA9550-21-1-0196). A.~D.~Kim also acknowledges support by
the National Science Foundation (DMS-1840265).

% bibliography

\end{document}